\documentclass[a4paper,english]{smfart}

\usepackage[T1]{fontenc} 
\usepackage[utf8]{inputenc}
\usepackage{babel} 
\usepackage{newpxtext} %
\usepackage{eulervm} %

\usepackage{amssymb,textcomp,mathrsfs,braket,commath,relsize,stmaryrd,textcase}
	\expandafter\def\csname opt@stmaryrd.sty\endcsname  %
	{only,shortleftarrow,shortrightarrow}

\usepackage{bm} %
	\SetSymbolFont{stmry}{bold}{U}{stmry}{m}{n} %

\usepackage{etoolbox} 
\newcommand{\addQEDstyle}[2]{\AtBeginEnvironment{#1}{\pushQED{\qed}\renewcommand{\qedsymbol}{#2}}
	\AtEndEnvironment{#1}{\popQED}} %
	\addQEDstyle{rema}{$\diamondsuit$}
	\addQEDstyle{rema*}{$\diamondsuit$} %
	\addQEDstyle{exem}{$\diamondsuit$}
	\addQEDstyle{exem*}{$\diamondsuit$} %

\usepackage{tikz}
	\usetikzlibrary{cd} %

\usepackage{float,extpfeil,accents,enumitem,longtable}
\AtBeginDocument{\def\MR#1{}} %
\apptocmd{\sloppy}{\hbadness 10000\relax}{}{} %
\usepackage{mathtools}
    \mathtoolsset{showonlyrefs} %
\usepackage{smfthm} 
    \NumberTheoremsAs{subsubsection}\SwapTheoremNumbers

\usepackage{mymacros}

\usepackage[colorlinks,psdextra,pdfencoding=auto,backref=page]{hyperref}
	\hypersetup{colorlinks,linkcolor={red!50!black},filecolor={green!50!black},urlcolor={blue!80!black},citecolor={blue!50!black}}

\begin{document}
\frontmatter

\title[Wild genus-0 \MakeLowercase{q}-de Rham spaces]{Wild genus-zero quantum de Rham spaces}

\author[M.~Chaffe]{Matthew Chaffe\thanks{In the course of this work,
		M.~C. was funded by the project `Next Generation EU',
		under the National Recovery and Resilience Plan (NRRP),
		Mission 4,
		Component 2,
		Investment 1.1,
		Call PRIN 2022 No.~104 of Feb.~2,
		2022,
		of the Italian Ministry of University and Research;
		Project 2022S8SSW2 (subject area: PE -- Physical Sciences and Engineering),
		`Algebraic and geometric aspects of Lie theory'}}
\address[M.~Chaffe]{Department of Mathematics,
	University of Manchester,
	M13 9PL,
	UK}
\email{matthew.chaffe@manchester.ac.uk}

\author[G. Rembado]{Gabriele Rembado\thanks{G.~R. is supported by the European Commission,
		under the grant agreement n.~101108575 (HORIZON-MSCA project~\href{https://cordis.europa.eu/project/id/101108575}{QuantMod});
		more recently,
		also by the \emph{Ministero de Ciencia},
		under the \emph{Innovación y Universidades}' grant PID2024-155686NB-I00.}
}
\address[G.~Rembado]{Institut Montpelliérain Alexander Grothendieck,
	University of Montpellier,
	Place Eugène Bataillon,
	34090 Montpellier,
	France}
\email{gabriele.rembado@umontpellier.fr}

\author[D.~Yamakawa]{Daisuke Yamakawa\thanks{D.~Y. is supported by JSPS KAKENHI Grant Number 24K06695.}
}
\address[D.~Yamakawa]{Department of Mathematics,
	Faculty of Science Division I,
	Tokyo University of Science,
	1-3 Kagurazaka,
	Shinjuku-ku,
	Tokyo 162-8601,
	Japan}
\email{yamakawa@rs.tus.ac.jp}

\begin{abstract}
	The wild de Rham spaces parameterize isomorphism classes of (stable) meromorphic connections,
	defined on principal bundles over wild Riemann surfaces.
	Working on the Riemann sphere,
	we will deformation-quantize the standard open part of de Rham spaces,
	which corresponds to the moduli of linear ordinary differential equations with meromorphic coefficients.
	We treat the general untwisted/unramified case with nonresonant semisimple formal residue,
	for any polar divisor and reductive structure group.

	The main ingredients are:
	(i) constructing the quantum Hamiltonian reduction of a (tensor) product of quantized coadjoint orbits in dual truncated-current Lie algebras,
	involving the corresponding category-$\mc O$ Verma modules;
	and (ii) establishing sufficient conditions on the coadjoint orbits,
	so that generically all meromorphic connections are stable,
	and the (semiclassical) moment map for the gauge-group action is faithfully flat.
\end{abstract}

{\let\newpage\relax\maketitle} %

\setcounter{tocdepth}{1}  %
\tableofcontents

\mainmatter

\section{Introduction,
  main results,
  layout}
\label{sec:intro}

\subsection{Main aim}

This is a text on the quantization of complex symplectic varieties arising in meromorphic 2d gauge theory,
parameterizing isomorphism classes of irregular-singular connections defined on principal bundles over Riemann surfaces:
the wild \emph{de Rham spaces} $\mc M_{\dR}$.

\subsubsection{}

In brief,
the genus-zero examples contain open subspaces $\mc M^*_{\dR} \sse \mc M_{\dR}$,
whose points naively correspond to linear meromorphic ODEs for a holomorphic function $z \mt \psi(z) \in \mb C^m$,
up to a global base-change in the target.
Here we replace $\GL_m(\mb C)$ by any connected complex reductive group $G$,
and then build on~\cite{boalch_2001_symplectic_manifolds_and_isomonodromic_deformations,
	boalch_2007_quasi_hamiltonian_geometry_of_meromorphic_connections,
	calaque_naef_2015_a_trace_formula_for_the_quantization_coadjoint_orbits,
	yamakawa_2019_fundamental_two_forms_for_isomonodromic_deformations,chaffe_2023_category_o_for_takiff_lie_algebras,chaffe_topley_2023_category_o_for_truncated_current_lie_algebras,felder_rembado_2023_singular_modules_for_affine_lie_algebras_and_applications_to_irregular_wznw_conformal_blocks,calaque_felder_rembado_wentworth_2024_wild_orbits_and_generalised_singularity_modules_stratifications_and_quantisation} in order to:
\begin{enumerate}
	\item
	      recall that $\mc M^*_{\dR}$ is naturally identified with a Hamiltonian reduction of products of coadjoint orbits $\mc O'_a$ in dual \emph{truncated-current Lie algebras} (= TCLAs);\fn{
		      A.k.a.~\emph{Takiff} Lie algebras~\cite{takiff_1971_rings_of_invariant_polynomials_for_a_class_of_lie_algebras},
		      cf.~\cite{geoffriau_1994_sur_le_centre_de_l_algebre_enveloppante_d_une_algebre_de_takiff,geoffriau_1995_homomorphismo_de_harish_chandra_pour_les_algebres_de_takiff_generalisees}.}

	\item
	      construct a strong \emph{quantum} comoment map for the quantization of $\mc O'_a$;

	\item
	      study the flatness of the `semiclassical' moment map defining $\mc M^*_{\dR}$;

	\item
	      and construct a deformation quantization of $\mc M_{\dR}^*$ via \emph{quantum} Hamiltonian reductions of products of quantized coadjoint orbits.
\end{enumerate}

The rest of the introduction gathers more background/motivation (cf.~\S~\ref{sec:background}),
before moving on to a statement of the main results (cf.~\S~\ref{sec:results}),
and to a layout of the article (cf.~\S~\ref{sec:layout}).

\subsection{More background/motivation}
\label{sec:background}

Let $\Sigma$ be a closed Riemann surface,
and mark a finite set $\bm a \sse \Sigma$ of unordered points.
Let also $\mf g \ceqq \Lie(G)$ be the Lie algebra of $G$.

\subsubsection{}

A \emph{meromorphic connection on} $(\Sigma,\bm a)$ is a pair $(\bm\pi,\bm{\mc A})$,
consisting of:
(i) a holomorphic principal $G$-bundle $\bm\pi \cl \bm{\mc E} \to \Sigma$;
and (ii) a $\mf g$-valued meromorphic $1$-form $\bm{\mc A}$ on the total space $\bm{\mc E}$,
with poles along the divisor $\bm{\mc E}_{\bm a} \ceqq \bm\pi^{-1}(\bm a) \sse \bm{\mc E}$,
such that:
\begin{enumerate}
	\item
	      $R^*_g (\bm{\mc A}) = \Ad_{g^{-1}} (\bm{\mc A})$,
	      where $R_g$ is the structural right action of $g \in G$ on $\bm{\mc E}$;

	\item
	      and $\braket{ \bm{\mc A},X^\sharp } = X$,
	      where $X^\sharp$ is the fundamental vector field of $X \in \mf g$ on $\bm{\mc E}$.
\end{enumerate}

For many different reasons,
in complex Poisson/symplectic geometry,
representation theory,
and low-dimensional topology,
one is interested in parameterizing isomorphism classes of meromorphic connections.
Our main motivation involves the actions of generalized braid/mapping class groups,
both before and after quantization,
cf.~\S\S~\ref{sec:wild_riemann_surfaces}--\ref{sec:quantum_actions}.
(This relates with the mathematical formalization of conformal/topological field theories,
cf.~Rmk.~\ref{rmk:hitchin} and~\cite{witten_2007_gauge_theory_and_wild_ramification}.)

The corresponding moduli spaces can be deformed,
in \emph{isomonodromic} fashion,
and quantized.
But in the irregular-singular case one fixes more data than the base pointed surface $(\Sigma,\bm a)$:
we prescribe---%
$G$-orbits of---%
nonresonant `very good' normal forms at each pole,
cf.~\cite{boalch_2001_symplectic_manifolds_and_isomonodromic_deformations,yamakawa_2019_fundamental_two_forms_for_isomonodromic_deformations}.

\subsubsection{}

Namely,
consider a marked point $a \in \bm a$,
and a disc in $\Sigma$ centred at $a$,
with coordinate $z$.
Upon trivializing the bundle thereon,
pulling back the connection 1-form determines an element of $\mf g \set{\!\set{z}\!} \dif z$,
where $\mf g\set{\!\set{z}\!} \ceqq \mf g \ots_{\mb C} \bigl( \mb C \set{\!\set{z}\!} \bigr)$,
invoking the field of convergent Laurent series.

However,
we are only interested in the formal germ of $\bm{\mc A}$ at $a$.
Precisely,
consider the \emph{completed} local ring $\ms O_a$ of $\Sigma$ at $a$,
with maximal ideal $\mf m_a$.
Then $\mc D_a \ceqq \Spec \ms O_a$ is a formal neighbourhood of $a$,
and $z$ yields a \emph{uniformizer} $\varpi_a \in \mf m_a$.
With these choices,
pulling back $\bm{\mc A}$ even further determines an element
\begin{equation}
	\label{eq:formal_laurent_expansion}
	\wh{\mc A}_a
	\in \mf g (\!( \varpi_a )\!) \dif \varpi_a,
	\qquad \mf g (\!( \varpi_a )\!)
	\ceqq \mf g \ots_{\mb C} \bigl( \mb C (\!( \varpi_a )\!) \bigr).
\end{equation}
Then we act by the formal-holomorphic gauge group $G \llb \varpi_a \rrb$,
looking for a distinguished element in the corresponding gauge orbit $\wh{\mc O}'_a$ of $\wh{\mc A}_a$:
we suppose that $\wh{\mc O}'_a$ contains a (formal) \emph{normal form}
\begin{equation}
	\label{eq:normal_form_intro}
	\mc A'_a
	= \Lambda'_a \varpi_a^{-1}\dif \varpi_a + \dif Q'_a,
	\qquad Q'_a
	= \sum_{i = 1}^{s_a-1} A'_{a,i} \frac{\varpi_a^{-i}}{-i}.
\end{equation}
Here $s_a \geq 1$ is the pole order of~\eqref{eq:formal_laurent_expansion},
while the coefficients $\Lambda'_a,A'_{a,1},\dc,A'_{a,s_a-1} \in \mf g$ are semisimple and commute with each other,
and the (formal) residue $A'_{a,0} \ceqq \Lambda'_a$ is \emph{nonresonant}.
(Cf.~Def.~\ref{def:nuts_connections}.)

\begin{rema}
	Note that~\eqref{eq:normal_form_intro} is a meromorphic 1-form which coincides with its \emph{principal part},
	and that there are definitions which do \emph{not} use uniformizers (cf.~Rmk.~\ref{rmk:global_intrinsicality}).
	Moreover,
	the normal form prescribes the polar divisor of $\bm{\mc A}$,
	viz.,
	\begin{equation}
		\label{eq:polar_divisor}
		D
		= D \bigl( \bm a,\smash{\wh{\bm{\mc O}}}' \bigr)
		\ceqq \sum_{\bm a} s_a [a]. \qedhere
	\end{equation}
\end{rema}

\subsubsection{}

Given a multiset of such gauge orbits $\smash{\wh{\bm{\mc O}}}' = \set{ \wh{\mc O}'_a }_{\bm a}$,
the \emph{naive de Rham groupoid}
\begin{equation}
	\label{eq:de_rham_groupoid}
	\mc C_{\dR}
	= \mc C_{\dR} \bigl( \wh{\bm\Sigma};G \bigr),
	\qquad \wh{\bm\Sigma}
	\ceqq \bigl( \Sigma,\bm a,\smash{\wh{\bm{\mc O}}}' \bigr),
\end{equation}
is the category of:
\begin{description}
	\item[objects]
	      meromorphic connections $(\bm\pi,\bm{\mc A})$ on $(\Sigma,\bm a)$,
	      such that $\wh{\mc A}_a \in \wh{\mc O}'_a$ for each $a \in \bm a$,
	      in the notation of~\eqref{eq:formal_laurent_expansion};

	\item[(iso)morphisms $(\bm\pi,\bm{\mc A}) \to \bigl( \wt{\bm \pi},\wt{\bm{\mc A}} \bigr)$]
	      (iso)morphisms $\Phi \cl \bm{\mc E} \lxra{\simeq} \wt{\bm{\mc E}}$ of holomorphic principal $G$-bundles,
	      covering the identity of $\Sigma$,
	      such that $\Phi^*\bigl( \wt{\bm{\mc A}} \bigr) = \bm{\mc A}$.
\end{description}
In the nonresonant setting,
one may equivalently impose that the principal part of~\eqref{eq:formal_laurent_expansion} lies in the orbit of~\eqref{eq:normal_form_intro} for the action of a \emph{truncated} gauge group $G_{s_a}$ (cf.~\S~\ref{sec:classification_on_disc}).
The latter integrates the TCLA
\begin{equation}
	\label{eq:tcla_intro}
	\mf g_{s_a}
	= \mf g \ots_{\mb C} \bigl( \ms O_{s_a[a]} \bigr),
	\qquad \ms O_{s_a[a]}
	\ceqq \ms O_a \slash \mf m_a^{s_a},
\end{equation}
and the truncated gauge-orbit of the principal part can be identified with a coadjoint orbit $\mc O'_a \sse \mf g_{s_a}^{\dual}$:
its KKS structure~\cite{kirillov_1962_unitary_representations_of_nilpotent_lie_groups,kostant_1970_quantisation_and_unitary_representations_i_prequantization,souriau_1970_structure_des_systemes_dynamiques} is what we quantize later on.
Moreover,
this truncation leads to finite-dimensional descriptions of open moduli (sub)spaces,
in genus-zero,
as complex-affine $G$-Hamiltonian reductions.
(Cf.~\S~\ref{sec:fin_dim_description_intro};
the orbits $\mc O'_a$ are \emph{affine} varieties by Prop.~\ref{prop:affine_orbits}.)

\subsubsection{}
\label{sec:wild_riemann_surfaces}

On the subject of moduli spaces,
recall that one way to start constructing (analytic) moduli stacks is to work in families.
One considers deformations of the starting data of $\wh{\bm\Sigma}$,
parameterized by complex manifolds/analytic spaces $\bm B$,
and defines a groupoid for each such $\bm B$:
cf.,
e.g.,
\cite[App.~A]{doucot_rembado_tamiozzo_2024_moduli_spaces_of_untwisted_wild_riemann_surfaces} (and~\S~5 of op.~cit. for the algebraic setup,
working over nonsingular complex projective curves,
etc.).

In turn,
there are \emph{admissible deformations}~\cite[Def.~10.1 + Rmk.~10.6]{boalch_2014_geometry_and_braiding_of_stokes_data_fission_and_wild_character_varieties} of the wild Riemann surface underlying the decorated pointed surface $\wh{\bm\Sigma}$ of~\eqref{eq:de_rham_groupoid},
i.e.,
the triple
\begin{equation}
	\bm \Sigma \ceqq (\Sigma,\bm a,\bm\Theta'),
\end{equation}
where $\bm\Theta' \ceqq \set{\Theta'_a}_{\bm a}$ is the multiset of \emph{irregular classes} obtained from $\smash{\bm{\wh{\mc O}}}'$ (cf.~Def.~\ref{def:unframed_nuts_connections_2}).
The main motivational statement is that these admissible deformations provide an intrinsic/topological view on the nonlinear PDEs which govern the isomonodromic deformations of $(\bm\pi,\bm{\mc A})$.
This leads to algebraic Poisson/symplectic actions of `wild' generalizations of the mapping class group of pointed surfaces,
on \emph{wild} character varieties~\cite{boalch_2014_geometry_and_braiding_of_stokes_data_fission_and_wild_character_varieties,boalch_yamakawa_2015_twisted_wild_character_varieties},
generalizing the much-studied representations of surface groups.
(Please refer to~\cite{doucot_rembado_tamiozzo_2022_local_wild_mapping_class_groups_and_cabled_braids,doucot_rembado_2025_topology_of_irregular_isomonodromy_times_on_a_fixed_pointed_curve,boalch_doucot_rembado_2025_twisted_local_wild_mapping_class_groups_configuration_spaces_fission_trees_and_complex_braids,doucot_rembado_tamiozzo_2024_moduli_spaces_of_untwisted_wild_riemann_surfaces,doucot_rembado_yamakawa_twisted_g_local_wild_mapping_class_groups} for more details and references to the past work of many more people,
besides the authors and their collaborators.)

\subsubsection{}
\label{sec:quantum_actions}

Another important piece of motivation is that the isomonodromy equations can be locally expressed as the flow of an integrable nonautonomous Hamiltonian system:
the latter can also be quantized,
leading to `quantum' actions of braid and mapping class groups.
E.g.,
in the case of Fuchsian systems on the sphere,
the quantization of the Schlesinger system~\cite{schlesinger_1905_ueber_die_loesungen_gewisser_linearer_differentialgleichungen_als_funktionen_der_singularen_punkte,schlesinger_1912_ueber_eine_klasse_von_differentialsystemen_beliebiger_ordnung_mit_festen_kritischen_punkten} yields the Knizhnik--Zamolodchikov connection (= KZ~\cite{knizhnik_zamolodchikov_1984_current_algebra_and_wess_zumino_model_in_two_dimensions});
and adding one generic pole of order two at infinity leads to the connections of De Concini--Millson--Toledano Laredo/Felder--Markov--Tarasov--Varchenko (= DMT/FMTV~\cite{millson_toledanolaredo_2005_casimir_operators_and_monodromy_representations_of_generalised_braid_groups,felder_markov_tarasov_varchenko_2000_differential_equations_compatible_with_kz_equations}),
quantizing~\cite{jimbo_miwa_mori_sato_1980_density_matrix_of_an_impenetrable_bose_gas_and_the_fifth_painleve_transcendent}.
The monodromy of KZ is a representation of Artin's standard braid group,
featuring in the statement of the Drinfel'd--Kohno theorem~\cite{kohno_1987_monodromy_representation_of_braid_groups_and_yang_baxter_equations,drinfeld_1989_quasi_hopf_algebras,drinfeld_1989_quasi_hopf_algebras_and_knizhnik_zamolodchikov_equations};
while the monodromy of DMT/FMTV adds on an action of Brieskorn--Deligne's $G$-braid groups~\cite{brieskorn_1971_die_fundamentalgruppe_des_raumes_der_regulaeren_orbits_einer_endlichen_komplexen_spiegelungsgruppe,deligne_1972_les_immeubles_des_groupes_de_tresses_generalises},
and it recovers Lusztig's action~\cite{lusztig_1990_quantum_groups_at_roots_of_1} (cf.~\cite{soibelman_1990_algebra_of_functions_on_a_compact_quantum_group_and_its_representations,kirillov_reshetikhin_1990_q_weyl_group_and_a_multiplicative_formula_for_universal_r_matrices}),
in a Drinfel'd--Kohno theorem for $q$-Weyl groups~\cite{toledanolaredo_2002_a_kohno_drinfeld_theorem_for_quantum_weyl_groups} (cf.~\cite{boalch_2002_g_bundles_isomonodromy_and_quantum_weyl_groups}).
These `quantum' flat connections were generalized in~\cite{rembado_2019_simply_laced_quantum_connections_generalising_kz,yamakawa_2022_quantization_of_simply_laced_isomonodromy_systems_by_the_quantum_spectral_curve_method},
allowing for a nongeneric pole of order 3,
in a quantization of~\cite{boalch_2012_simply_laced_isomonodromy_systems}.

In this text,
instead,
we treat meromorphic connections with \emph{arbitrary} polar divisor,
without relying on quiver-modularity results~\cite{boalch_2012_simply_laced_isomonodromy_systems,hiroe_yamakawa_2014_moduli_spaces_of_meromorphic_connections_and_quiver_varieties}.
Moreover,
contrary to~\cite{gaiur_mazzocco_rubtsov_2023_isomonodromic_deformations_confluence_reduction_and_quantisation},
fixing the formal gauge orbits $\smash{\wh{\bm{\mc O}}}'$ selects a symplectic leaf within a Poisson moduli space,
and so we are naturally led to quotients of the universal enveloping algebra of~\eqref{eq:tcla_intro} (cf.~Rmk.~\ref{rmk:vermas_are_quotients}).
This is required,
in order to relate the family of quantizations with bundles of conformal blocks for the category $\mc O$ of~\cite{chaffe_2023_category_o_for_takiff_lie_algebras,chaffe_topley_2023_category_o_for_truncated_current_lie_algebras} (cf.~\cite{felder_rembado_2023_singular_modules_for_affine_lie_algebras_and_applications_to_irregular_wznw_conformal_blocks,calaque_felder_rembado_wentworth_2024_wild_orbits_and_generalised_singularity_modules_stratifications_and_quantisation}).
Nonetheless,
we also use the `confluence' of simple poles to establish the most difficult criteria for flatness (cf.~\cite{nagoya_sun_2010_confluent_primary_fields_in_the_conformal_theory, nagoya_sun_2011_confluent_kz_equations_for_sl_n_with_poincare_rank_2_at_infinity}),
noting however that our constructions rather go in the opposite direction:
we \emph{unfold}~\cite{hiroe_2024_deformation_of_moduli_spaces_of_meromorphic_connections_on_p_1_via_unfolding_of_irregular_singularities} one irregular singularity into several regular ones,
embedding wild de Rham spaces into `residue manifolds'~\cite{hitchin_1997_frobenius_manifolds} (cf.~Cor.~\ref{cor:unfolding_in_action}).
Moreover,
we construct Darboux coordinates on untwisted semisimple coadjoint orbits (cf.~Thm.~\ref{thm:darboux_coordinates}).

\begin{rema}
	\label{rmk:hitchin}

	There is quite some more literature around isomonodromy Hamiltonians,
	also in higher genus.
	In turn,
	the latter relates with (semiclassical limits of) the connections of Bernard/Tsuchiya--Ueno--Yamada (= KZB/TUY~\cite{bernard_1988_on_the_wess_zumino_witten_models_on_the_torus,bernard_1989_on_the_wess_zumino_witten_models_on_riemann_surfaces,tsuchiya_ueno_yamada_1989_conformal_field_theory_on_universal_family_of_stable_curves_with_gauge_symmetries}).

	As a last bibliographical point,
	we point out the `geometrized' versions of KZ(B)/TUY~\cite{egsgaard_2015_hitchin_connection_for_genus_0_quantum_representation,biswas_mukhopadhyay_wentworth_2023_a_hitchin_connection_on_nonabelian_theta_functions_for_parabolic_g_bundles},
	related with `parabolic' extensions of the Hitchin connection~\cite{hitchin_1990_flat_connections_and_geometric_quantization,axelrod_dellapietra_witten_1991_geometric_quantisation_of_chern_simons_gauge_theory} in geometric quantization.
	They are constructed by algebro-geometric means (cf.~\cite{baier_bolognesi_martens_pauly_2023_the_hitchin_connection_in_arbitrary_characteristic}),
	extending the famous relation~\cite{beauville_1993_monodromie_des_systems_differentiels_lineaires_a_poles_simples_sur_la_sphere_d_apres_a_bolibruch,faltings_1994_a_proof_for_the_verlinde_formula,laszlo_1998_hitchin_s_and_wzw_connections_are_the_same} between nonabelian theta functions and WZNW conformal blocks in 2$d$ conformal field theory~\cite{belavin_polyakov_zamolodchikov_1984_infinite_conformal_symmetry_in_two_dimensional_quantum_field_theory}.
	Much of this generalizes into the wild case.
\end{rema}

\subsubsection{}

We plan to further pursue the above perspective on the `time-dependent' quantization of the isomonodromy systems,
and the relation with WZNW conformal blocks---%
elsewhere.
In this text we rather work on a \emph{fixed} wild Riemann surface,
in order to prepare the construction of the `quantum' phase-spaces.
Thus,
in first approximation,
one can consider the set of isomorphism classes in the de Rham groupoid,
denoted by $\mc M_{\dR} = \mc M_{\dR} \bigl( \wh{\bm\Sigma};G \bigr)$.

Then there is also a different gauge-theoretic construction of the moduli spaces,
which rephrases the moduli problem as a group action,
matching up isomorphism classes with orbits in a large parameter space:
this typically endows $\mc M_{\dR}$ with much more geometric structure,
viewing it as the Hamiltonian reduction of an infinite-dimensional complex affine space.
E.g.,
when on vector bundles,
the main construction of~\cite{biquard_boalch_2004_wild_nonabelian_hodge_theory_on_curves} (cf.~\cite{sabbah_1999_harmonic_metrics_and_connections_with_irregular_singularities}) yields a \emph{complete} hyperkähler manifold.
The underlying holomorphic symplectic structure is a generalization of the Narasimhan/Atiyah--Bott/Goldmann form~\cite{narasimhan_1970_geometry_of_moduli_spaces_of_vector_bundles,atiyah_bott_1983_yang_mills_equations_over_riemann_surfaces,goldman_1984_the_symplectic_nature_of_fundamental_groups_of_surfaces},
first introduced in~\cite{boalch_2001_symplectic_manifolds_and_isomonodromic_deformations} (in the generic case),
which also extends to more general reductive structure groups~\cite{boalch_2007_quasi_hamiltonian_geometry_of_meromorphic_connections}.

\subsubsection{}
\label{sec:fin_dim_description_intro}

And indeed,
as in~\cite{boalch_2007_quasi_hamiltonian_geometry_of_meromorphic_connections,yamakawa_2019_fundamental_two_forms_for_isomonodromic_deformations},
we do \emph{not} restrict to vector bundles.
Rather,
we consider the classical setup for the theory of isomonodromic deformations.
Precisely,
we consider meromorphic connections on the \emph{trivial} holomorphic principal $G$-bundle over $\Sigma \ceqq \mb CP^1$.
This yields a subgroupoid $\mc C_{\dR}^* \sse \mc C_{\dR}$,
whose objects map bijectively\fn{
	Cf.~\eqref{eq:fin_dim_description_de_rham},
	and note that nonresonance is required to ensure surjectivity;
	otherwise,
	cf.~Rmk.~\ref{rmk:resonant_de_rham}.}~to the $\mb C$-points of a \emph{finite-dimensional} irreducible complex affine variety $M$.
The latter is a subvariety of the product of truncated gauge orbits $\bm{\mc O}' \ceqq \set{ \mc O'_a }_{\bm a}$,
i.e.,
the multiset of coadjoint orbits through the normal forms $\mc A'_a \in \mf g_{s_a}^{\dual}$ (cf.~\eqref{eq:normal_form_intro});
and two connections are isomorphic if and only if the corresponding points in $M$ are related by the diagonal $G$-action.
(Cf.~\cite{boalch_2001_symplectic_manifolds_and_isomonodromic_deformations,boalch_2007_quasi_hamiltonian_geometry_of_meromorphic_connections,yamakawa_2019_fundamental_two_forms_for_isomonodromic_deformations};
this is recalled in \S~\ref{sec:unframed_de_rham}.)
To get the corresponding moduli (sub)space $\mc M_{\dR}^* \sse \mc M_{\dR}$ we then choose sufficiently generic orbits $\bm{\mc O}'$,
so that all connections are \emph{naively} stable (cf.~\cite[Rmk.~2.37]{boalch_1999_symplectic_geometry_and_isomonodromic_deformations} and~\cite[Rmk.~9.7]{boalch_2012_simply_laced_isomonodromy_systems} in the vector-bundle case).
I.e.,
we give a sufficient condition so that \emph{all} points of $M$ are stable for the action of the projectivized group $\bm P(G) \ceqq G \slash Z(G)$ (cf.~\S~\ref{sec:stability}).

\begin{rema}
	\label{rmk:using_marking}

	The normal form $\mc A'_a$ at the marked point $a \in \bm a$ is determined upon framing the bundle there.
	Else,
	it is only the $G$-\emph{orbit} of $\mc A'_a$ which is intrinsic (cf.~Def.~\ref{def:unframed_nuts_connections_2}.)
	With this caveat,
	all our `semiclassical' constructions remain canonical,
	while the `quantum' ones rely on the choice of a polarization of $\mc O'_a$ (as it typically happens,
	cf.~\S~\ref{sec:using_marking_2}).
\end{rema}

\begin{rema}
	The stable locus $M^{\on s} \sse M$ might be empty:
	the (irregular) \emph{additive}~\cite{kostov_1999_on_the_deligne_simpson_problem,kostov_2004_on_the_deligne_simpson_problem_and_its_weak_version} Deligne--Simpson problem~\cite{simpson_1991_products_of_matrices} is to characterize the choices of orbits $\bm{\mc O}'$ such that there exist stable connections with such local data.
	E.g.,
	in the logarithmic vector-bundle case,
	the necessary/sufficient condition~\cite{crawleyboevey_2003_on_matrices_in_prescribed_conjugacy_classes_with_no_common_invariant_subspace_and_sum_zero} relies on the quiver-theoretic description of the moduli spaces~\cite{crawleyboevey_2001_geometry_of_the_moment_map_for_representations_of_quivers};
	this was extended to the irregular-singular case in~\cite{boalch_2008_irregular_connections_and_kac_moody_root_systems} (cf.~\cite[\S~10]{boalch_2012_simply_laced_isomonodromy_systems}).
	In our situation,
	the level set $M$ is \emph{nonempty} if the $G$-moment map~\eqref{eq:moment_map} is faithfully flat,
	as per Thm.~\ref{thm:main_result};
	the condition~\eqref{eq:stability_condition} then insures that $M^{\on s} = M \neq \vn$.
\end{rema}

\begin{rema}
	\label{rmk:nonabelian_hodge}

	A finer notion of stability involves parabolic degrees/slopes,
	cf.~again~\cite{biquard_boalch_2004_wild_nonabelian_hodge_theory_on_curves};
	and in principle \emph{parahoric} structures are required for a general group $G$,
	even in the tame case~\cite{boalch_2011_riemann_hilbert_for_tame_complex_parahoric_connections}.
	We will \emph{not} discuss this in this text,
	particularly because we focus on the de Rham side of wild nonabelian Hodge theory---%
	and so we are not concerned with the rotation of the weights,
	etc.
\end{rema}

\subsection{Main results}
\label{sec:results}

In the end,
the affine GIT quotient of $M = M^{\on s}$ by $G$ yields a genuine orbit space,
which is our model for $\mc M^*_{\dR}$.

\subsubsection{}

Recall that~\cite{calaque_felder_rembado_wentworth_2024_wild_orbits_and_generalised_singularity_modules_stratifications_and_quantisation} set up a representation-theoretic quantization of the coadjoint orbits $\mc O'_a \sse \mf g_{s_a}^{\dual}$ through the principal parts $\mc A'_a$ of~\eqref{eq:normal_form_intro},
based on~\cite{alekseev_lachowska_2005_invariant_star_product_on_coadjoint_orbits_and_the_shapovalov_pairing}.
It relies on certain induced modules $M^\pm_a = M^\pm_{\mc A'_a}$ for TCLAs to be generically simple (cf.~Def.~\ref{def:vermas}).
These are the Verma modules of the category $\mc O$ of~\cite{chaffe_2023_category_o_for_takiff_lie_algebras,chaffe_topley_2023_category_o_for_truncated_current_lie_algebras},
generalizing Verma modules for reductive Lie algebras---%
referred to as `singularity modules' in~\cite{calaque_felder_rembado_wentworth_2024_wild_orbits_and_generalised_singularity_modules_stratifications_and_quantisation};
cf.~\cite{wilson_2011_highest_weight_theory_for_truncated_current_lie_algebras,felder_rembado_2023_singular_modules_for_affine_lie_algebras_and_applications_to_irregular_wznw_conformal_blocks} in the regular/generic case,
and Rmk.~\ref{rmk:vermas_are_quotients} in general.

Concretely,
let $\mc R_{a,0} \ceqq \mb C[\mc O'_a]$ be the commutative Poisson ring of regular functions on the coadjoint orbit $\mc O'_a$ through an untwisted semisimple principal part $\mc A'_a$.
In this text we first complement the deformation quantization of $\mc R_{a,0}$,
proving the following:

\begin{theo}[cf.~Prop.~\ref{prop:quantum_nonresonance} +~Thm.~\ref{thm:strong_quantum_comoment}]
	\label{thm:nuts_quantization}

	Choose two parabolic sequences in $G$ as in~\eqref{eq:parabolic_sequence},
	such that the associated polarizations of $\on T \! \mc O'_a$ are `balanced' (cf.~Rmk.~\ref{rmk:balanced_polarizations}).
	Then:
	\begin{enumerate}
		\item
		      the $\ast$-product $\mc R_{a,0} \ots_{\mb C} \mc R_{a,0} \to \mc R_{a,0} \llb \hs \rrb$ of~\cite{calaque_felder_rembado_wentworth_2024_wild_orbits_and_generalised_singularity_modules_stratifications_and_quantisation} admits a \emph{strong} quantum comoment map,
		      generating the natural `quantum' $G_{s_a}$-action;

		\item
		      and if $\mc A'_a$ is nonresonant,
		      then the $U(\mf g_{s_a})$-modules $M^\pm_a$,
		      are \emph{simple}.
	\end{enumerate}
\end{theo}

\subsubsection{}
\label{sec:main_theorem}

Second,
we construct the quantum Hamiltonian reduction of the corresponding product of `quantum' orbits $\set{ \wh{\mc R}_{a,\hs} }_{\bm a}$,
modulo the diagonal $G$-action,
in the standard way~\cite{etingof_2007_calogero_moser_systems_and_representation_theory} (cf.~\cite{rembado_2020_symmetries_of_the_simply_laced_quantum_connections_and_quantisation_of_quiver_varieties}).
This yields in principle a quantization of wild de Rham spaces:
the caveat is proving that the corresponding deformation of the coordinate ring is \emph{flat}.
A uniform way to ensure it lies in the geometry of the $G$-moment map for the `semiclassical' action on the orbit product (cf.~Lem.~\ref{lem:quantum_flatness}).

At this final stage,
it is cleaner to assume that $G$ is \emph{semisimple} (w.l.o.g.,
cf.~Lem.-Def.~\ref{lem:semisimple_reduction}).
Then we give an explicit `global' criterion for the faithful flatness of the moment map.
Namely,
for each marked point $a \in \bm a$ denote by $\nu_a \geq 0$ the number of nested $\Ad_G$-stabilizers of the coefficients of~\eqref{eq:normal_form_intro},
starting from the top and including the residue,
which are equal to a maximal torus of $G$.
This involves a `fission' sequence of reductive subgroups of $G$~\cite{boalch_2009_through_the_analytic_halo_fission_via_irregular_singularities,boalch_2014_geometry_and_braiding_of_stokes_data_fission_and_wild_character_varieties} (cf.~Rmk.~\ref{rmk:fission}).
Importantly,
the multiset $\set{ \nu_a }_{\bm a}$ is \emph{intrinsic} (cf.~Lem.~\ref{lem:coordinate_invariance}).
Moreover,
it only depends on the $G$-orbit of~\eqref{eq:normal_form_intro},
whence it is uniquely determined by the choice of wild Riemann surface,
together with a residue orbit at each marked point.
Then one can prove the following (main):

\begin{theo}[cf.~\S~\ref{sec:flatness_general}]
	\label{thm:main_result}

	Suppose that $G$ is semisimple,
	and that
	\begin{equation}
		\label{eq:wild_euler_characteristic}
		\chi \bigl( \wh{\bm\Sigma} \bigr) < 0,
		\qquad \text{where}
		\qquad \chi \bigl( \wh{\bm\Sigma} \bigr)
		\ceqq 2 - \sum_{\bm a} \nu_a \in \mb Z.
	\end{equation}
	Then:
	\begin{enumerate}
		\item
		      the semiclassical `diagonal' $G$-moment map on the orbit product is faithfully \emph{flat};

		\item
		      and the level-zero quantum $G$-Hamiltonian reduction of the tensor product of quantum orbits is a \emph{flat} deformation quantization of $\mc M^*_{\dR}$.
	\end{enumerate}
\end{theo}

\subsubsection{}

Recall that in the \emph{generic} case the leading coefficient at the marked point $a \in \bm a$ is regular semisimple.
Then one readily extracts the following important:

\begin{coro}
	\label{cor:main_corollary}

	Suppose that $G$ is semisimple,
	that all poles are generic,
	and that
	\begin{equation}
		2 - \abs D < 0,
		\qquad \text{where}
		\qquad \abs D
		\ceqq \sum_{\bm a} s_a,
	\end{equation}
	in the notation of~\eqref{eq:polar_divisor}.
	Then the conclusions of Thm.~\ref{thm:main_result} hold true.
\end{coro}

\subsubsection{}

Thus,
generically,
the criterion for flatness only depends on the pole orders.

More importantly,
Cor.~\ref{cor:main_corollary} includes the $G$-bundle generalization of the examples considered in the seminal work of Jimbo--Miwa--Ueno (= JMU~\cite{jimbo_miwa_ueno_1981_monodromy_preserving_deformation_of_linear_ordinary_differential_equations_with_rational_coefficients_i_general_theory_and_tau_function,jimbo_miwa_1981_monodromy_preserving_deformaton_of_linear_ordinary_differential_equations_with_rational_coefficients_ii,jimbo_miwa_1982_monodromy_preserving_deformaton_of_linear_ordinary_differential_equations_with_rational_coefficients_iii}).
E.g.,
under this genericity assumption one thus quantizes (Hamiltonian reductions of) the symplectic phase-spaces for the standard Lax pairs/representations of the Schlesinger system~\cite{schlesinger_1905_ueber_die_loesungen_gewisser_linearer_differentialgleichungen_als_funktionen_der_singularen_punkte,schlesinger_1912_ueber_eine_klasse_von_differentialsystemen_beliebiger_ordnung_mit_festen_kritischen_punkten},
the `dual' Schlesinger system~\cite{harnad_1994_dual_isomonodromic_deformations_and_moment_maps_to_loop_algebras},
and the system of Jimbo--Miwa--Môri--Sato (= JMMS~\cite{jimbo_miwa_mori_sato_1980_density_matrix_of_an_impenetrable_bose_gas_and_the_fifth_painleve_transcendent});
and all of the untwisted/unramified coalescence hierarchy of the Painlevé equations,
from PVI down to PII.
But importantly Thm.~\ref{thm:main_result} also includes nongeneric cases,
such as the `simply-laced' isomonodromy systems~\cite{boalch_2012_simply_laced_isomonodromy_systems},
which encompass all the previous examples besides PIII and PII.
In particular,
the `higher' Painlevé equations fit this setup; cf.~\S~11.4 of op.~cit,
as well as~\cite{bertola_cafasso_rubtsov_2018_noncommutative_painleve_equations_and_systems_of_calogero_type}.

The main point,
however,
is that the main result covers a plethora of unlisted cases with arbitrary polar divisor and nongeneric irregular singularities~\cite{yamakawa_2019_fundamental_two_forms_for_isomonodromic_deformations}.

\begin{rema}
	A \emph{necessary} condition for flatness will depend on the exact sequences of Levi factors of parabolic subgroups of $G$ at each pole,
	and we do \emph{not} look for one.\fn{
		If $G = \PGL_m(\mb C)$ for some integer $m \geq 1$,
		then the `Birkhoff' orbits~\eqref{eq:birkhoff_orbit} are isomorphic to cotangent bundles of quiver representation spaces,
		and the explicit flatness conditions are in~\cite[Thm.~1.1]{crawleyboevey_2001_geometry_of_the_moment_map_for_representations_of_quivers};
		e.g.,
		there are flat examples involving five nongeneric simple poles with $m = 3$.}~Nonetheless,
	the `wild' Euler--Poincaré characteristic of~\eqref{eq:wild_euler_characteristic} is intimately related with the Deligne--Mumford condition for the stacks $\mc{WM}_{0,n,\bullet} \thra \mc M_{0,n}$ of wild Riemann spheres~\cite{doucot_rembado_tamiozzo_2024_moduli_spaces_of_untwisted_wild_riemann_surfaces},
	where again $n \ceqq \abs{\bm a}$.
	More precisely,
	in the tame/logarithmic case,
	flatness generically holds if $2 - n < 0$,
	which is precisely the condition for $\mc M_{0,n}$ to be Deligne--Mumford;
	in the wild setting,
	instead,
	there are also flat examples with just $n \in \set{1,2}$ marked points,
	and in turn $\mc{WM}_{0,n,\bullet}$ can be Deligne--Mumford if the pole orders there are high enough,
	coherently with Thm.~\ref{thm:main_result} +~Cor.~\ref{cor:main_corollary}.
\end{rema}

\subsection{Layout}
\label{sec:layout}

The layout of this article is as follows.

\subsubsection{}

In \S~\ref{sec:setup_on_disc} we set up the basic terminology for (formal) connections on (formal) discs,
and the (formal) gauge action thereon.
In \S~\ref{sec:classification_on_disc} we explain that the moduli of nonresonant untwisted semisimple connections on discs are ($G$-orbits through) distinguished spaces of principal parts,
in dual TCLAs.
In \S~\ref{sec:quantum_orbits} we recall the definition of a $\ast$-product on untwisted semisimple coadjoint orbits in dual TCLAs,
and prove Thm.~\ref{thm:nuts_quantization}.
In \S~\ref{sec:coordinate_invariance} we show that the previous constructions are independent of the choice of uniformizers,
up to isomorphism;
i.e.,
that one can work canonically and `globally' on---%
principal bundles over---%
pointed Riemann surfaces.
In \S~\ref{sec:unframed_de_rham} we recall the standard finite-dimensional presentation of genus-zero de Rham varieties,
for trivial $G$-bundles,
and define their putative deformation quantization via quantum $G$-Hamiltonian reduction.
In \S\S~\ref{sec:interlude}--\ref{sec:flatness_general} we establish sufficient criteria for the flatness of the $G$-moment map,
proving Thm.~\ref{thm:main_result}.
Finally,
in \S~\ref{sec:stability} we provide a sufficient condition ensuring that all points in the zero-level set of the $G$-moment map are stable for the projectivization of $G$.

App.~\ref{sec:tcla_automorphisms} describes the group of Lie-algebra automorphisms of TCLAs.
App.~\ref{sec:proofs} collects proofs which have been postponed---%
aiming to ease the reading flow.

\subsection{Conventions,
	terminology,
	etc.}

Henceforth,
in this text,
we tacitly use the following notations/conventions/facts (unless otherwise specified):

\begin{itemize}
	\item
	      vector spaces,
	      tensor products,
	      schemes/varieties,
	      algebraic/Lie groups,
	      and Lie/associative algebras,
	      are all defined over $\mb C$;

	\item
	      if $V$ is a finite-dimensional vector space,
	      we identify it with the $\mb C$-points of the affine variety $\Spec \Sym (V^{\dual})$,
	      which is also denoted by $V$;
	      if $S$ is a nonsingular variety,
	      an algebraic map $\mu \cl S \to V$ (of varieties) induces a holomorphic map $\mu^{\on{an}} \cl S(\mb C) \to V$ (of manifolds);
	      the notation will \emph{not} distinguish between $\mu$ and its analytification $\mu^{\on{an}}$,
	      and,
	      at times,
	      neither between $S$ and $S(\mb C)$;
	      in the previous situation,
	      $\mu$ is surjective,
	      with equidimensional (scheme-theoretic) fibres,
	      if and only if this holds for $\mu^{\on{an}}$;

	\item
	      if $\bm G$ is an algebraic group,
	      an algebraic action $\bm G \ts S \to S$ induces a holomorphic action $G \ts S(\mb C) \to S(\mb C)$,
	      where $G \ceqq \bm G(\mb C)$ is the Lie group of $\mb C$-points of $\bm G$,
	      and again the notation might \emph{not} distinguish them;
	      a \emph{Hamiltonian} $\bm G$-\emph{variety} is a Poisson variety $\bigl( S,\set{\cdot,\cdot} \bigr)$,
	      equipped with a $\bm G$-action by Poisson automorphisms,
	      such that the infinitesimal action of $\mf g \ceqq \Lie(\bm G)$ is generated by a $\bm G$-equivariant \emph{moment map} $\mu \cl S \to \mf g^{\dual}$;

	\item
	      if $G$ is a reductive Lie group,
	      and $T \sse G$ a maximal torus,
	      the \emph{Weyl group of} $(G,T)$ is denoted by $W \ceqq N_G(T) \slash T$;
	      if $\mf t \ceqq \Lie(T) \sse \mf g$ is the corresponding Cartan subalgebra,
	      then the \emph{root system of} $(\mf g,\mf t)$ is denoted by $\Phi \sse \mf t^{\dual}$;
	      we let $W$ act on $\mf t$ and $\mf t^{\dual}$ in the standard way,
	      permuting the roots and the coroots $\alpha^{\dual} \in \Phi^{\dual} \sse \mf t$.
\end{itemize}

The end of remarks/examples is signalled by a `$\diamondsuit$',
inspired by~\cite{kirillov_2004_lectures_on_the_orbit_method}.

\section{Formal connections}
\label{sec:setup_on_disc}

\subsection{Principal bundles on DVRs}
\label{sec:bundles_on_disc}

We first consider `formal germs' of untwisted/unramified meromorphic connections on principal bundles over Riemann surfaces,
in an abstract setting. (cf~\cite{babbitt_varadarajan_1983_formal_reduction_theory_of_meromorphic_differential_equations_a_group_theoretic_view};
and~\cite{fernandezherrero_reduction_theory_for_connections_over_the_formal_punctured_disc},
\cite[\S~2.3.1]{doucot_rembado_tamiozzo_2024_moduli_spaces_of_untwisted_wild_riemann_surfaces},
for a modern view).

(Experts might want to skip to the `quantum' version of \S~\ref{sec:quantum_orbits}.)

\subsubsection{}

Let $\bm G$ be a connected reductive algebraic group,
with Lie algebra $\mf g \ceqq \Lie(\bm G)$.
Let also $\ms O$ be a \emph{complete} DVR,
with maximal ideal $\mf m \sse \ms O$ and residue field isomorphic to $\mb C$:
we fix an identification $\ms O \bs \mf m \simeq \mb C$,
and a coefficient field $\mb C \hra \ms O$~\cite{cohen_1946_on_the_structure_and_ideal_theory_of_complete_local_rings}.
If $\ms O \hra \ms K \ceqq \Frac(\ms O)$ is the fraction field,
consider the $\ms O$-submodules
\begin{equation}
	\label{eq:inverse_power_max_ideal}
	\mf m^{-s} \ms O
	\ceqq \Set{ f \in \ms K | \mf m^s f \sse \ms O },
	\qquad s \in \mb Z_{> 0}.
\end{equation}

In the terminology of~\cite[\S~A.1.1]{frenkel_benzvi_2001_vertex_algebras_and_algebraic_curves},
the \emph{disc} is $\mc D \ceqq \Spec \ms O$,
with \emph{origin} $0 \ceqq \mf m$.

\subsubsection{}

Now consider a principal $\bm G$-bundle $\pi \cl \mc E \to \mc D$.
One typically requires $\pi$ to be locally trivial in the fppf topology,
which over a local Henselian ring is the same as being étale-locally-trivial~\cite[Prop.~8.1]{demazure_grothendieck_1970_schemas_en_groupes_iii_structure_des_schemas_en_groupes_reductifs}.
In turn,
the first cohomology $H^1_{\on{\textnormal{ét}}} \bigl( \mc D;\bm G \bigr)$ is trivial,
and so $\pi$ is \emph{globally} trivializable.
Hereafter,
fix a trivialization $\mc E \simeq \mc D \ts \bm G \to \mc D$;
changing it amounts to acting by the \emph{gauge group of} $\pi$:
\begin{equation}
	\label{eq:formal_gauge_group}
	\bm G (\ms O)
	\ceqq \Hom_{\on{Sch}} (\mc D,\bm G).
\end{equation}

A \emph{frame of} $\pi$ \emph{at} $0$ is the choice of a $\mb C$-point in the special fibre,
i.e.,
equivalently,
an element $g \in G \ceqq \bm G(\mb C)$.
The triple $(\pi,0,g)$ is a \emph{framed} principal $\bm G$-bundle over the pointed disc $( \mc D,0 )$.
The \emph{based} gauge group of $(\pi,0,g)$ is the kernel of the group morphism $\bm G (\ms O) \thra G$,
suggestively denoted by $\bm G(\mf m)$.
The corresponding group sequence splits,
and the factorization $\bm G (\ms O) \simeq G \lts \bm G \bigl( \mf m \bigr)$ will be implicit.
There is a corresponding split exact sequence of Lie algebras:
\begin{equation}
	0 \lra \mf g (\mf m) \lra \mf g (\ms O) \lra \mf g \lra 0,
\end{equation}
where $\mf g (\mf m) \ceqq \mf g \ots \mf m$ and $\mf g(\ms O) \ceqq \mf g \ots \ms O$ inherit a Lie bracket from $\mf g$ via
\begin{equation}
	\label{eq:lie_algebra_from_associative_algebra}
	\bigl[ X \ots f,Y \ots f' \bigr]
	\ceqq [X,Y] \ots (ff'),
	\qquad X,Y \in \mf g,
	\quad f,f' \in \ms O.
\end{equation}
More generally,
for any $\ms O$-module $V$ we consider the vector space $\mf g(V) \ceqq \mf g \ots V$.
If $W \sse V$ is a submodule,
the $\mb C$-linear isomorphism $\mf g (V \slash W) \simeq \mf g(V) \bs \mf g(W)$ is tacit in what follows.
If $\mc R$ is a (possibly nonunital) commutative algebra,
then we regard $\mf g(\mc R)$ as a Lie algebra (cf.~\eqref{eq:lie_algebra_from_associative_algebra}).

\subsection{Connections}

We then look at connections on $\pi$,
possibly singular at the origin.
As in the nonformal setting (cf.~\S~\ref{sec:background}),
one can define them in terms of 1-forms on the total space of the principal bundle;
but here we will just work on the base---%
in the given trivialization.

\subsubsection{}

Consider the module of \emph{continuous} Kähler differentials of $\ms K \!\slash \mb C$,
denoted by
\begin{equation}
	\label{eq:kaehler_differentials}
	\ms K \lxra{\dif} \Omega^1_{\ms K}
	= \Omega^1_{\ms K \! \slash \mb C}.
\end{equation}
A \emph{connection on} $\pi$,
possibly singular at $0$,
can be just defined as an element $\wh{\mc A} \in \mf g (\Omega^1_{\ms K})$.
We will also use the $\ms O$-submodule $\Omega^1_{\ms O} = \Omega^1_{\ms O \slash \mb C} \hra \Omega^1_{\ms K}$,
so that $\mf g (\Omega^1_{\ms O}) \sse \mf g(\Omega^1_{\ms K})$ is the vector subspace of \emph{nonsingular connections}.
There is an $\ms O$-linear identification $\Omega^1_{\ms K} \simeq \Omega^1_{\ms O} \ots_{\ms O} \ms K$,
and analogously to~\eqref{eq:inverse_power_max_ideal} we consider the $\ms O$-submodule
\begin{equation}
	\mf m^{-s} \Omega^1_{\ms O}
	\ceqq \Omega^1_{\ms O} \ots_{\ms O} \bigl( \mf m^{-s} \ms O \bigr)
	\simeq \Set{ \alpha \in \Omega^1_{\ms K} | \mf m^s \alpha \in \Omega^1_{\ms O} },
	\qquad s \in \mb Z_{> 0}.
\end{equation}
Then $\mf g ( \mf m^{-s}\Omega^1_{\ms O} )$ is the subspace of connections \emph{with pole order bounded by} $s$.
Finally,
given a connection $\wh{\mc A}$,
its \emph{principal part} $\mc A$ is the equivalence class in $\mf g \bigl( \Omega^1_{\ms K} \slash \Omega^1_{\ms O} \bigr)$.

\subsubsection{}
\label{sec:tcla_duality}

Fix an invariant pairing $( \cdot \mid \cdot ) \cl \mf g \ots \mf g \to \mb C$.
Via the usual 2-cocycle of the loop algebra $\mf g (\ms K)$,
which involves the \emph{residue map} $\Res \cl \Omega^1_{\ms K} \to \mb C$,
we pair the space of pole-order-bounded principal parts with the TCLA
\begin{equation}
	\label{eq:takiff_lie_algebra}
	\mf g_s = \mf g (\ms O_s),
	\qquad \ms O_s \ceqq \ms O \bs \mf m^s,
\end{equation}
analogously to~\eqref{eq:tcla_intro} (cf.~\cite{boalch_2001_symplectic_manifolds_and_isomonodromic_deformations}).
This remains nondegenerate:
hereafter,
we consider the principal part $\mc A$,
of a pole-order-bounded connection $\wh{\mc A}$,
as an element of $\mf g_s^{\dual}$.

Denote by $\Lambda \ceqq \Res(\wh{\mc A}) \in \mf g$ the residue of a connection $\wh{\mc A}$.
Somewhat conversely,
the \emph{irregular part of} $\wh{\mc A}$ is the equivalence class of $\wh{\mc A}$ in $\mf g \bigl( \mf m^{-s} \Omega^1_{\ms O} \slash \mf m^{-1}\Omega^1_{\ms O} \bigr)$.
The structural derivation in~\eqref{eq:kaehler_differentials} induces a $\mb C$-linear isomorphism
\begin{equation}
	\ms K \!\slash \ms O \simeq \Omega^1_{\ms K} \bs \mf m^{-1} \Omega^1_{\ms O},
\end{equation}
and so:
(i) there exists $Q \in \mf g ( \mf m^{1-s} \ms O )$ such that the irregular part is represented by $\dif Q$;
and (ii) the class of $Q$ modulo $\mf g (\ms O)$ is uniquely determined by the irregular class.

\subsubsection{}
\label{sec:birkhoff}

Note that~\eqref{eq:takiff_lie_algebra} is also the Lie algebra of the (nonreductive) Lie group $G_s \ceqq \bm G(\ms O_s)$,
i.e.,
the \emph{truncated-current Lie group} (= TCLG).
Then the canonical projection $\ms O_s \thra \ms O_s \bs (\mf m \slash \mf m^s) \simeq \mb C$ yields an exact group sequence:
\begin{equation}
	\label{eq:tclg_sequence}
	1 \lra \on{Bir}_s \lra G_s \lra G \lra 1,
\end{equation}
whose kernel is sometimes referred to as the \emph{Birkhoff} subgroup.
Moreover,
the given splitting $\mb C \hra \ms O$ induces an analogous one $\mb C \hra \ms O_s$,\fn{
	Making $\ms O_s$ into a \emph{finite-dimensional} algebra,
	so that $G_s$ can also be regarded as an algebraic group;
	the notation will not distinguish the two structures.}~whence a semidirect factorization $G_s \simeq G \lts \on{Bir}_s$.
Infinitesimally,
this yields a Lie-algebra decomposition
\begin{equation}
	\mf g_s \simeq \mf g \lts \mf{bir}_s,
	\qquad \mf{bir}_s \ceqq \mf g ( \mf m \slash \mf m^s )
	\simeq \Lie(\on{Bir}_s).
\end{equation}
Then the dual vector space splits as $\mf g_s^{\dual} \simeq \mf g^{\dual} \ops \mf{bir}_s^{\dual}$:
the former direct summand corresponds to the residue,
and the latter to the irregular part (cf.~Rmk.~\ref{rmk:residue}).

\begin{rema}
	\label{rmk:nilradicals_etc}

	By hypothesis,
	the Lie algebra $\mf g \simeq \mf g_s \slash \mf{bir}_s$ is reductive,
	and $\mf{bir}_s \sse \mf g_s$ is a nilpotent ideal:
	hence,
	it is the nilradical.
	Analogously,
	$\on{Bir}_s \sse G_s$ is the unipotent radical.
	It follows that $\exp \cl \mf{bir}_s \to \on{Bir}_s$ is bijective,
	and that the product of the latter is controlled by the Lie bracket of the former via the truncated Baker--Campbell--Hausdorff formula (= BCH~\cite{dynkin_1947_calculation_of_the_coefficients_in_the_campbell_hausdorff_formula}).
	In particular,
	geometrically,
	one can view $\on{Bir}_s$ as an affine space.
\end{rema}

\subsection{Formal gauge action}

The gauge group~\eqref{eq:formal_gauge_group} acts by pullback on connections,
on the \emph{right},
in affine fashion:
cf.,
e.g.,
\cite[\S~2.1]{fernandezherrero_reduction_theory_for_connections_over_the_formal_punctured_disc}.
We conclude this section of intrinsic definitions by giving an explicit formula for it.
This relies on one further choice (but cf.~\S~\ref{sec:coordinate_invariance}).

\subsubsection{}

Let $\varpi$ be a \emph{uniformizer for} $\ms O$.
Then there are isomorphisms of algebras
\begin{equation}
	\ms O \simeq \mb C \llb \varpi \rrb,
	\qquad
	\qquad \ms K \simeq \mb C (\!( \varpi )\!).
\end{equation}
Moreover,
there are $\mb C$-linear isomorphisms $\mf m^s \simeq \varpi^s \mb C \llb \varpi \rrb$ and $\mf m^{-s} \ms O \simeq \varpi^{-s} \mb C \llb \varpi \rrb$,
as well as
\begin{equation}
	\Omega^1_{\ms O} \simeq \mb C \llb \varpi \rrb \dif \varpi,
	\qquad
	\Omega^1_{\ms K} \simeq \mb C (\!( \varpi )\!) \dif \varpi,
	\qquad
	\mf m^{-s} \Omega^1_{\ms O} \simeq \varpi^{-s} \mb C \llb \varpi \rrb \dif \varpi,
\end{equation}
which intertwine the structures of $\ms O$- and $\mb C \llb \varpi \rrb$-modules.
As customary,
we write $\bm G (\ms O) = G \llb \varpi \rrb$,
and $\mf g (\ms O) = \mf g \llb \varpi \rrb$,
etc.
Connections $\wh{\mc A} \in \mf g(\!( \varpi )\!) \dif \varpi$ of pole order bounded by $s \geq 1$ can be uniquely written as $\wh{\mc A} = \mc A + \mc B$,
with
\begin{equation}
	\label{eq:explicit_connection_disc}
	\mc A
	= \sum_{i = 0}^{s-1} A_i \varpi^{-i-1} \dif \varpi,
	\qquad \mc B \in \mf g \llb \varpi \rrb \dif \varpi,
\end{equation}
for coefficients $A_0,\dc,A_{s-1} \in \mf g$.
Assume that $A_{s-1} \neq 0$,
whence it is the \emph{leading coefficient}.
We say that $\mc B$ is the \emph{nonsingular part of} $\wh{\mc A}$,
and decompose the principal part as
\begin{equation}
	\mc A
	= \mc A_0 + \dif Q,
	\qquad \mc A_0
	= \Lambda \varpi^{-1}\dif \varpi,
\end{equation}
where in turn
\begin{equation}
	\label{eq:explicit_principal_part_disc}
	\qquad \Lambda
	\ceqq A_0
	= \Res(\wh{\mc A}),
	\qquad Q
	\ceqq \sum_{i = 1}^{s-1} A_i \frac{\varpi^{-i}}{-i}.
\end{equation}

Finally,
in the isomorphism $\ms O_s \simeq \mb C \llb \varpi \rrb \slash \varpi^s \mb C\llb \varpi \rrb$,
we will also denote by $\varpi$ the nilpotent class of the uniformizer,
and at times use the corresponding $\varpi$-grading of $\mf{bir}_s \sse \mf g_s$.

\subsubsection{}

Then the `constant' gauge transformations act in $\varpi$-graded fashion:
\begin{equation}
	\label{eq:constant_gauge_action}
	\wh{\mc A}.g
	= \Ad_{g^{-1}} \bigl( \wh{\mc A} \, \bigr),
	\qquad g \in G.
\end{equation}
On the other hand,
given $\bm X \in \varpi\mf g\llb \varpi \rrb$,
the based gauge transformation $\bm h \ceqq e^{\bm X}$ reads
\begin{equation}
	\label{eq:based_gauge_action}
	\wh{\mc A}.\bm h
	= e^{-\ad_{\bm X}} \bigl( \wh{\mc A} \, \bigr) + \wt e^{\, \ad_{\bm X}} (d \bm X),
\end{equation}
setting
\begin{equation}
	\wt e^{\, \ad_{\bm X}}
	= \frac{1 - e^{-\ad_{\bm X}}}{\ad_{\bm X}}
	\ceqq \sum_{i \geq 0} \frac{ (-1)^i }{ (i+1)! } \ad_{\bm X}^i \in \End_{\mb C} \bigl( \mf g \llb \varpi \rrb \dif \varpi \bigr).
\end{equation}
(The point is the identity $\dif \, (e^{\bm X}) = e^{\bm X} \cdot \wt e^{\, \ad_{\bm X}} (\dif \bm X)$.)

\begin{rema}
	When a uniformizer is given,
	denote the based gauge group by $G_1 \llb \varpi \rrb \sse G \llb \varpi \rrb$,
	and set $\mf g_1 \llb \varpi \rrb \ceqq \varpi \mf g \llb \varpi \rrb = \Lie \bigl( G_1 \llb \varpi \rrb \bigr)$.
	The exponential map $\mf g_1 \llb \varpi \rrb \to G_1 \llb \varpi \rrb$ is bijective,
	and every element of the based gauge group can be \emph{uniquely} written as an infinite product of elementary transformations of the form $e^{X\varpi^i}$,
	for $X \in \mf g$ and for an integer $i > 0$ (cf.~\cite[\S~1.5]{babbitt_varadarajan_1983_formal_reduction_theory_of_meromorphic_differential_equations_a_group_theoretic_view}).

	Moreover,
	one can equip $\mf g (\!( \varpi )\!)$ with the standard $\varpi$-adic (ultra)metric.
	Via the exponential map,
	we regard the based gauge group as a topological group.\fn{
		We will \emph{not} use proalgebraic structures;
		the main point is that we consider $G \llb \varpi \rrb$-orbits which are controlled by their truncated analogues,
		and work with the latter.}~Transferring the $\varpi$-adic topology onto $\mf g(\!( \varpi )\!) \dif \varpi$,
	the action~\eqref{eq:based_gauge_action} is then \emph{continuous}.
	(We omit the proof of this fact.)
\end{rema}

\begin{rema}
	\label{rmk:residue}

	In principle,
	the duality of \S~\ref{sec:tcla_duality} matches up an element of $\mf g$ with a residue 1-form $\mc A_0 = \Lambda \varpi^{-1}\dif \varpi$---%
	as in~\eqref{eq:explicit_connection_disc}.
	But the residue $\Lambda$ is independent of $\varpi$,
	and we just view it as an element of $\mf g^{\dual}$.
	This does \emph{not} work for the coefficients of the irregular part.
	Nonetheless,
	when a uniformizer is given,
	the $\Ad_G^{\dual}$-stabilizer of a principal/irregular part matches up with the $\Ad_G$-stabilizer of its coefficients,
	and these Adjoint/coadjoint actions correspond to the (inverse) constant gauge action~\eqref{eq:constant_gauge_action}.
\end{rema}

\section{Normal forms/orbits}
\label{sec:classification_on_disc}

\subsection{Based gauge invariants}
\label{sec:formal_normal_forms}

Here we recall one main statement around the classification of untwisted connections on the trivial principal $G$-bundle over the disc.
This amounts to looking at certain $G_1 \llb \varpi \rrb$-orbits inside $\mf g (\!( \varpi )\!) \dif \varpi$;
cf.~again~\cite{babbitt_varadarajan_1983_formal_reduction_theory_of_meromorphic_differential_equations_a_group_theoretic_view},
as well as~\cite{babbitt_varadarajan_1989_local_moduli_for_meromorphic_differential_equations},
amongst many others.

\subsubsection{}

Analogously to Rmk.~\ref{rmk:residue},
note that~\eqref{eq:based_gauge_action} acts \emph{linearly} on principal parts.
It preserves those of pole order bounded by $s \geq 1$,
on which the subgroup $\exp \bigl( \varpi^s \mf g \llb \varpi \rrb \bigr) \sse G_1 \llb \varpi \rrb$ acts trivially.
Hence,
truncating the gauge action,\fn{
	And tacitly inverting it,
	for convenience,
	turning it into a \emph{left} action.}~a $G_1\llb \varpi \rrb$-orbit yields an $\Ad^{\dual}_{\on{Bir}_s}$-orbit in $\mf g_s^{\dual}$,
in the duality of \S~\ref{sec:tcla_duality} (cf.~\eqref{eq:tclg_sequence}).
Conversely,
in this text we consider connections whose based gauge orbit is determined by the Birkhoff orbit of their principal part,
as follows.
(We view this terminology as standard.)

\begin{defi}
	\label{def:nuts_connections}

	Let $\wh{\mc A}$ be a connection on the framed principal $G$-bundle $(\pi,0,g)$.
	Then:
	\begin{enumerate}
		\item
		      $\wh{\mc A}$ is \emph{untwisted} (a.k.a.~\emph{unramified}) if the based gauge orbit $G_1\llb \varpi \rrb.\wh{\mc A}$ contains an element $\wh{\mc A}' = (\mc A'_0 + \dif Q') + \mc B'$ such that the coefficients of $Q'$ commute with each other and are semisimple,
		      in the notation of~\eqref{eq:explicit_connection_disc}--\eqref{eq:explicit_principal_part_disc};

		\item
		      if this holds,
		      $\wh{\mc A}$ is \emph{semisimple} if moreover $\Lambda' = A'_0 \ceqq \Res \bigl( \wh{\mc A}' \bigr)$ is semisimple and commutes with (the coefficients of) $Q'$;

		\item
		      if both hold,
		      $\wh{\mc A}$ is \emph{nonresonant} if moreover $\ad_{\Lambda'}$ has no nonzero integer eigenvalues upon restriction to the $\ad$-stabilizer of (the coefficients of) $Q'$.
	\end{enumerate}
\end{defi}

\begin{defi}
	\label{def:uts_nuts_terminology}

	For short,
	framed UnTwisted Semisimple connections (resp.,
	Nonresonant such) are said to be \emph{UTS} (resp.,
	\emph{NUTS}).
	Moreover,
	if $\wh{\mc A}$ is NUTS:
	\begin{enumerate}
		\item
		      the principal part $\mc A' = \mc A'_0 + \dif Q'$ is the \emph{normal form of} $\wh{\mc A}$,
		      and---%
		      conversely---%
		      such principal parts are also said to be \emph{NUTS};

		\item
		      $Q'$ is the \emph{irregular type of} $\wh{\mc A}$;

		\item
		      and $\Lambda' \in \mf g$ is the \emph{normal residue of} $\wh{\mc A}$.
	\end{enumerate}
	More generally,
	if $\wh{\mc A}$ is UTS then $\mc A'$ is said to be its (UTS) \emph{normal principal part},
	and we use the same terminology for its \emph{irregular type} and \emph{normal residue}.
\end{defi}

\begin{rema}
	One can allow for nonresonant residues with nilpotent part.
	And one can enter into the twisted/ramified setting,
	where the normal forms involve formal Puiseux gauge transformations (again,
	cf.~\cite{boalch_yamakawa_2015_twisted_wild_character_varieties,		boalch_doucot_rembado_2025_twisted_local_wild_mapping_class_groups_configuration_spaces_fission_trees_and_complex_braids,		doucot_rembado_yamakawa_twisted_g_local_wild_mapping_class_groups} and references therein,
	e.g.,
	~\cite{balser_jurkat_lutz_1979_a_general_theory_of_invariants_for_meromorphic_differential_equations_i_formal_invariants,		malgrange_1983_sur_les_deformations_isomonodromiques_ii_singularites_irregulieres,		martinet_ramis_1991_elementary_acceleration_and_multisummability_i},
	etc.).

	On the contrary,
	Deff.~\ref{def:nuts_connections}--\ref{def:uts_nuts_terminology} include \emph{generic} connections,
	whose based gauge orbit contain an element with regular semisimple leading coefficient,
	which used to be the classical irregular-singular setup (on vector bundles~\cite{jimbo_miwa_ueno_1981_monodromy_preserving_deformation_of_linear_ordinary_differential_equations_with_rational_coefficients_i_general_theory_and_tau_function}).
	More generally,
	they include the case of `complete fission'~\cite{boalch_2009_through_the_analytic_halo_fission_via_irregular_singularities},
	i.e.,
	when the $\Ad_G$-stabilizer of (the coefficients of) $Q'$ is a maximal torus of $G$.
\end{rema}

\begin{rema}
	\label{rmk:very_good_connections}

	If the condition of Def.~\ref{def:nuts_connections} (1.)~holds,
	then up to based gauge one can assume that $\Lambda'$ commutes with $Q'$:
	the condition (2.)~is precisely that the projection of the residue on the $\ad_\mf g$-stabilizer of (the coefficients of) $Q'$ be semisimple.
	Moreover,
	the former is equivalent to the following:
	there exists a maximal torus $T \sse G$,
	with Lie subalgebra $\mf t \ceqq \Lie(T) \sse \mf g$,
	such that $Q' \in \mf t (\!( \varpi )\!) \bs \mf t \llb \varpi \rrb$.\fn{
		\label{fn:intrinsic_irregular_type}
		Again,
		beware that this description is coordinate-dependent,
		and that the intrinsic definition is rather as an element of $\mf t (\ms K \bs \ms O )$ (cf.~\cite[Def.~7.1]{boalch_2014_geometry_and_braiding_of_stokes_data_fission_and_wild_character_varieties});
		if one insists in \emph{not} fixing a maximal torus from the start,
		then one should consider $\mb C \llb \varpi \rrb^{\ts}$-orbits of elements $Q'$ having semisimple commuting coefficients (cf.~\S~\ref{sec:coordinate_invariance}).}
\end{rema}

\subsubsection{}

The main point here is that the normal form is a \emph{complete} invariant of NUTS based gauge orbits.
To state it precisely,
for any element $Q$ as in~\eqref{eq:explicit_principal_part_disc} denote by $\mf l_Q \sse \mf g$ the intersection of the $\ad_\mf g$-stabilizer subalgebras
\begin{equation}
	\mf g^{A_i} \ceqq \ker (\ad_{A_i}),
	\qquad i \in \set{1,\dc,s-1}.
\end{equation}

\begin{enonce}{Proposition-Definition}
	\label{prop:based_gauge_principal_bundle}

	Let $X''_s \sse \varpi^{-s} \mf g \llb \varpi \rrb \dif \varpi$ be the topological subspace of NUTS connections of pole order bounded by $s$,
	and denote by $\tau''_s \cl X''_s \to \mf g_s^{\dual}$ the map $\wh{\mc A} \mt \mc A'$ (taking the normal form).
	Then $\tau''_s$ is a trivializable topological principal $G_1 \llb \varpi \rrb$-bundle over its image,
	viz.,
	over the subspace
	\begin{equation}
		\label{eq:nuts_normal_forms}
		\mf g''_s \ceqq
		\Set{ \Lambda' \varpi^{-1}\dif \varpi + \dif Q' \in \mf g'_s | \coker \bigl( \ad_{\Lambda'}(Q',\kappa) \bigr) = 0 \text{ for } \kappa \in \mb Z \sm \set{0} },
	\end{equation}
	where,
	in turn:
	\begin{enumerate}
		\item
		      we let
		      \begin{equation}
			      \label{eq:uts_normal_forms}
			      \mf g'_s
			      \ceqq \Set{ \Lambda' \varpi^{-1}\dif \varpi + \dif Q' \in \mf g_s^{\dual} | \Lambda',A'_1,\dc,A'_{s-1} \text{ are semisimple and commute } };
		      \end{equation}

		\item
		      we set
		      \begin{equation}
			      \ad_{\Lambda'}(Q',\kappa)
			      \ceqq \bigl( \eval[1]{\ad_{\Lambda'}}_{\mf l_{Q'}} - \kappa \cdot \Id_{\mf l_{Q'}} \bigr) \in \mf{gl}_{\mb C}(\mf l_{Q'}),
			      \qquad \kappa \in \mb Z;
		      \end{equation}

		\item
		      and we identify $\mf g_s^{\dual}$ with the topological quotient $\varpi^{-s} \mf g \llb \varpi \rrb \dif \varpi \bs \mf g \llb \varpi \rrb \dif \varpi$.
	\end{enumerate}
\end{enonce}

\begin{proof}[Scheme of a proof]
	Besides the fact that based gauge orbits through NUTS connections meet~\eqref{eq:nuts_normal_forms} in one point,
	and only one,
	the statement also summarizes the well-known fact that the $G_1 \llb \varpi \rrb$-stabilizers of NUTS connections are trivial.

	As per the topology/trivialization,
	one can show that $\tau''_s$ is continuous,
	and that the restriction $\mf g''_s \ts G_1 \llb \varpi \rrb \to X''_s$ of the based-gauge-action map is a homeomorphism.
	By construction,
	it intertwines the canonical projection onto $\mf g''_s$ with $\tau''_s$.
\end{proof}

\begin{rema}
	\label{rmk:eigenvalues_symmetry}

	In~\eqref{eq:nuts_normal_forms},
	it would be equivalent to ask that
	\begin{equation}
		\coker \bigl( \ad_{\Lambda'}(Q',\kappa) \bigr)
		= 0,
		\qquad \kappa \in \mb Z_{> 0},
	\end{equation}
	as the eigenvalues of the adjoint action arise in opposite pairs (cf.~Rmk.~\ref{rmk:about_stability}).
\end{rema}

\subsection{About full gauge invariants}
\label{sec:formal_normal_orbits}

The upshot is that the based gauge orbits of NUTS connections are determined by the coadjoint Birkhoff orbit of their principal part,
and that the latter meets~\eqref{eq:nuts_normal_forms} in exactly one point (cf.~the first part of~\cite{calaque_felder_rembado_wentworth_2024_wild_orbits_and_generalised_singularity_modules_stratifications_and_quantisation}).

\subsubsection{}

Hereafter we thus focus on $\mf g_s^{\dual}$,
viewed as an affine Poisson variety.
To get to the unframed principal bundle,
one must act by the `constant' part~\eqref{eq:constant_gauge_action} of the full gauge action.
To this extent,
note that the subspaces $\mf g''_s \sse \mf g'_s \sse \mf g_s^{\dual}$ of~\eqref{eq:nuts_normal_forms}--\eqref{eq:uts_normal_forms} are $G$-invariant.
Therefore,
the splitting $G_s \simeq G \lts \on{Bir}_s$ and Prop.-Def.~\ref{prop:based_gauge_principal_bundle} together imply the following fact:
if an $\Ad^{\dual}_{G_s}$-orbit $\mc O' \sse \mf g_s^{\dual}$ intersects either of the subspaces~\eqref{eq:nuts_normal_forms}--\eqref{eq:uts_normal_forms},
then the intersection is precisely a $G$-orbit.

Then one can use the following variations of Deff.~\ref{def:nuts_connections}--\ref{def:uts_nuts_terminology}:

\begin{defi}
	\label{def:unframed_nuts_connections}

	A connection $\wh{\mc A}$ on $\pi$ is \emph{nonresonant/untwisted/semisimple} if either condition holds upon framing $\pi$ at $0$.
\end{defi}

\begin{defi}
	\label{def:unframed_nuts_connections_2}

	For short,
	UnTwisted Semisimple connections (resp.,
	Nonresonant such) are said to be \emph{UTS} (resp.,
	\emph{NUTS}).
	Moreover,
	if $\wh{\mc A}$ is NUTS,
	then upon choosing any frame of $\pi$:
	\begin{enumerate}
		\item
		      the \emph{normal orbit of} $\wh{\mc A}$ is the $G$-orbit of the normal principal part $\mc A'$;

		\item
		      the \emph{irregular class} $\Theta' = \Theta(Q')$ \emph{of} $\wh{\mc A}$ is the $G$-orbit of the irregular type $Q'$;

		\item
		      and the \emph{normal residue orbit of} $\wh{\mc A}$ is the $G$-orbit of the normal residue $\Lambda'$.
	\end{enumerate}
	More generally,
	if $\wh{\mc A}$ is UTS then the $G$-orbit of $\mc A'$ is said to be its \emph{normal principal orbit},
	and we use the same terminology for its \emph{irregular class} and \emph{normal residue orbit}.
\end{defi}

\begin{rema}
	If we choose a starting maximal torus $T \sse G$ (cf.~Rmk.~\ref{rmk:very_good_connections}),
	then we find the usual notion that untwisted/unramified irregular classes are $W$-orbits in the quotient $\mf t (\!( \varpi )\!) \bs \mf t \llb \varpi \rrb$~\cite[Rmk.~10.6]{boalch_2014_geometry_and_braiding_of_stokes_data_fission_and_wild_character_varieties} (cf.~\cite{doucot_rembado_2025_topology_of_irregular_isomonodromy_times_on_a_fixed_pointed_curve}).\fn{
		Therefore,
		a notion of unramified irregular classes which is both coordinate- and torus-free can be obtained by looking at the orbits of elements $Q'$ (with semisimple commuting coefficients) under the group $G \ts \mb C\llb \varpi \rrb^{\ts}$,
		noting that the actions of the two factors commute (cf.~Fn.~\ref{fn:intrinsic_irregular_type},
		and again \S~\ref{sec:coordinate_invariance}).}
\end{rema}

\subsubsection{}

In a corollary of Prop.-Def.~\ref{prop:based_gauge_principal_bundle},
taking the normal orbit yields a bijection
\begin{equation}
	X''_s \bs G \llb \varpi \rrb \lxra{\simeq} \mf g''_s \slash G.
\end{equation}
What matters here is that choosing an $\Ad^{\dual}_{G_s}$-orbit which intersects $\mf g''_s \sse \mf g_s^{\dual}$ is thus the same as prescribing a full gauge orbit.
As mentioned in the introduction \S~\ref{sec:intro},
the former can therefore be viewed as local pieces of (nonresonant) wild de Rham spaces.

Hereafter,
the $\Ad_{G_s}^{\dual}$-orbits $\mc O' \sse \mf g_s^{\dual}$ are regarded as symplectic algebraic varieties;
but recall that we want to quantize \emph{affine} ones.
To this end,
we use the following generalization of the analogous fact for standard semisimple orbits:

\begin{prop}
	\label{prop:affine_orbits}

	The UTS $\Ad^{\dual}_{G_s}$-orbits are closed \emph{affine} subvarieties of $\mf g_s^{\dual}$.
\end{prop}

\begin{proof}
	Choose $\mc A' \in \mf g'_s$ and let $\mc O' \ceqq G_s.\mc A'$.
	Recall~\cite[Prop.~2.12]{hiroe_yamakawa_2014_moduli_spaces_of_meromorphic_connections_and_quiver_varieties} that $\mc O'$ is symplectomorphic to the Hamiltonian reduction of an `extended' orbit $\wt{\mc O}' \sse G \ts \mf g_s^{\dual}$ (cf.~\eqref{eq:extended_orbit}),
	with respect to the free action of the $\Ad^{\dual}_G$-stabilizer $L' \ceqq L_{Q'} \sse G$ of the irregular type $Q'$,
	along the $\Ad^{\dual}_{L'}$-orbit of $\Lambda' \in \mf g \simeq \mf g^{\dual}$ (cf.~\eqref{eq:orbit_from_extended_orbit}).
	In turn,
	$\wt{\mc O}'$ is symplectomorphic to the direct product of $\on T^*\!G$ with the $\Ad^{\dual}_{\on{Bir}_s}$-orbit of the irregular part $\dif Q'$~\cite[Cor.~3.2]{yamakawa_2019_fundamental_two_forms_for_isomonodromic_deformations} (cf.~\eqref{eq:decoupling}),
	and so it is affine.
	Finally,
	by hypothesis,
	the residue orbit is also affine.
	Altogether,
	it follows that $\mc O'$ is the quotient of an affine variety by a free $L'$-action.
	Now~\cite[Thm.~4.2]{popov_vinberg_1994_invariant_theory} implies that this is a geometric quotient:
	by uniqueness of the categorical quotient~\cite{mumford_fogarty_kirwan_1994_geometric_invariant_theory},
	it must be the \emph{affine} one.

	Furthermore,
	let $\ol{\mc O'} \sse \mf g_s^{\dual}$ be the orbit closure,
	and suppose by contradiction that the complement $Z \ceqq \ol{\mc O'} \sm \mc O'$ is nonempty.
	Since $\mc O'$ is affine and open in $\ol{\mc O'}$,
	$Z$ has pure codimension 1.
	On the other hand,
	$Z$ is a finite union of $\Ad^{\dual}_{G_s}$-orbits,
	and so it has codimension greater than 1.
\end{proof}

\section{Quantization of coadjoint orbits in dual TCLAs}
\label{sec:quantum_orbits}

\subsection{Stabilizers and Levi/parabolic filtrations}

Here we recall and complement the construction of a deformation quantization of UTS orbits.

\subsubsection{}

Let $\mc O' \sse \mf g_s^{\dual}$ be an $\Ad^{\dual}_{G_s}$-orbit such that $\mc O' \cap \mf g'_s \neq \vn$:
under a further technical hypothesis (cf.~Rmk.~\ref{rmk:balanced_polarizations}),
the second part of~op.~cit.~considers `parabolically' induced $U(\mf g_s)$-modules,
which ultimately lead to a deformation quantization of the restriction of the Lie--Poisson bracket of $\mf g_s^{\dual}$ to the symplectic leaf $\mc O'$.
Please refer to~\cite[\S~16]{calaque_felder_rembado_wentworth_2024_wild_orbits_and_generalised_singularity_modules_stratifications_and_quantisation} for details,
and note that~op.~cit.~is rather phrased in the complex-analytic category.

\subsubsection{}
\label{sec:levi_and_parabolic_filtrations}

Choose a UTS principal part $\mc A' \in \mc O'$.
The $\Ad^{\dual}_{G_s}$-stabilizer $\bm L \ceqq G_s^{\mc A'} \sse G_s$ of $\mc A'$ is
\begin{equation}
	\label{eq:stabilizer_nuts_principal_parts}
	\bm L
	= L_s \lts \prod_{i = 1}^{s-1} \exp (\mf l_{s-i} \cdot \varpi^i),
\end{equation}
where,
in turn:
(i) $L_s = L_{\mc A'} \sse G$ is the $\Ad_G$-stabilizer of $\mc A'$;
and (ii)
\begin{equation}
	\mf l_i
	\ceqq \mf g^{A'_{s-1}} \cap \dm \cap \mf g^{A'_{s-i}} \sse \mf g,
	\qquad i \in \set{1,\dc,s-1}.
\end{equation}
Set also $\mf l_s \ceqq \Lie(L_s) \sse \mf g$:
then $\mf l_i$ is a reductive Lie subalgebra,
integrating to a connected reductive subgroup $L_i \sse G$,
for $i \in \set{1,\dc,s}$.

Finally,
denote by $\bm{\mf l} \ceqq \mf g_s^{\mc A'} \sse \mf g_s$ the $\ad^{\dual}_{\mf g_s}$-stabilizer of $\mc A'$,
i.e.,
the Lie algebra of~\eqref{eq:stabilizer_nuts_principal_parts}.

\begin{rema}
	\label{rmk:fission}

	Equivalently,
	consider the decreasing sequence of reductive subgroups of $G$ obtained from the nested $\Ad_G$-stabilizers of the coefficients of $\mc A'$,
	i.e.:
	\begin{equation}
		L_0
		\ceqq G,
		\qquad L_{i+1}
		= L_i^{A'_{s-i}}
		\ceqq \Set{ g \in L_i | \Ad_g(A'_{s-i}) = A'_{s-i} },
	\end{equation}
	for $i \in \set{1,\dc,s-1}$.
	(And then take their Lie algebras.)

	This is a.k.a.~the \emph{fission sequence of} $G$ determined by $\mc A'$~\cite{boalch_2009_through_the_analytic_halo_fission_via_irregular_singularities,boalch_2014_geometry_and_braiding_of_stokes_data_fission_and_wild_character_varieties}.
	Note that opp.~citt.~refer to the subsequence determined by the irregular type $Q'$ alone,
	excluding $L_s$;
	but in this text the residue centralizers also play an important role in light of the main Thm.~\ref{thm:main_result}.
\end{rema}

\subsubsection{}
\label{sec:polarizations}

Now choose also two decreasing sequences
\begin{equation}
	\label{eq:parabolic_sequence}
	G \eqqc P^\pm_0 \supseteq P^\pm_1 \supseteq \dm \supseteq P^\pm_s,
\end{equation}
of parabolic subgroups of $G$,
such that $P_i^+ \cap P_i^- = L_i$,
so that $L_i$ is a Levi factor of $P^\pm_i$.
The corresponding sequence of parabolic subalgebras $\mf p^\pm_i \ceqq \Lie(P^\pm_i) \sse \mf g$ determines two \emph{opposite polarizations} on $\mc O'$,
in the form of a $G_s$-invariant Lagrangian splitting of the tangent bundle
\begin{equation}
	\label{eq:polarizations}
	\on T \! \mc O'
	\simeq G_s \ts^{\bm L} ( \mf g_s \slash \bm{\mf l} ) \lra \mc O',
\end{equation}
in the usual identification $\mc O' \simeq G_s \slash \bm L$.
This involves the nested nilradicals of the parabolic subalgebras,
which must satisfy a further condition (cf.~Rmk.~\ref{rmk:balanced_polarizations}).

\subsection{Representation-theoretic setup}
\label{sec:representation_theory_quantization}

We now use the parabolic filtrations $\mf p^\pm_s \sse \dm \sse \mf p^\pm_1 \sse \mf g$ to define parabolically-induced Verma modules for TCLAs.

\subsubsection{}

By construction (and duality),
the coefficient $A'_{s-i}$ corresponds to a linear map on the Lie-centre $\mf Z(\mf l_i) \sse \mf l_i$,
extending to a character $\chi_i^\pm \cl \mf p^\pm_i \to \mb C$.
We assemble them into a character $\bm \chi^\pm$ of the Lie subalgebra $\bm{\mf p}^\pm \ceqq \bops_{i = 0}^{s-1} ( \mf p^\pm_{s-i} \cdot \varpi^i ) \sse \mf g_s$,
in $\varpi$-graded fashion:
\begin{equation}
	\bm\chi^\pm \cl \sum_{i = 0}^{s-1} X_i \varpi^i \lmt \sum_{i = 0}^{s-1} \chi^\pm_{s-i}(X_i),
	\qquad X_i \in \mf p^\pm_{s-i}.
\end{equation}

\begin{defi}[Cf.~\cite{chaffe_topley_2023_category_o_for_truncated_current_lie_algebras,calaque_felder_rembado_wentworth_2024_wild_orbits_and_generalised_singularity_modules_stratifications_and_quantisation}]
	\label{def:vermas}

	The balanced tensor products
	\begin{equation}
		\label{eq:vermas}
		M^\pm = M^{(s)}_{\bm{\mf p}^\pm,\bm \chi^\pm}
		\ceqq \Ind_{U(\bm{\mf p}^\pm)}^{U(\mf g_s)} (\bm\chi^\pm)
		= U(\mf g_s) \ots_{U(\bm{\mf p}^\pm)} (\bm \chi^\pm),
	\end{equation}
	are the \emph{Verma modules of highest/lowest weight} $\mc A'$,
	\emph{parabolic filtration} $\bm{\mf p}^\pm$,
	and \emph{truncation index} $s$.
\end{defi}

\begin{exem}
	The generic case correspond to $\mf l_1 \sse \mf g$ being a Cartan subalgebra.
	In this case $\mf p^\pm_s = \dm = \mf p^\pm_1$ are two opposite Borel subalgebras (cf.~\cite{wilson_2011_highest_weight_theory_for_truncated_current_lie_algebras,felder_rembado_2023_singular_modules_for_affine_lie_algebras_and_applications_to_irregular_wznw_conformal_blocks}).
	Somewhat conversely,
	if $s = 1$ then $M^\pm$ are two `opposite' standard parabolic $U(\mf g)$-Verma modules (cf.~\cite[Chp.~9]{humphreys_2008_representations_of_semisimple_lie_algebras_in_the_bgg_category_o}).
\end{exem}

\begin{rema}
	\label{rmk:vermas_are_quotients}

	It follows that $M^\pm$ are cyclically generated by the vectors
	\begin{equation}
		\label{eq:cyclic_vectors}
		v^\pm
		= v^\pm_{\bm{\mf p}^\pm,\bm \chi^\pm}
		\ceqq 1 \ots_{U\bm{\mf p}\pm} 1,
	\end{equation}
	on which $U\bm{\mf p}^\pm$ act by $\bm \chi^\pm$.
	This yields an alternative presentation of~\eqref{eq:vermas},
	as quotients of $U(\mf g_s)$ by the left submodules annihilating~\eqref{eq:cyclic_vectors}.

	Now $v^+$ is a \emph{highest-weight vector} in the sense of~\cite[\S~4.1]{chaffe_topley_2023_category_o_for_truncated_current_lie_algebras},
	and so $M^+$ is a \emph{highest-weight module} (cf.~\S~4.2 of op.~cit.).
	Moreover,
	being a quotient of a regular/non-parabolic Verma module of the same highest weight,
	$M^+$ lies in the category $\mc O^{(\chi^+_s,\dc,\chi^+_1)}$ of~\cite[\S~4.3]{chaffe_topley_2023_category_o_for_truncated_current_lie_algebras}.
	(Cf.~\cite{chaffe_2023_category_o_for_takiff_lie_algebras} when $s = 2$.)
	Analogously,
	$M^-$ is a \emph{lowest-weight module},
	generated by the \emph{lowest-weight vector} $v^-$,
	etc.
\end{rema}

\subsubsection{}

One also has a flag $\mf Z(\mf g) \sse \mf Z(\mf l_1) \sse \dm \sse \mf Z(\mf l_s)$,
and $M^\pm$ are weight modules for the largest Levi centre,
which acts in semisimple fashion.
Moreover,
there is a canonical antipode-contragredient bilinear form
\begin{equation}
	\label{eq:shapovalov}
	\mc S
	= \mc S_{\bm{\mf p}^\pm,\bm \chi^\pm} \cl M^- \ots M^+ \lra \mb C,
\end{equation}
such that $\mc S(v^- \ots v^+) = 1$,
which we call after Shapovalov~\cite{shapovalov_1972_a_certain_bilinear_form_on_the_universal_enveloping_algebra_of_a_complex_semisimple_lie_algebra}.
The simultaneous decomposition into $\mf Z(\mf l_s)$-weight spaces is $\mc S$-orthogonal,
and the left/right radicals of~\eqref{eq:shapovalov} coincide with the maximal proper submodules of $M^\pm$.
What matters is that the deformation quantization of $\mc O'$ involves the inverse tensor of~\eqref{eq:shapovalov} (cf.~\cite{alekseev_lachowska_2005_invariant_star_product_on_coadjoint_orbits_and_the_shapovalov_pairing}).
Concretely,
denoting as customary by $\hs$ a formal deformation parameter,
there is a natural $\ast$-product
\begin{equation}
	\label{eq:star_product}
	\ast \cl \mc R_0 \ots \mc R_0 \lra \mc R_0 \llb \hs \rrb,
	\qquad \mc R_0 \ceqq \mb C [\mc O'],
\end{equation}
on the commutative Poisson ring of regular functions.

Now the construction of~\eqref{eq:star_product} requires that $M^\pm$ is simple/irreducible for generic values of $\mc A'$,
i.e.,
that the Shapovalov form is nondegenerate.
To this extent,
the most direct link with meromorphic gauge theory seems to be the following:

\begin{prop}[cf.~\cite{calaque_felder_rembado_wentworth_2024_wild_orbits_and_generalised_singularity_modules_stratifications_and_quantisation}, Conj.~15.2.9]
	\label{prop:quantum_nonresonance}

	If $\mc A' \in \mf g''_s$,
	then the modules~\eqref{eq:vermas} are \emph{simple}.
\end{prop}

\begin{proof}
	Let us write the modules of truncation index $s \geq 1$ as
	\begin{equation}
		M^\pm_s = M^{(s)}(\mf p^\pm_s,\dc,\mf p^\pm_1;\chi^\pm_s,\dc,\chi^\pm_1).
	\end{equation}

	The `parabolic restriction' functor of~\cite[Thm.~4.1]{chaffe_topley_2023_category_o_for_truncated_current_lie_algebras} takes $M^\pm_s$ to the parabolic Verma modules $\wt M^{\pm}_s = M^{(s)} \bigl( \, \wt{\mf p}^\pm_s,\dc,\wt{\mf p}^\pm_1;\wt\chi^\pm_s,\dc,\wt\chi_1^\pm \bigr)$,
	for the TCLA on the reductive Lie subalgebra ${\mf l}_1 \subseteq \mf g$,
	where
	\begin{equation}
		\wt{\mf p}^\pm_i
		\ceqq \mf p_i \cap \mf l_1,
		\qquad \wt\chi^\pm_i
		\ceqq \eval[1]{\chi_i^\pm}_{\wt{\mf p}^\pm_i},
		\qquad i \in \set{1,\dc,s}.
	\end{equation}
	(Note that $\wt{\mf p}^\pm_i$ are parabolic subalgebras of $\mf l_1$.)
	Applying the functor of Lem.~3.7 of op.~cit.~then yields the modules
	\begin{equation}
		\ul{\wt M}^\pm_s
		\ceqq M^{(s)} \bigl( \, \wt{\mf p}^\pm_s,\dc,\wt{\mf p}^\pm_2,\wt{\mf p}^\pm_1;\wt\chi^\pm_s,\dc,\wt\chi^\pm_2,0 \bigr).
	\end{equation}
	Since $\wt{\mf p}_1 = \mf l_1$,
	the latter are precisely the modules
	\begin{equation}
		\wt M_{s-1}^\pm = M^{(s-1)} \bigl( \, \wt{\mf p}^\pm_s,\dc,\wt{\mf p}^\pm_2;\wt\chi^\pm_s,\dc,\wt\chi^\pm_2 \bigr),
	\end{equation}
	for the $(s-1)$-truncated TCLA on ${\mf l}_1$,
	with an additional zero-action of ${\mf l}_1 \cdot \varpi^{s-1} \sse \mf g_s$.
	Since these functors are exact equivalences,
	the modules $M^\pm_s$ are simple if and only if $\wt M_{s-1}^\pm$ are simple.
	By induction on $s \geq 1$,
	$M^\pm_s$ are simple if and only this holds for
	\begin{equation}
		M_1^\pm = M^{(1)}_{{\mf p}^\pm_s \cap {\mf l}_{s-1},\chi^\pm_s},
	\end{equation}
	with tacit restriction of the characters.
	Moreover,
	it follows that $\mf l_s \sse \mf p^\pm_s \cap \mf l_{s-1}$ is a Levi factor,
	so that $\chi^\pm_s$ are determined by a linear functional $\lambda' \in \mf Z(\mf l_s)^{\dual}$.

	We conclude the proof in the highest-weight case.
	Choose a Cartan subalgebra $\mf t \sse \mf g$ such that $\mf t \sse \mf l_s$,
	and view $\lambda'$ as a linear function $\mf t \to \mb C$ vanishing on the subspace $\mf t \cap [\mf l_s,\mf l_s] \sse \mf t$.
	Denote by $\phi_s \sse \Phi$ the Levi subsystem determined by $\mf l_s$.
	Finally,
	let $\mf b \sse \mf g$ be a Borel subalgebra such that $\mf t \sse \mf b \sse \mf p^+_s$.
	Then~\cite[Thm.~9.12~(a)]{humphreys_2008_representations_of_semisimple_lie_algebras_in_the_bgg_category_o} implies that the Verma module $M_{\lambda'} = M^+_1$ is simple if
	\begin{equation}
		\label{eq:simplicity_verma}
		\Braket{ \lambda' + \rho, \alpha^{\dual} } \not\in \mb Z \sm \set{0},
		\qquad \alpha \in \Phi \sm \phi_s,
	\end{equation}
	where $\alpha^{\dual} \in \Phi^{\dual} \sse \mf t$ are the coroots and $\rho \in \mf t^{\dual}$ is the Weyl vector.
	But $\lambda' \in \mf t^{\dual}$ matches up with the normal residue $\Lambda' \in \mf t$ under the $G$-invariant duality used above,
	and the nonresonance condition of~\eqref{eq:nuts_normal_forms} can be rewritten as
	\begin{equation}
		\label{eq:cartan_nonresonance}
		\Braket{ \lambda',\alpha^{\dual}}
		= \Braket{ \alpha,  \Lambda' } \not\in \mb Z \sm \set{0},
		\qquad         \alpha \in \Phi \sm \phi_s.
	\end{equation}
	Finally,
	recalling that $\braket{ \rho, \Phi^{\dual} } \sse \mb Z$,
	one sees that~\eqref{eq:cartan_nonresonance} implies~\eqref{eq:simplicity_verma}.
\end{proof}

\begin{rema}
	\label{rmk:about_stability}

	The usual sufficient stability condition for $M_{\lambda'}$ rather reads
	\begin{equation}
		\label{eq:true_simplicity_verma}
		\Braket{ \lambda' + \rho,\alpha^{\dual} } \not\in \mb Z_{> 0},
		\qquad \alpha \in \Phi^+ \sm \phi_s,
	\end{equation}
	and it is thus \emph{weaker} than~\eqref{eq:simplicity_verma}.
	The fact is that the Weyl-invariant nonresonance condition~\eqref{eq:nuts_normal_forms} does \emph{not} depend on the choice of a system of positive roots,
	and conversely it implies that \emph{all} the induced Verma modules will be simple,
	regardless of the parabolic sequence containing $\mf t$.

	Nonetheless,
	recall that if $\lambda$' is regular then the wall-avoiding constraints~\eqref{eq:true_simplicity_verma} are also \emph{necessary} for $M_{\lambda'}$ to be simple (cf.~\cite[Thm~9.12~(b)]{humphreys_2008_representations_of_semisimple_lie_algebras_in_the_bgg_category_o}).
	In the proper parabolic case,
	instead,
	a sharp criterion is given in~\cite[Satz~4]{janzten_1977_kontravariante_formen_auf_induzierten_darstellungen_halbeinfacher_lie_algebren}.
\end{rema}

\subsection{Quantum comoments}

For later use,
we prove that the $\ast$-product admits a \emph{strong} quantum comoment map,
adapting a result of~\cite{calaque_naef_2015_a_trace_formula_for_the_quantization_coadjoint_orbits}.
(Cf.~again~\cite[\S~16]{calaque_felder_rembado_wentworth_2024_wild_orbits_and_generalised_singularity_modules_stratifications_and_quantisation}.)

\subsubsection{}

Concretely~\eqref{eq:star_product} is defined as follows on pure tensors $f \ots f' \in \mc R_0 \ots \mc R_0$:
(i) denote by $\wt f,\wt f' \in \mb C[G_s]$ the pullbacks of $f,f'$ along the fibre-bundle projection $G_s \thra \mc O'$;
(ii) apply to $\wt f \ots \wt f'$ the formal bidifferential operator $F_{\hs^{-1}}$,
obtained as the formal Taylor expansion of
\begin{equation}
	\label{eq:inverse_shapovalov}
	F_c
	= F_{c,\bm{\mf p}^\pm,\bm\chi^\pm}
	\ceqq \mc S^{-1}_{\bm{\mf p}^\pm,c\bm\chi^\pm},\fn{
		This element exists when the Verma modules of parameters $(\bm{\mf p}^\pm,c\bm\chi^\pm)$ are simple,
		thereby defining a meromorphic function on the $c$-plane:
		with poles at the degeneracy locus of the Shapovalov form,
		but always nonsingular at infinity.}
\end{equation}
as $c \in \mb C^{\ast}$ tends to infinity;
(iii) compose with the multiplication of $\mb C[G_s]$ to obtain a formal function on $G_s$;
and (iv) note that the latter descends to an element $f \ast f' \in \mc R_0 \llb \hs \rrb$,
in view of the invariance of the Shapovalov form.

\begin{rema}
	The fact that $F_1$ is well-defined,
	i.e.,
	that the starting Shapovalov form $\mc S$ is nondegenerate,
	is immaterial.
	Rather,
	it is sufficient that~\eqref{eq:inverse_shapovalov} is well-defined for all but countably many values of $c$,
	which follows from Prop.~\ref{prop:quantum_nonresonance}.
	Albeit the quantization detects nonresonance only `asymptotically',
	the latter is necessary on the semiclassical side:
	to identify the moduli of principal parts with the local moduli of meromorphic connections,
	cf.~\S~\ref{sec:unframed_de_rham}.
\end{rema}

\begin{rema}
	\label{rmk:balanced_polarizations}

	Consider the nilradicals $\mf u^\pm_i \sse \mf p^\pm_i$ of the parabolic subalgebras of \S~\ref{sec:representation_theory_quantization}.
	Then we also assume that the vector subspaces $\bm{\mf u}^\pm \ceqq \bops_{i = 0}^{s-1} (\mf u^\pm_{s-i} \cdot \varpi^i) \sse \mf g_s$ are Lie subalgebras,
	a condition which was referred to as having `balanced polarizations' in~\cite{calaque_felder_rembado_wentworth_2024_wild_orbits_and_generalised_singularity_modules_stratifications_and_quantisation}.
	This is used to expand the formal bidifferential operator $F_{\hs^{-1}}$,
	as follows:
	the choice of PBW bases of the Verma modules yields $U(\bm{\mf u}^{\mp})$-linear isomorphisms $M^\pm \simeq U(\bm{\mf u}^{\mp})$,
	and one can write
	\begin{equation}
		\label{eq:variable_separation}
		F_{\hs^{-1}}
		= \sum_{i \geq 0} F^{(i)}_{\hs^{-1}} \cdot \hs^i,
		\qquad F^{(i)}_{\hs^{-1}} \in U(\bm{\mf u}^-) \ots U(\bm{\mf u}^+),
	\end{equation}
	regarding as usual $U(\mf g_s)$ as the ring of left-invariant differential operators on $\mb C[G_s]$.

	(One can perhaps get rid of this requirement by a more extensive usage of PBW bases,
	considering $\Sym (\bm{\mf u}^\mp )$ in general.)
\end{rema}

\subsubsection{}

Now recall that the embedding $\mu = \mu_{\mc O'} \cl \mc O' \hra \mf g_s^{\dual}$ is a moment map for the $\Ad^{\dual}_{G_s}$-action,
which contravariantly yields a \emph{comoment map}
\begin{equation}
	\label{eq:classical_comoment}
	\mu^* \cl \Sym (\mf g_s) \lra \mc R_0,
	\qquad X \lmt \mu^*(X) = f_X \ceqq \eval[1]X_{\mc O'},
\end{equation}
for $X \in \Sym (\mf g_s) \simeq \mb C [\mf g_s^{\dual}]$.
Then~\eqref{eq:classical_comoment} generates the infinitesimal action of $\mf g_s$ on $\mc O'$,
via the following commutative triangle of Lie-algebra morphisms:
\begin{equation}
	\label{eq:hamiltonian_action}
	\begin{tikzcd}
		& \mc R_0 \ar{d}{\ad_{\mc R_0}} \\
		\mf g_s \ar{ru}{\mu^*} \ar{r} & \mf{der}(\mc R_0)
	\end{tikzcd}.
\end{equation}
A \emph{quantum comoment map} instead is a continuous $\mb C\llb \hs \rrb$-algebra morphism
\begin{equation}
	\label{eq:quantum_comoment}
	\wh\mu^* = \wh\mu^*_{\hs} \cl U_{\hs} (\mf g_s) \lra \wh{\mc R}_{\hs} \ceqq \bigl( \mc R_0 \llb \hs \rrb,\ast \bigr),
\end{equation}
where $U_{\hs} (\mf g_s)$ is the completed Rees algebra of $U(\mf g_s)$ (a.k.a.~the \emph{homogenized} universal enveloping algebra~\cite[Exmp.~2.6.2]{schedler_2012_deformations_of_algebras_in_noncommutative_geometry},
cf.~\cite{etingof_2007_calogero_moser_systems_and_representation_theory}).
Analogously to~\eqref{eq:hamiltonian_action},
this generates a $\mf g_s$-action on $\wh{\mc R}_{\hs}$;
and the latter is by definition a \emph{quantization} of the `semiclassical' action,
provided that there is a commutative square
\begin{equation}
	\label{eq:quantization_of_action}
	\begin{tikzcd}
		U_{\hs} (\mf g_s) \ar{r}{\wh{\mu}^*} \ar[two heads]{d} & \wh{\mc R}_{\hs} \ar[two heads]{d} \\
		\Sym (\mf g_s) \ar{r}[swap]{\mu^*} & \mc R_0
	\end{tikzcd},
\end{equation}
where the vertical arrows are the semiclassical limits.

But of course in our setting there is a natural `quantum' action,
viz.,
the infinitesimal version of the $G_s$-action on the coefficients of $\mc R_0$-valued formal power series.
Then the following statement implies in particular that~\eqref{eq:quantum_product} is $G_s$-invariant:

\begin{enonce}{Theorem-Definition}
	\label{thm:strong_quantum_comoment}

	Consider the continuous $\mb C\llb \hs \rrb$-algebra morphism $\wh\mu^* \cl U_{\hs}(\mf g_s) \to \wh{\mc R}_{\hs}$ extending
	\begin{equation}
		X \lmt f_X \cdot \hs^{-1} \in \wh{\mc R}_{\hs}[\hs^{-1}],
		\qquad X \in \mf g_s,
	\end{equation}
	in the notation of~\eqref{eq:classical_comoment}.
	Then:
	\begin{enumerate}
		\item
		      $\wh\mu^*$ is a \emph{strong} quantum comoment map generating the natural `quantum' $\mf g_s$-action;

		\item
		      and the corresponding square~\eqref{eq:quantization_of_action} commutes.
	\end{enumerate}
\end{enonce}

\begin{proof}
	The point is proving the identity
	\begin{equation}
		f_X \ast f' - f' \ast f_X \eqqcolon \bigl[ f_X,f' \bigr]_\ast = \hs \set{ f_X,f' },
		\qquad X \in \mf g_s,
		\quad g \in \mc R_0,
	\end{equation}
	for which one can generalize the proof of~\cite[Prop.~3.2]{calaque_naef_2015_a_trace_formula_for_the_quantization_coadjoint_orbits}.
	In some detail,
	the evaluation
	\begin{equation}
		\label{eq:evaluation_linear_function}
		\ev_X \cl \mf g_s^{\dual} \lra \mb C[G_s], \qquad \ev_X(\xi) \cl g \lmt \braket{ \Ad^{\dual}_g \xi, X},
	\end{equation}
	maps $\mc A'$ to the pullback $\wt f_X \in \mb C[G_s]$ of $f_X \in \mc R_0$.
	Moreover, the arrow~\eqref{eq:evaluation_linear_function} intertwines the $\ad^{\dual}_{\mf g_s}$-action on $\mc A'$ with the standard action of $\mf g_s$ on $\mb C[G_s]$.
	Hence,
	for all $c \in \mb C^\ast$ such that~\eqref{eq:inverse_shapovalov} is defined,
	and for all $\wt f' \in \mb C[G_s]$,
	one has
	\begin{align}
		F_c \bigl( \, \wt f' \ots \wt f_X \bigr) & = (1 \ots \ev_X) F_c \bigl( \, \wt f' \ots \mc A' \bigr)                                                           \\
		                                         & = \wt f' \ots \wt f_X - \frac 1 c \sum_j \bigl( Y_j.\wt f' \bigr) \ots X_j.\wt f_X \in \mb C[G_s] \ots \mb C[G_s],
	\end{align}
	in view of Lem.~\ref{lem:image_inverse_shapovalov}.
	Now take formal Taylor expansions in $\hs = c^{-1}$,
	and compose with the multiplication of $\mb C[G_s]$;
	this yields the $\bm L$-invariant formal function
	\begin{equation}
		\label{eq:stronger_quantum_comoment}
		m \circ F_{\hs^{-1}} \bigl( \, \wt f' \ots \wt f_X \bigr) = m \Bigl( \wt f' \ots \wt f_X - \hs \sum_j \bigl( Y_j.\wt f' \bigr) \ots X_j.\wt f_X \Bigr) \in \bigl( \mb C[G_s] \bigr) \llb \hs \rrb.
	\end{equation}
	Swapping the factors and projecting to the orbit yields
	\begin{equation}
		\bigl[ f',f_X \bigr]_{\! \ast} = - \hs \sum_j Y_j \wdg X_j \bigl( f' \ots f_X \bigr) = \hs \set{ f',f_X } \in \wh{\mc R}_{\hs},
	\end{equation}
	as $\Pi \ceqq \sum_j X_j \wdg Y_j \in \bm{\mf u}^+ \ots \bm{\mf u}^- \sse \bigwedge^2 \bigl( \mf g_s \slash \bm{\mf l} \bigr)$ is the Poisson bivector field.
\end{proof}

\begin{lemm}
	\label{lem:image_inverse_shapovalov}

	For any $c \in \mb C^\ast$ such that~\eqref{eq:inverse_shapovalov} is defined,
	let $v^+_c \ceqq v^+_{\bm{\mf p}^+,c\bm\chi^+}$ be the cyclic vector of the Verma module $M^+_c \ceqq M^+_{\bm{\mf p}^+,c\bm\chi^+}$.
	Denote also by $(Y_j)_j$ and $(X_j)_j$ two mutually-dual bases of $\bm{\mf u}^-$ and $\bm{\mf u}^+$,
	respectively,
	for the nondegenerate pairing $\bm{\mf u}^- \ots \bm{\mf u}^+ \to \mb C$ defined by the character $\bm \chi^+$.
	Then
	\begin{equation}
		\label{eq:image_inverse_shapovalov}
		F_c (v^+_c \ots \mc A') = v^+_c \ots \mc A' - \frac 1 c \sum_j ( Y_j v^+_c ) \ots \ad^{\dual}_{X_j} (\mc A') \in M^+_c \ots \mf g_s^{\dual},
	\end{equation}
	viewing $\mf g_s^{\dual}$ as a $U(\mf g_s)$-module for the $\ad^{\dual}_{\mf g_s}$-action.
\end{lemm}

\begin{proof}
	Postponed to \S~\ref{proof:lem_image_inverse_shapovalov}.
\end{proof}

\begin{rema}
	\label{rmk:no_cartan_quantum}

	Outside of proofs we phrased everything with \emph{no} Cartan subalgebra $\mf t \sse \mf g$.
	Choosing one,
	$\mc A'$ can be regarded as a tuple of linear functions on $\mf t$,
	lying in a root-valuation stratum of $(\mf t^{\dual})^s \simeq \mf t^s$ determined by the nested Levi subsystems of roots which annihilate the irregular tail at each step (cf.~\cite{yamakawa_2019_fundamental_two_forms_for_isomonodromic_deformations,doucot_rembado_tamiozzo_2022_local_wild_mapping_class_groups_and_cabled_braids,		calaque_felder_rembado_wentworth_2024_wild_orbits_and_generalised_singularity_modules_stratifications_and_quantisation},
	Rmk.~\ref{rmk:fission},
	and the proof of Prop.~\ref{prop:flatness_high_order}).
\end{rema}

\subsection{About canonicity}
\label{sec:using_marking_2}

Besides a UTS $\Ad_{G_s}^{\dual}$-orbit $\mc O' \sse \mf g_s^{\dual}$,
in this section we made two further related choices:
(i) a point $\mc A' \in \mc O' \cap \mf g'_s$,
a.k.a.~a \emph{marking of} $\mc O'$;
and (ii) a sequence of parabolic subalgebras of $\mf g$.
About the former:

\begin{lemm}
	\label{lem:no_orbit_marking}

	For any given parabolic sequence $\mf p^\pm_s \sse \dm \sse \mf p^\pm_1$,
	the above deformation quantization of $\mc R_0$ is \emph{independent} of the choice of marking.
\end{lemm}

\begin{proof}
	We constructed a $G_s$-invariant $\ast$-product with a strong quantum comoment,
	and with \emph{separation of variables} with respect to the polarizations determined by the sequences of nested nilradicals $\mf u_i^\pm \sse \mf p^\pm_i$:
	the latter follows immediately from the expansion~\eqref{eq:variable_separation},
	cf.~\cite[\S~2.2 + Rmk.~3.5]{calaque_naef_2015_a_trace_formula_for_the_quantization_coadjoint_orbits},
	and the conclusion then follows from Prop.~3.6 of op.~cit.
\end{proof}

\subsubsection{}

Thus,
after establishing the coordinate independence of \S~\ref{sec:coordinate_invariance},
a parabolic sequence makes it possible to \emph{canonically} quantize a coadjoint UTS orbit $\mc O' \sse \mf g_s^{\dual}$.
This arises from the formal germ of an untwisted/unramified meromorphic connection,
at a marked point of a (wild) Riemann surface,
which is our underlying viewpoint.

On the contrary,
the dependence upon the choice of polarization hinges on the parabolic sequence,
and we shall \emph{not} discuss this further in this text.

\section{Coordinate independence}
\label{sec:coordinate_invariance}

\subsection{Changing uniformizer}

Here we clarify which of the previous definitions/constructions is intrinsically determined by a triple $(\bm G,\ms O,s)$,
consisting of:
(i) a connected reductive algebraic group;
(ii) a complete DVR with residue/coefficient field $\mb C$;
and (iii) a positive integer.

\subsubsection{}
\label{sec:ring_automorphisms}

The set of uniformizers is a torsor for the group of continuous algebra automorphisms of $\mb C \llb \varpi \rrb$,
which will be identified with the units $\mb C \llb \varpi \rrb^{\ts} \sse \mb C \llb \varpi \rrb$:
an automorphism $F$ is determined by the image $F(\varpi) = \varpi f \in \varpi \mb C \llb \varpi \rrb$,
for a unique element $f = f(F)$ such that $f(0) \neq 0 \in \mb C$.

Then let $\mb C \llb \varpi \rrb^{\ts}$ act on $\mb C (\!( \varpi )\!) \dif \varpi$,
on the \emph{right},
by pullback:
\begin{equation}
	\label{eq:pullback_action}
	\bigl( \, \wt f \dif \varpi \bigr).F
	= F^* \bigl( \, \wt f \dif \varpi \bigr)
	\ceqq \bigl( \, \wt f \circ F(\varpi) \bigr) \cdot (f + \varpi f') \dif \varpi,
	\qquad \wt f \in \mb C (\!( \varpi )\!).
\end{equation}
This action extends in bilinear fashion to pole-order-bounded principal parts.
In the identifications/dualities of~\S~\ref{sec:tcla_duality},
we thus find a Poisson action on $\mf g_s^{\dual}$,
mapping $\Ad^{\dual}_{G_s}$-orbits onto each other and fixing the subspace $\mf g^{\dual}$ (of residues) pointwise.
Then:

\begin{lemm}
	\label{lem:coordinate_invariance}

	\leavevmode

	\begin{enumerate}
		\item
		      The subspaces $\mf g''_s \sse \mf g'_s \sse \mf g_s^{\dual}$ of~\eqref{eq:nuts_normal_forms}--\eqref{eq:uts_normal_forms} are $\mb C \llb \varpi \rrb^{\ts}$-invariant.

		\item
		      And if $\wt{\mc A}'$ lies in the $\mb C \llb \varpi \rrb^{\ts}$-orbit of $\mc A' \in \mf g'_s$,
		      then the sequence $\mf l_s \sse \dm \sse \mf l_1 \sse \mf g$ of reductive subalgebras determined by $\wt{\mc A}'$ and $\mc A'$---%
		      as in \S~\ref{sec:levi_and_parabolic_filtrations}---%
		      coincide.

	\end{enumerate}
\end{lemm}

\begin{proof}[Proof postponed to~\ref{proof:lem_coordinate_invariance}]
\end{proof}

\subsubsection{}

It follows from Lem.~\ref{lem:coordinate_invariance} that the spaces of UTS/NUTS connections of Deff.~\ref{def:nuts_connections} +~\ref{def:unframed_nuts_connections} admit an invariant description.
Namely,
define subspaces $\mf g''_s \sse \mf g'_s \sse \mf g \bigl( \Omega^1_{\ms K} \slash \Omega^1_{\ms O} \bigr)$ of pole-order-bounded principal parts by requiring that they coincide with~\eqref{eq:nuts_normal_forms}--\eqref{eq:uts_normal_forms} in \emph{any} choice of uniformizer,
and ask that the $\Ad^{\dual}_{G_s}$-orbit of the principal part of a connection $\wh{\mc A} \in \mf g (\Omega^1_{\ms K})$ intersects them (cf.~Rmk.~\ref{rmk:global_intrinsicality}).
Moreover,
one can associate two parabolic filtrations $\mf p^\pm_s \sse \dm \sse \mf p^\pm_1 \sse \mf g$ to any element $\mc A'$ of these subspaces,
so that the underlying Levi filtration controls the nested $\Ad_G$-stabilizers of the coefficients of $\mc A'$ in some/any choice of uniformizer.
Finally,
The Verma modules~\eqref{eq:vermas} still depend on $\varpi$,
precisely to select a character $\bm \chi^\pm$ for the Lie subalgebras $\bm{\mf p}^\pm \sse \mf g_s$;
nonetheless,
the isomorphism class of the deformation quantization of \S~\ref{sec:quantum_orbits} is \emph{intrinsic}:

\begin{lemm}
	\label{lem:invariant_vermas}

	\leavevmode

	\begin{enumerate}
		\item
		      Any element $F \in \mb C \llb \varpi \rrb^{\ts}$ yields a $U(\mf g_s)$-linear isomorphism
		      \begin{equation}
			      \label{eq:coordinate_change_on_vermas}
			      \prescript{}F\varphi \cl M^{(s)}_{\bm{\mf p}^\pm,\bm\chi^\pm} \lxra{\simeq} M^{(s)}_{\bm{\mf p}^\pm,\wt{\bm\chi}^\pm},
			      \qquad \wt{\bm\chi}^\pm
			      \ceqq \bm\chi^\pm.F.\fn{
				      Here we view $\bm\chi^\pm$ as elements of~\eqref{eq:uts_normal_forms}.
				      Equivalently,
				      the characters correspond to $s$-uples of elements in the dual centres of $\mf l_i = \mf p^+_i \cap \mf p^-_i$,
				      and these centres are also preserved---%
				      and acted on---%
				      by $\mb C \llb \varpi \rrb^{\ts}$.}
		      \end{equation}

		\item
		      And $\prescript{}{F'F}\varphi = \prescript{}{F'}\varphi \circ \prescript{}F\varphi$,
		      for any other element $F' \in \mb C \llb \varpi \rrb^{\ts}$.
	\end{enumerate}
\end{lemm}

\begin{proof}
	The $\mb C\llb \varpi \rrb^{\ts}$-action on $\mf g_s^{\dual}$ also comes by dualizing an action on $\mf g_s$,
	by Lie-algebra automorphisms (cf.~\S~\ref{sec:tcla_automorphisms}).
	Then the latter extends to an action on $U(\mf g_s)$,
	by algebra automorphisms,
	mapping the left-ideals annihilating the cyclic vectors~\eqref{eq:cyclic_vectors} onto each other.
	The compatibility with compositions is automatic.
\end{proof}

\begin{rema}
	The group $\mb C \llb \varpi \rrb^{\ts}$ also acts on $G \llb \varpi \rrb$,
	and one can consider the corresponding semidirect product:
	the latter can be identified with the group of bundle automorphisms of $\mc D \ts \bm G \to \mc D$,
	covering \emph{any} automorphism of the base.
	After truncation,
	this leads to $\ms O_s^{\ts} \lts G_s$,
	which is intimately related with the group of Lie-algebra automorphisms of $\mf g_s$ (cf.~Thm.~\ref{thm:tcla_automorphisms}).
\end{rema}

\begin{rema}
	The optional choice of a maximal torus $T \sse G$ is compatible with the above,
	as the corresponding truncated-current Cartan subalgebra $\mf t_s \sse \mf g_s$ is $\mb C \llb \varpi \rrb^{\ts}$-invariant.
	Then Lem.~\ref{lem:coordinate_invariance} implies that each root-valuation stratum of $\mf t_s^{\dual} \simeq (\mf t^{\dual})^s$ is $\mb C \llb \varpi \rrb^{\ts}$-invariant,
	and recall that the $\Ad^{\dual}_{G_s}$-orbits through a given stratum are canonically isomorphic as $G_s$-varieties.
	(They are also \emph{symplectomorphic},
	provided that they are related by the $\mb C \llb \varpi \rrb^{\ts}$-action).
\end{rema}

\section{Wild de Rham spaces,
  and their quantization}
\label{sec:unframed_de_rham}

\subsection{(Semiclassical) Hamiltonian reduction}
\label{sec:semiclassical_reduction}

Keep all the notation from \S~\ref{sec:intro}:
we work on $\Sigma = \mb CP^1$,
with finitely many marked points $\bm a \sse \Sigma$,
and consider the `open part' $\mc C^*_{\dR} \sse \mc C_{\dR}$ of the naive de Rham groupoid;
its objects are certain meromorphic connections $\bm{\mc A}$,
with polar divisor bounded by $D = \sum_{\bm a} s_a[a]$,
defined on the trivial holomorphic principal $G$-bundle over $\Sigma$.

\subsubsection{}

Equivalently,
$\Ob(\mc C_{\dR}^*)$ consists of suitably constrained $\mf g$-valued meromorphic 1-form on the Riemann sphere---%
still denoted by $\bm{\mc A}$.
We regard $\Sigma$ as the 1-point compactification of the standard complex $z$-plane,
and suppose that all marked points $a \in \bm a$ are at finite distance.
If $t_a \ceqq z(a) \in \mb C$ is the position of each marked point,
then one can uniquely write
\begin{equation}
	\label{eq:connection_on_p1}
	\bm{\mc A} = \sum_{\bm a} \mc A_a,
	\qquad \mc A_a = \Lambda_a z_a^{-1} \dif z + \dif Q_a,
	\quad z_a
	\ceqq z - t_a,
\end{equation}
for suitable residues $\Lambda_a = A_{a,0} \in \mf g$,
where as usual
\begin{equation}
	Q_a = \sum_{i = 1}^{s_a - 1} A_{a,i} \frac{z_a^{-i}}{-i},
	\qquad A_{a,1},\dc,A_{a,s_a-1} \in \mf g.
\end{equation}
Moreover,
the identity $\sum_a \Lambda_a = 0$ must hold in $\mf g$,
lest there is a regular singularity at infinity.

Now $z_a$ yields a uniformizer $\varpi_a$,
and the principal part of the formal Laurent expansion $\wh{\mc A}_a$ of~\eqref{eq:connection_on_p1}---%
as in~\eqref{eq:formal_laurent_expansion}---%
coincides with $\mc A_a$ (just written in a formal variable).
The nonsingular part
\begin{equation}
	\label{eq:nonsingular_part_p1}
	\mc B_a
	\ceqq \wh{\mc A}_a - \mc A_a \in \mf g \llb \varpi_a \rrb \dif \varpi_a,
\end{equation}
instead,
is computed by gathering the formal Taylor expansions of $\mc A_{a'}$ at $a$,
for all $a' \neq a$.
The final hypothesis is that $\wh{\mc A}_a$ lies in the $G \llb \varpi_a \rrb$-orbit $\wh{\mc O}'_a$ of a NUTS principal part $\mc A'_a$ (cf.~\eqref{eq:normal_form_intro}),
so that in particular~\eqref{eq:nonsingular_part_p1} can be formally gauged away.
This is \emph{equivalent} to asking that $\mc A_a$ lies in the $\Ad^{\dual}_{G_{s_a}}$-orbit $\mc O'_a \ceqq G_{s_a}.\mc A'_a$,
invoking a TCLG (cf.~\eqref{eq:tclg_sequence}).
Denote by $\bm{\mc O}' = \set{\mc O'_a}_{\bm a}$ the multiset of these finite-dimensional orbits.

\begin{rema}
	\label{rmk:global_intrinsicality}

	The fact that~\eqref{eq:connection_on_p1} admits a NUTS normal form at each pole is an \emph{intrinsic} notion.
	Namely,
	introduce the completed local field $\ms K_a = \Frac(\ms O_a)$ at each marked point,
	and the module of continuous Kähler differentials $\ms K_a \lxra{\dif} \Omega^1_{\ms K_a}$ (cf.~\eqref{eq:kaehler_differentials}).
	Then choose an orbit for the $\bm G (\ms O_a)$-action on $\mf g(\Omega^1_{\ms K_a})$:
	define it to be \emph{NUTS} (resp.,
	\emph{UTS}) if it meets an element of~\eqref{eq:nuts_normal_forms} (resp.,
	of~\eqref{eq:uts_normal_forms}),
	in some/any choice of uniformizer,
	which makes sense by Lem.~\ref{lem:coordinate_invariance}.
	This works on arbitrary principal $G$-bundles,
	over arbitrary base surfaces.
\end{rema}

\subsubsection{}

Overall,
denoting the orbit product by
\begin{equation}
	\label{eq:orbit_product}
	\mc O'_{\bm a} \ceqq \prod_{\bm a} \mc O'_a \sse \prod_{\bm a} \mf g_{s_a}^{\dual},
\end{equation}
there is a natural injective function
\begin{equation}
	\label{eq:fin_dim_description_de_rham}
	\Ob(\mc C_{\dR}^*) \lhra M
	= M(\bm{\mc O}';G)
	\ceqq \Set{ \prod_{\bm a} \mc A_a \in \mc O'_{\bm a} | \sum_{\bm a} \Lambda_a = 0 }.
\end{equation}
Conversely,
it is \emph{surjective} by the NUTS condition.
Finally,
the global holomorphic gauge transformations of the trivial bundle are \emph{constant},
and match up with the diagonal $G$-action on the principal parts $\mc A_a \in \mf g_{s_a}^{\dual}$.
So~\eqref{eq:fin_dim_description_de_rham} induces a bijection
\begin{equation}
	\Ob(\mc C^*_{\dR}) \bs \on{iso.} \lxra{\simeq} M \bs G.
\end{equation}
But we keep working in the complex-algebraic category,
considering the affine GIT quotient
\begin{equation}
	\label{eq:de_rham_as_git}
	\mc M_{\dR}^*
	\ceqq M \bs\!\!\bs \bm G,
\end{equation}
viewing $M \sse \mc O'_{\bm a}$ as an affine subvariety (cf.~Prop.~\ref{prop:affine_orbits}).
Recall that we assume that all points of~\eqref{eq:fin_dim_description_de_rham} are \emph{stable} (cf.~\S~\ref{sec:stability}),
so that~\eqref{eq:de_rham_as_git} is a genuine geometric quotient---%
and orbit space.

\begin{rema}
	\label{rmk:resonant_de_rham}

	In principle,
	one might consider a groupoid of (possibly resonant) meromorphic connections,
	with prescribed formal gauge orbits through UTS principal parts,
	at each pole.
	As in~\eqref{eq:fin_dim_description_de_rham},
	its set of objects can still be identified with a subset of the $\mb C$-points of a finite-dimensional affine variety:
	it is unclear---%
	to the authors---%
	whether this has nice algebro-geometric/symplectic structures.
\end{rema}

\subsection{Quantum Hamiltonian reduction}

In brief,
consider the \emph{affine-GIT Hamiltonian reduction}
\begin{equation}
	\label{eq:semiclassical_reduction}
	\mc M^*_{\dR}
	= \mc O'_{\bm a} \, \bs\!\!\bs\!\!\bs_{\!\!0} \, \bm G
	\ceqq (\bm\mu')^{-1}(0) \, \bs\!\!\bs \, \bm G,
\end{equation}
of a product of coadjoint orbits in dual TCLAs,
with respect to the moment map
\begin{equation}
	\label{eq:moment_map}
	\bm\mu' = \bm\mu'_{\mc O'_{\bm a}} \cl \mc O'_{\bm a} \lra \mf g^{\dual} \simeq \mf g,
	\qquad \prod_{\bm a} \mc A_a \lmt \sum_{\bm a} \Res(\mc A_a).
\end{equation}
The aim is to quantize~\eqref{eq:semiclassical_reduction},
using the material of \S~\ref{sec:quantum_orbits}.

\subsubsection{}

Denote by $\mc R_{a,0} \ceqq \mb C[\mc O'_a]$ (resp.,
by $\mc R_{\bm a,0} \ceqq \mb C[\mc O'_{\bm a}] \simeq \bots_{\bm a} \mc R_{a,0}$) the ring of global regular functions on each orbit (resp.,
on the orbit product).
Then there is a comoment map $(\bm\mu')^* \cl \Sym(\mf g) \to \mc R_{\bm a,0}$,
corresponding to the pullback along~\eqref{eq:moment_map}.
Now the contravariant version of~\eqref{eq:semiclassical_reduction} goes as follows (cf.~\cite{etingof_2007_calogero_moser_systems_and_representation_theory}):
(i) the extension of the \emph{augmentation ideal} along $\bm\mu'$,
i.e.,
\begin{equation}
	\label{eq:augmentation_pushforward}
	\mf J_{\bm a,0} \ceqq \mc R_{\bm a,0}\cdot (\bm\mu')^*\bigl( \Sym(\mf g) \cdot \mf g \bigr) \sse \mc R_{\bm a,0},
\end{equation}
intersects the subring $\mc R_{\bm a,0}^\mf g \sse \mc R_{\bm a,0}$ in a Poisson ideal;
and (ii) there is an isomorphism
\begin{equation}
	\label{eq:semiclassical_reduction_2}
	\mc M_{\dR}^* \simeq \Spec \bigl( \mc R_{\bm a,0} \, \bs\!\!\bs\!\!\bs_{\! 0} \, G \bigr),
	\qquad  \mc R_{\bm a,0} \,\bs\!\!\bs\!\!\bs_{\! 0} \, G
	\ceqq \mc R_{\bm a,0}^\mf g \bs \bigl( \mf J_{\bm a,0} \cap \mc R_{\bm a,0}^\mf g \bigr),
\end{equation}
of affine Poisson varieties.

Now,
for each $a \in \bm a$,
let $\wh{\mc R}_{a,\hs} \simeq \mc R_{a,0} \llb \hs \rrb$ be the deformation quantization of $\mc R_{a,0}$ from \S~\ref{sec:quantum_orbits}.
Then their completed $\mb C \llb \hs \rrb$-multilinear tensor product,
viz.,
\begin{equation}
	\label{eq:quantum_product}
	\wh{\mc R}_{\bm a,\hs}
	\ceqq \bigl(\wh\bots_{\mb C \llb \hs \rrb}\bigr)_{\!\bm a} \bigl( \wh{\mc R}_{a,\hs} \bigr) \simeq \mc R_{\bm a,0} \llb \hs \rrb,
\end{equation}
yields a deformation quantization of~\eqref{eq:orbit_product}.
Moreover,
Thm.-Def.~\ref{thm:strong_quantum_comoment} yields a \emph{quantization} $(\wh{\bm\mu}')^*$ of the comoment map $(\bm\mu')^*$,
as per~\eqref{eq:quantum_comoment}--\eqref{eq:quantization_of_action}.
It generates an action of $\mf g$ on $\wh{\mc R}_{\bm a,\hs}$,
quantizing the `semiclassical' $\mf g$-action on the orbit product.
The `quantum' analogue of~\eqref{eq:semiclassical_reduction_2} is the $\mb C \llb \hs \rrb$-algebra
\begin{equation}
	\label{eq:quantum_reduction}
	\wh{\mc R}_{\bm a,\hs} \, \bs\!\!\bs\!\!\bs_{\! 0} G
	\ceqq \wh{\mc R}_{\bm a,\hs}^\mf g \bs \bigl( \mf J_{\bm a,\hs} \cap \wh{\mc R}_{\bm a,\hs}^\mf g \bigr),
\end{equation}
where $\mf J_{\bm a,\hs} \sse \wh{\mc R}_{\bm a,\hs}$ is the left extension of the augmentation ideal.

We will conclude by the following fact,
ensuring that~\eqref{eq:quantum_reduction} is a \emph{flat} deformation (the crux of the matter is $\hs$ being a \emph{nonzerodivisor},
cf.~\cite[Prop.~XVI.2.4]{kassel_1995_quantum_groups},
~\cite[Prop.~1.1.2]{bruns_herzog_1993_cohen_macaulay_rings},
and~\cite[Prop.~1]{losev_2021_quantization_commutes_with_reduction}):

\begin{lemm}
	\label{lem:quantum_flatness}

	Suppose that the $G$-moment map~\eqref{eq:moment_map} is \emph{flat}.
	Then~\eqref{eq:quantum_reduction} is a deformation quantization of $\mc M^*_{\dR}$.
\end{lemm}

\begin{proof}[Proof omitted]
\end{proof}

\section{Interlude:
  preparation for flatness}
\label{sec:interlude}

The upshot of \S\S~\ref{sec:quantum_orbits}--\ref{sec:unframed_de_rham} is a deformation quantization of wild de Rham spaces,
whenever the $G$-moment map~\eqref{eq:moment_map} is a flat morphism.
In \S~\ref{sec:flatness_general} we provide sufficient conditions to ensure that it is \emph{faithfully flat},
and to this end we first recall further notions/terminology.

\subsection{Miracle flatness}

First,
note that~\eqref{eq:orbit_product} is an \emph{irreducible} variety,
being a product of homogeneous spaces.
(Incidentally,
the affine-GIT quotient~\eqref{eq:semiclassical_reduction} is also \emph{irreducible}.)

Second,
recall that a morphism of irreducible nonsingular varieties is faithfully flat if and only if:
(i) it is surjective;
and (ii) it has equidimensional fibres.
In general flatness is stronger,
and the nontrivial implication is proven,
e.g.,
in~\cite{nowak_1997_flat_morphisms_between_regular_varieties}---%
under the weaker assumption that the map is \emph{dominant}.
(Cf.~also~\cite[Thm.~26.2.11]{vakil_2024_the_rising_sea_foundations_of_algebraic_geometry}.)

\subsection{Tori,
	Borels,
	and reduction to the semisimple case}

Hereafter,
it will be convenient to fix---%
w.l.o.g.---%
a maximal torus $T \sse G$,
with Lie algebra $\mf t \sse \mf g$.
We will also use two Borel subgroups $B^\pm \sse G$ intersecting at $T$,
with Lie algebras $\mf b^\pm \sse \mf g$ and unipotent radicals $U^\pm \sse B^\pm$;
and let $\mf u^\pm \ceqq \Lie(U^\pm)$.
Then consider a $\mf t$-\emph{valued} principal part,
which is automatically UTS:
\begin{equation}
	\label{eq:normal_form_reindexing}
	\mc A'
	= \sum_{i = 0}^{s-1} A'_i \varpi^{-i-1} \dif \varpi,
	\qquad \Lambda = A'_0,A'_1,\dc,A'_{s-1} \in \mf t,
\end{equation}
for an integer $s \geq 1$.
It has \emph{leading term}
\begin{equation}
	\label{eq:leading_term}
	\mc A'_{\on{top}} \ceqq A'_{s-1} \varpi^{-s} \dif \varpi,
\end{equation}
and its $\Ad^{\dual}_{G_s}$-orbit is still denoted by $\mc O' \ceqq G_s.\mc A' \sse \mf g_s^{\dual}$.
Note that the choice of a uniformizer is also w.l.o.g.:
changing $\varpi$ just moves to another such orbit,
by a symplectomorphism,
in $G$-equivariant fashion.

Now observe the following (cf.~\cite[Rmk.~2.9]{yamakawa_2019_fundamental_two_forms_for_isomonodromic_deformations}):

\begin{enonce}{Lemma-Definition}
	\label{lem:semisimple_reduction}

	Let $\mf g = \mf Z(\mf g) \ops [\mf g,\mf g]$ be the splitting into centre plus semisimple part;
	extend/dualize it to $\mf g_s^{\dual}$,
	and denote by $\mc A' = \mc A'_{\mf Z} + \wt{\mc A}'$ the corresponding decomposition of the normal form.
	Then $\mc O'$ is \emph{canonically} isomorphic to the $\Ad^{\dual}_{G_s}$-orbit through $\wt{\mc A}'$,
	as a Hamiltonian $G_s$-variety.
\end{enonce}

\begin{proof}[Proof omitted]
\end{proof}

\subsubsection{}

In view of Lem.~\ref{lem:semisimple_reduction},
up to symplectomorphisms of de Rham spaces one can consider normal forms with \emph{no} central component.
Until \S~\ref{sec:stability} let us just suppose that $G$ is \emph{semisimple}.

\subsection{Birkhoff/extended orbits}
\label{sec:extended_orbits}

Introduce also:
(i) the \emph{Birkhoff orbit}
\begin{equation}
	\label{eq:birkhoff_orbit}
	\check{\mc O}'
	= \check{\mc O}'_{Q'} \ceqq \on{Bir}_s.(\dif Q') \sse \mf{bir}_s^{\dual},
\end{equation}
through the irregular part of $\mc A'$;
and (ii) the \emph{extended orbit}
\begin{equation}
	\label{eq:extended_orbit}
	\wt{\mc O}'
	= \wt{\mc O}'_{Q'}
	\ceqq \Set{ (g,\mc A) \in G \ts \mf g_s^{\dual} | \pi_{\irr} \bigl( \Ad^{\dual}_g(\mc A) \bigr) \in \check{\mc O}' },
\end{equation}
using the canonical projection $\pi_{\irr} \cl \mf g_s^{\dual} \thra \mf{bir}_s^{\dual}$ along $\mf g^{\dual}$ (cf.~\S~\ref{sec:birkhoff}).

\subsubsection{}
\label{sec:setup_flatness}

We will use results about the Hamiltonian geometry of~\eqref{eq:birkhoff_orbit}--\eqref{eq:extended_orbit},
keeping the notation of \S~\ref{sec:quantum_orbits} for the reductive subgroups/subalgebras $L_i \sse G$ and $\mf l_i \sse \mf g$ determined by $\mc A'$.
In addition,
let $\pi_i \cl \mf g^{\dual} \thra \mf l_i^{\dual}$ be the transposition of the Lie-algebra inclusion $\mf l_i \hra \mf g$,
for $i \in \set{1,\dc,s}$.
Under the usual dualities,
one can view them as linear maps $\mf g \thra \mf l_i$.
Write in particular
\begin{equation}
	\pi'
	\ceqq \pi_{s-1} \cl \mf g^{\dual} \lthra (\mf l')^{\dual},
	\qquad \mf l'
	= \mf l_{Q'}
	\ceqq \mf l_{s-1}.
\end{equation}
Finally,
trivialize $\on T^*\!G \simeq G \ts \mf g^{\dual}$ using left translations:
elements are written $(g,\Lambda)$,
with $g \in G$ and $\Lambda \in \mf g \simeq \mf g^{\dual}$ (cf.~Rmk.~\ref{rmk:residue}).
Then the following holds true (cf.~\cite{boalch_2001_symplectic_manifolds_and_isomonodromic_deformations,yamakawa_2019_fundamental_two_forms_for_isomonodromic_deformations}):
\begin{enumerate}
	\item
	      there is a `decoupling' symplectomorphism $\on T^*\!G \ts \check{\mc O}' \lxra{\simeq} \wt{\mc O}'$,
	      i.e.,
	      \begin{equation}
		      \label{eq:decoupling}
		      \bigl( (g,\Lambda),\dif Q \bigr) \lmt (g,\mc A),
		      \qquad \mc A \ceqq \Lambda \varpi^{-1}\dif \varpi + \Ad^{\dual}_g(\dif Q);
	      \end{equation}

	\item
	      the $\Ad^{\dual}_{L'}$-action on $\check{\mc O}'$ is Hamiltonian,
	      with a moment map
	      \begin{equation}
		      \label{eq:birkhoff_moment_map}
		      \check\mu'
		      = \check\mu'_{\check{\mc O}'} \cl \check{\mc O}' \lra (\mf l')^{\dual} \simeq \mf l',
	      \end{equation}
	      satisfying $\check\mu'(\dif Q') = 0$;

	\item
	      the direct product $G \ts L'$ acts on $\wt{\mc O}'$,
	      via
	      \begin{equation}
		      \label{eq:product_action_extended_orbit}
		      \bigl( (g,\Lambda),\dif Q \bigr) \lmt \bigl( \bigl( h g \wt g^{-1},\Ad_{\wt g}(\Lambda) \bigr), \Ad_h (\dif Q) \bigr),
		      \qquad \wt g \in G,
		      \quad h \in L';
	      \end{equation}

	\item
	      the action~\eqref{eq:product_action_extended_orbit} is Hamiltonian,
	      and generated by the moment map
	      \begin{equation}
		      \wt\mu'
		      = (\wt\mu'_1,\wt\mu'_2) \cl \wt{\mc O}' \lra \mf g^{\dual} \ts (\mf l')^{\dual} \simeq \mf g \ts \mf l',
	      \end{equation}
	      where $\wt\mu'_1 \cl \bigl( (g,\Lambda),\dif Q \bigr) \lmt \Lambda$,
	      and
	      \begin{equation}
		      \label{eq:second_component_moment_map}
		      \wt\mu'_2 \cl \bigl( (g,\Lambda),\dif Q \bigr) \lmt \check\mu'(\dif Q) - \pi'\bigl( \Ad_g(\Lambda) \bigr);
	      \end{equation}

	\item
	      and there is a $G$-equivariant symplectomorphism
	      \begin{equation}
		      \label{eq:orbit_from_extended_orbit}
		      \wt{\mc O}' \, \bs\!\!\bs\!\!\bs_{\! (-\Lambda')} \\ \, L'
		      \ceqq (\wt \mu'_2)^{-1}(- \Lambda') \, \bs\!\!\bs \\ \, L_s \lxra{\simeq} \mc O',
	      \end{equation}
	      recalling that $L_s \sse L'$ is the $\Ad_{L'}$-stabilizer of the normal residue.
\end{enumerate}

\section{Criteria for flatness}
\label{sec:flatness_general}

\subsection{}

Here we prove the main Thm.~\ref{thm:main_result}.
Namely,
recall from \S~\ref{sec:main_theorem} that we associate an integer $\nu_a \geq 0$ to a UTS orbit $\mc O'_a \sse \mf g_{s_a}^{\dual}$,
via
\begin{equation}
	\label{eq:moduli_number}
	\nu_a
	= \abs{\mc T_a},
	\qquad \mc T_a
	\ceqq \Set{ i \in \set{1,\dc,s} | L_{a,i} = T },
\end{equation}
looking at the fission sequence $L_{a,s_a} \sse \dm \sse L_{a,1} \sse L_{a,0} \ceqq G$ determined by any UTS principal part $\mc A'_a \in \mc O'_a$ (cf.~\eqref{eq:normal_form_reindexing}).
Then we show flatness in the following cases,
in increasing difficulty:
\begin{enumerate}
	\item
	      there is one pole $a \in \bm a$ such that $\nu_a \geq 3$ (cf.~\S~\ref{sec:flatness_nu_3});

	\item
	      there are two poles $a,a' \in \bm a$ such that $\nu_a \geq 2$ and $\nu_{a'} \geq 1$ (cf.~\S~\ref{sec:flatness_nu_2});

	\item
	      or there are three poles $a,a',a'' \in \bm a$ such that $\nu_a,\nu_{a'},\nu_{a''} \geq 1$ (cf.~\S~\ref{sec:flatness_nu_1}).
\end{enumerate}

\subsection{Three maximal tori at a marked point}
\label{sec:flatness_nu_3}

Together with Lem.~\ref{lem:base_change},
the case where $\nu_a \geq 3$ is a consequence of the following:

\begin{prop}
	\label{prop:flatness_high_order}

	Suppose that $s \geq 3$,
	and that $L_{s-2} = T$.
	Then the $G$-moment map
	\begin{equation}
		\label{eq:g_moment_map_single_orbit}
		\mu' = \mu'_{\mc O'} \cl \mc O' \lra \mf g^{\dual} \simeq \mf g,
		\qquad \mc A \lmt \Res(\mc A),
	\end{equation}
	is \emph{faithfully flat}.
\end{prop}

\begin{proof}
	We prove that~\eqref{eq:birkhoff_moment_map} is surjective,
	and omit the proof that it has equidimensional fibres:
	the conclusion follows from Lem.~\ref{lem:flatness_high_order}.

	Let $\Phi^\pm \sse \Phi$ be the subsystems of positive/negative roots determined by $\mf b^\pm$.
	For $i \in \set{1,\dc,s}$,
	denote also by $\phi_i \sse \Phi$ the Levi subsystem corresponding to $\mf l_i$,
	and let $\phi^\pm_i \ceqq \Phi^\pm \cap \phi_i$ be the corresponding systems of positive/negative roots in $\phi_i$.
	Then there are also:
	(i) disjoint unions
	\begin{equation}
		\phi_{i-1}
		= \nu_i^- \cup \phi_i \cup \nu^+_i,
		\qquad \nu^\pm_i
		\ceqq \phi^\pm_{i - 1} \sm \phi_i;
	\end{equation}
	and (ii) nested triangular decompositions
	\begin{equation}
		\mf l_{i - 1}
		= \mf n^-_i \ops \mf l_i \ops \mf n^+_i,
		\qquad \mf n^\pm_i
		\ceqq \bops_{\nu^\pm_i} \mf g_\alpha,\fn{
			The nilpotent subalgebras $\mf n_i \sse \mf l_{i-1}$ are \emph{not} the nilradicals $\mf u^\pm_i \sse \mf p^\pm_i$ of the parabolic subalgebras $\mf p^\pm_i \sse \mf g$ of~\S~\ref{sec:quantum_orbits}:
			the latter rather satisfy $\mf g = \mf u^-_i \ops \mf l_i \ops \mf u^+_i$ (cf.~Cor.-Def.~\ref{cor:cotangent_splitting_2}).
		}
	\end{equation}
	invoking the root lines $\mf g_\alpha \sse \mf g$---%
	and setting $\mf l_0 \ceqq \mf g$,
	cf.~\cite[Def.~7.2]{boalch_2014_geometry_and_braiding_of_stokes_data_fission_and_wild_character_varieties}.
	Now define
	\begin{equation}
		\label{eq:nested_nilradical}
		\check{\mf n}^\pm_i
		\ceqq \Biggl\{ \, X
		= \sum_{j = 1}^{s-i-1} X_j \varpi^j \,\, \Biggl| \,\, X_j \in \mf n^\pm_i \, \Biggr\} \sse \mf{bir}_{s-i},
		\qquad i \in \set{1,\dc,s-2},
	\end{equation}
	and let $T$ act as usual.
	By~\cite[Cor.~3.2]{yamakawa_2019_fundamental_two_forms_for_isomonodromic_deformations},
	there is a $T$-equivariant symplectomorphism
	\begin{equation}
		\label{eq:cotangent_splitting}
		\check{\mc O}' \lxra{\simeq} \on T^* \! \check{\mf n}^+,
		\qquad \check{\mf n}^+
		\ceqq \check{\mf n}^+_1 \ts \dm \ts \check{\mf n}^+_{s-2},
	\end{equation}
	and we compute a $T$-moment map for each factor of $\on T^* \! \check{\mf n}^+ \simeq \prod_{i = 1}^{s-2} \bigl( \on T^*\!\check{\mf n}^+_i \bigr)$.
	As in~\S~\ref{sec:tcla_duality},
	identify
	\begin{equation}
		\label{eq:dual_nested_nilradical}
		(\check{\mf n}^\pm_i)^{\dual} \simeq \Biggl\{ \, \dif Q \,\, \Biggl| \,\, Q
		= \sum_{j = 1}^{s-i-1} A_j \frac{\varpi^{-j}}{-j} \, ,
		\quad A_j \in \mf n^\mp_i \, \Biggr\}.
	\end{equation}
	Then,
	in the notation of~\eqref{eq:nested_nilradical} +~\eqref{eq:dual_nested_nilradical},
	the moment map on the $i$-th factor which vanishes at the origin reads
	\begin{equation}
		\check\mu'_i \cl \check{\mf n}^+_i \ts (\check{\mf n}^+_i)^{\dual}
		\simeq T^*\check{\mf n}^+_i \lra \mf t^{\dual} \simeq \mf t,
		\qquad (X,\dif Q)
		\lmt \sum_{j = 1}^{s-i-1} \sum_{\phi^+_i} \bigl[ X_j^{(\alpha)}, A_j^{(-\alpha)} \bigr],
	\end{equation}
	invoking the root-line components of $X_j = \sum_{\phi_i^+} X_j^{(\alpha)} \in \mf n^+_i$ and $A_j = \sum_{\phi^-_i} A_j^{(\alpha)} \in \mf n^-_i$.
	Now the surjectivity of the moment map for the diagonal action on all factors is due to the identity
	\begin{equation}
		\bops_{i = 1}^{s-2} \mf n^\pm_i = [\mf b^\pm,\mf b^\pm] = \bops_{\Phi^\pm} \mf g_\alpha,
	\end{equation}
	because,
	in turn,
	one has $\mf t = \bops_{\Phi^\pm} [\mf g_\alpha,\mf g_{-\alpha}]$.
\end{proof}

\begin{lemm}
	\label{lem:flatness_high_order}

	If~\eqref{eq:birkhoff_moment_map} is surjective,
	with equidimensional fibres,
	then the same holds for~\eqref{eq:g_moment_map_single_orbit}.
\end{lemm}

\begin{proof}[Proof postponed to~\ref{proof:lem_flatness_high_order}]
\end{proof}

\begin{lemm}
	\label{lem:base_change}

	Let $M'_1$ and $M'_2$ be two varieties,
	and $\varphi_i \cl M'_i \to V$ two algebraic maps to a vector space $V$.
	If $\varphi_1$ is faithfully flat,
	then the same holds for the sum
	\begin{equation}
		\label{eq:base_change}
		\varphi_1 + \varphi_2 \cl M'_1 \ts M'_2 \lra V,
		\qquad (x_1,x_2) \lmt \varphi_1(x_1) + \varphi_2(x_2).
	\end{equation}
\end{lemm}

\begin{proof}[Proof omitted]
\end{proof}

\subsection{Three maximal tori at two marked points}
\label{sec:flatness_nu_2}

To treat the second case,
we construct Darboux charts on coadjoint orbits for TCLAs (cf.~Thm.~\ref{thm:darboux_coordinates}),
leading to Cor.~\ref{cor:composite_moment_2}.

\subsubsection{}
\label{sec:darboux_coordinates}

Analogously to the proof of Prop.~\ref{prop:flatness_high_order},
denote by $\mf n^\pm_i \sse \mf l_{i-1}$ the nilradicals of the parabolic subalgebras of $\mf l_{i-1}$ containing $\mf b^\pm \cap \mf l_{i-1}$,
whence $\mf l_i \sse \mf l_{i-1}$ is their Levi factor containing $\mf t$.
Furthermore,
if $N^\pm_i \sse L_{i-1}$ are the unipotent radicals of the parabolic subgroups of $L_{i-1}$ containing $B^\pm \cap L_{i-1}$,
so that $\Lie(N_i^\pm) = \mf n^\pm_i$,
define also
\begin{equation}
	\label{eq:levi_tclgs}
	\wt L_i
	\ceqq L_i (\ms O_{s-i}) \sse G_{s-i},
	\qquad i \in \set{0,\dc,s},
\end{equation}
and
\begin{equation}
	\label{eq:nilradical_tclgs}
	\wt N^\pm_i
	\ceqq N^\pm_i (\ms O_{s-i+1}) \sse \wt L_{i-1},
	\qquad i \in \set{1,\dc,s},
\end{equation}
with Lie algebras $\wt{\mf l}_i \ceqq \Lie \bigl( \wt L_i \bigr)$ and $\wt{\mf n}^\pm_i \ceqq \Lie \bigl( \wt N^\pm_i \bigr)$.
Then:

\begin{theo}[cf.~\cite{yamakawa_2019_fundamental_two_forms_for_isomonodromic_deformations}, Thm.~3.1]
	\label{thm:darboux_coordinates}

	There are:
	\begin{enumerate}
		\item
		      a finite open cover $\mc O' = \bigcup_i V'_i$,
		      by $L_s$-invariant symplectic subvarieties;

		\item
		      and $L_s$-equivariant symplectomorphisms
		      \begin{equation}
			      \label{eq:darboux_chart}
			      V'_i
			      \lxra{\simeq} \on T^* \! \wt N^+,
			      \qquad \wt N^+
			      \ceqq \wt N^+_1 \ts \dm \ts \wt N^+_s.
		      \end{equation}
	\end{enumerate}
\end{theo}

\begin{proof}
	Decompose~\eqref{eq:normal_form_reindexing} as $\mc A' = \mc A'_{s-1} + \mc A'_{\on{top}}$,
	with
	\begin{equation}
		\label{eq:subleading_term}
		\mc A'_{s-1}
		\ceqq \sum_{j = 0}^{s-2} A'_j \varpi^{-j-1} \dif \varpi \in \wt{\mf l}_1^{\,\dual},
	\end{equation}
	in the notation of~\eqref{eq:leading_term},
	and under the usual dualities for TCLAs.
	Moreover,
	let $\mc O'_{s-1} \ceqq \wt L_1.\mc A'_{s-1}$ be the $\Ad^{\dual}_{\wt L_1}$-orbit of the `subleading' term of $\mc A'$.
	The statement follows inductively on the pole order,
	by Prop.~\ref{prop:recursion_for_darboux},
	as the images~\eqref{eq:recursive_embedding_image} cover $\mc O'$ upon replacing $\mc A'$ by its Weyl-translated.
\end{proof}

\begin{prop}
	\label{prop:recursion_for_darboux}

	Denote by $\Ad^{\flat}$ the coadjoint $\wt N^+_1$-action,
	and trivialize $\on T^*\!\wt N^+_1$ via left translations.
	Then:
	\begin{enumerate}
		\item
		      the map
		      \begin{equation}
			      \label{eq:recursive_embedding}
			      \iota \cl \bigl( \, \wt N^+_1 \ts (\wt{\mf n}^+_1)^{\dual} \bigr) \ts \mc O'_{s-1}
			      \simeq \on T^*\!\wt N^+_1 \ts \mc O'_{s-1} \lra \mf g_s^{\dual},
		      \end{equation}
		      defined by
		      \begin{equation}
			      \bigl( (\bm u^+,\bm Y),\mc A_{s-1} \bigr)
			      \lmt \Ad^{\dual}_{(\bm u^+)^{-1}} \bigl( \Ad^{\flat}_{\bm u^+}(\bm Y) + \mc A_{s-1} + \mc A'_{\on{top}} \bigr),
		      \end{equation}
		      yields an $L_1$-equivariant symplectic open embedding into the orbit $\mc O'$;

		\item
		      and the image of~\eqref{eq:recursive_embedding} is
		      \begin{equation}
			      \label{eq:recursive_embedding_image}
			      \mc O'(T)
			      \ceqq \Set{ \Ad^{\dual}_{\bm g^{-1}} (\mc A') | \bm g \in G_s \text{ and } \bm g(0) \in L_1 \cdot N^-_1 \cdot N^+_1 \sse G }.
		      \end{equation}
	\end{enumerate}
\end{prop}

\begin{proof}
	We first show that~\eqref{eq:recursive_embedding} takes values into $\mc O'$.
	Choose a triple
	\begin{equation}
		\bigl( (\bm u^+,\bm Y),\mc A_{s-1} \bigr) \in \bigl( \,\wt N^+_1 \ts (\wt{\mf n}^+_1)^{\dual} \bigr) \ts \mc O'_{s-1},
	\end{equation}
	as well as an element $\bm h \in \wt L_1$,
	such that $\mc A_{s-1} = \Ad^{\dual}_{\bm h^{-1}} \bigl( \mc A'_{s-1} \bigr)$.
	Denote also by the same symbol a lift of $\bm h$ in $\wt{\ul L}_1 \thra \wt L_1$,
	in the notation of~\eqref{eq:levi_tclgs_lift}.
	Then
	\begin{equation}
		\label{eq:adjoint_action_levi}
		\Ad^{\dual}_{\bm h^{-1}} \bigl( \mc A' \bigr)
		= \Ad^{\dual}_{\bm h^{-1}} \bigl( \mc A'_{s-1} \bigr) + \mc A'_{\on{top}}
		= \mc A_{s-1} + \mc A'_{\on{top}},
	\end{equation}
	as $L_1 = G^{A_{s-1}}$.
	Moreover,
	by Lem.~\ref{lem:algebraic_iso},
	there exists $\bm u^- \in \wt N^-_1$ such that
	\begin{equation}
		\label{eq:adjoint_action_nilradical}
		\Ad^\flat_{\bm u^+}(\bm Y) + \mc A_{s-1} + \mc A'_{\on{top}}
		= \Ad^{\dual}_{\bm u^-} \bigl( \mc A_{s-1} + \mc A'_{\on{top}} \bigr).
	\end{equation}
	Then~\eqref{eq:adjoint_action_levi}--\eqref{eq:adjoint_action_nilradical} yield
	\begin{equation}
		\Ad^{\dual}_{(\bm u^+)^{-1}} \bigl( \Ad^\flat_{\bm u^+}(\bm Y) + \mc A_{s-1} + \mc A'_{\on{top}} \bigr)
		= \Ad^{\dual}_{(\bm u^+)^{-1}\bm u^-} \bigl( \mc A_{s-1} + \mc A'_{\on{top}} \bigr)
		= \Ad^{\dual}_{\bm g^{-1}} \bigl( \mc A' \bigr),
	\end{equation}
	setting $\bm g \ceqq \bm h (\bm u^-)^{-1} \bm u^+ \in \wt L_1 \cdot \wt N^-_1 \cdot \wt N^+_1$.

	The resulting map is $L_1$-equivariant,
	and we show that it maps isomorphically onto~\eqref{eq:recursive_embedding_image}.
	In view of Lem.~\ref{lem:algebraic_iso},
	it is enough to prove that the map
	\begin{equation}
		\wt N^+_1 \ts \wt N^-_1 \ts \mc O'_{s-1} \lra \mc O'(T),
		\qquad \bigl( \bm u^+,\bm u^-,\mc A_{s-1} \bigr) \lmt \Ad^{\dual}_{(\bm u^+)^{-1} \bm u^-} \bigl( \mc A_{s-1} + \mc A'_{\on{top}} \bigr),
	\end{equation}
	is an isomorphism.
	To this end,
	given a group element $\bm g \in G_s$ such that
	\begin{equation}
		g \ceqq \bm g(0) = hu^-u^+ \in L_1 \cdot N^-_1 \cdot N^+_1,
	\end{equation}
	consider the factorization $\bm g = \bm h (\bm u^-)^{-1} \bm u^+$ provided by Lem.~\ref{lem:unique_factorization}.
	Then define
	\begin{equation}
		\mc A_{s-1} = \mc A_{s-1}(\bm g) \ceqq \Ad^{\dual}_{\bm h^{-1}} \bigl( \mc A' \bigr) - \mc A'_{\on{top}} \in \wt{\mf l}_1^{\, \dual}.
	\end{equation}
	This yields an algebraic map
	\begin{equation}
		\Set{ \bm g \in G_s | g \in L_1 \cdot N^-_1 \cdot N^+_1 } \lra \wt N^+_1 \ts \wt N^-_1 \ts \mc O'(T),
		\qquad \bm g \lmt \bigl( \bm u^+,\bm u^-,\mc A_{s-1} \bigr),
	\end{equation}
	which is invariant under the action of the $\Ad^{\dual}_{G_s}$-stabilizer of $\mc A'$ (as the latter is contained in $\wt{\ul L}_1$):
	it induces the desired inverse.
	Finally,
	the fact that~\eqref{eq:recursive_embedding} is a symplectic map follows from a computation quite similar to the proof of~\cite[Thm.~3.1]{yamakawa_2019_fundamental_two_forms_for_isomonodromic_deformations}.
	(We omit it.)
\end{proof}

\begin{lemm}
	\label{lem:unique_factorization}

	Choose $\bm g \in G_s$ such that $g \ceqq \bm g(0) \in L_1 \cdot N^-_1 \cdot N^+_1 \sse G$,
	and set
	\begin{equation}
		\label{eq:levi_tclgs_lift}
		\wt{\ul L}_1 \ceqq L_1 (\ms O_s)  \sse G_s.
	\end{equation}
	Then:
	\begin{enumerate}
		\item
		      there is a \emph{unique} factorization
		      \begin{equation}
			      \bm g
			      = \bm h \bm u^- \bm u^+,
			      \qquad \bm u^\pm \in \wt N^\pm_1,
			      \quad \bm h \in \wt{\ul L}_1;
		      \end{equation}

		\item
		      and the component maps $\bm g \mt \bm h$ and $\bm g \mt \bm u^\pm$ are $L_1$-equivariant \emph{algebraic} morphisms.
	\end{enumerate}
\end{lemm}

\begin{proof}{Proof postponed to~\ref{proof:lem_unique_factorization}.}
\end{proof}

\begin{lemm}
	\label{lem:algebraic_iso}

	Choose elements $\bm u^- \in \wt N^-_1$ and $\mc A_{s-1} \in \mc O'_{s-1}$.
	Then:
	\begin{enumerate}
		\item
		      the following vector lies in $(\wt{\mf n}^+_1)^{\dual}$:
		      \begin{equation}
			      \label{eq:nilpotent_element}
			      \bm Y'
			      \ceqq \Ad^{\dual}_{\bm u^-} \bigl( \mc A_{s-1} + \mc A'_{\on{top}} \bigr) - \bigl( \mc A_{s-1} + \mc A'_{\on{top}} \bigr);
		      \end{equation}

		\item
		      and the map
		      \begin{equation}
			      \label{eq:algebraic_iso}
			      \wt N^-_1 \ts \mc O'_{s-1} \lra (\wt{\mf n}^+_1)^{\dual} \ts \mc O'_{s-1},
			      \qquad \bigl( \bm u^-,\mc A_{s-1} \bigr) \lmt \bigl( \bm Y',\mc A_{s-1} \bigr),
		      \end{equation}
		      is an $L_1$-equivariant algebraic \emph{isomorphism}.
	\end{enumerate}
\end{lemm}

\begin{proof}{Proof postponed to~\ref{proof:lem_algebraic_iso}.}
\end{proof}

\begin{coro}
	\label{cor:composite_moment_2}

	Suppose that $L_s = T$.
	Then the composition
	\begin{equation}
		\pi_s \circ \mu' \cl \mc O' \lra \mf t^{\dual} \simeq \mf t,
		\qquad \mc A \lmt \pi_s \bigl( \Res(\mc A) \bigr),
	\end{equation}
	is \emph{faithfully flat}---%
	using the $G$-moment map~\eqref{eq:g_moment_map_single_orbit}.
\end{coro}

\begin{proof}
	In view of Thm.~\ref{thm:darboux_coordinates},
	it is enough to show that the $T$-moment map
	\begin{equation}
		\on T^*\!\wt N_1^+ \ts \dm \ts \on T^*\! \wt N_s \lra \mf t^{\dual},
	\end{equation}
	vanishing at the origin,
	is surjective,
	with equidimensional fibres.
	Up to replacing the $i$-th factor with $\on T^*\!\wt{\mf n}^+_i$,
	via the $T$-equivariant isomorphism provided by the exponential map $\wt{\mf n}^+_i \to \wt N^+_i$,
	the argument now follows very closely the proof of Prop.~\ref{prop:flatness_high_order}.
\end{proof}

\begin{rema}
	In particular,
	taking $s = 1$ in Cor.~\ref{cor:composite_moment_2} implies that a standard regular semisimple $\Ad_G^{\dual}$-orbit projects onto $\mf t^{\dual}$ in faithfully flat fashion:
	this trivializes the complexified version of Kostant's convexity theorem~\cite{kostant_1973_on_convexity_the_weyl_group_and_the_iwasawa_decomposition}.
\end{rema}

\begin{coro}
	\label{cor:flatness_nu_2}

	Choose two integers $s,\ul s \geq 1$.
	Let $\mc O' = G_s.\mc A'$ and $\ul{\mc O}' = G_{\ul s}.\ul{\mc A}$ be two orbits through UTS principal parts $\mc A'$ and $\ul{\mc A}'$ (cf.~\eqref{eq:normal_form_reindexing}),
	of pole orders $s$ and $\ul s$---%
	respectively.
	Moreover,
	suppose that $\nu(\mc O') \geq 2$ and that $\nu(\ul{\mc O}') \geq 1$ (cf.~\eqref{eq:moduli_number}).
	Then the corresponding $G$-moment map is \emph{faithfully flat}.
\end{coro}

\begin{proof}
	We prove that it is `miraculously' flat.
	Let $\check{\mc O}' \sse \mf{bir}_s^{\dual}$ be the Birkhoff orbit associated with the irregular part of $\mc A'$ (cf.~\eqref{eq:birkhoff_orbit}),
	with $T$-moment map~\eqref{eq:birkhoff_moment_map}.
	Set also $A'_0 \ceqq \on{Res}(\mc A') \in \mf t$.
	Then~\eqref{eq:decoupling} +~\eqref{eq:orbit_from_extended_orbit} yield an identification $(\bm \mu')^{-1}(X) \simeq Z \slash T$ for all $X \in \mf g \simeq \mf g^{\dual}$,
	where
	\begin{align}
		\label{eq:fibre_count_3}
		Z
		= Z_X
		\ceqq \left\{ \bigl(\!\right. (g,A_0),\dif Q,\wt{\ul{\mc A}}          & \bigr) \in \on T^*\!G \ts \check{\mc O}' \ts \ul{\mc O}'                                                                                               \\
		                                                                      & \left. \Big| \, A_0 + \Res \bigl( \wt{\ul{\mc A}} \bigr) = X, \quad \pi_s \bigl( \Ad^{\dual}_g(A_0) \bigr) - \check\mu'(\dif Q) = A'_0 \right\}        \\
		\simeq \left\{ \bigl( \!\right. g,\dif Q,\wt{\ul{\mc A}} \bigr) \in G & \ts \check{\mc O}' \ts \ul{\mc O}' \left. \Big|  \, \pi_s \bigl( X - \Res \bigl( \wt{\ul{\mc A}} \bigr) \bigr) = \check{\mu}'(\dif Q) + A'_0 \right\}.
	\end{align}
	Thus,
	each fibre of the projection $Z \to G \ts \check{\mc O}'$ is isomorphic to a fibre of the map
	\begin{equation}
		\ul{\mc O}' \lra \mf t^{\dual} \simeq \mf t,
		\qquad \wt{\ul{\mc A}}
		\lmt \pi_s \bigl( \Res(\wt{\ul{\mc A}}) \bigr).
	\end{equation}
	By Cor.~\ref{cor:composite_moment_2},
	the latter is surjective,
	with equidimensional fibres;
	hence the same holds for $Z \to G \ts \check{\mc O}'$,
	so that $Z \neq \vn$ has constant dimension as $X$ varies.
	The conclusion follows from the fact that the $T$-action on $\on T^*\! G \ts \check{\mc O}' \ts \ul{\mc O}'$ is \emph{free}.
\end{proof}

\subsection{Three maximal tori at three marked points}
\label{sec:flatness_nu_1}

To treat the last case,
we construct an \emph{unfolding map} (cf.~Thm.~\ref{thm:unfolding_map}),
leading to Cor.~\ref{cor:unfolding_in_action}.

We will first need a different version of the cotangent splitting~\eqref{eq:cotangent_splitting}:

\begin{enonce}{Corollary-Definition}
	\label{cor:cotangent_splitting_2}

	For $i \in \set{1,\dc,s}$ let $\mf u^\pm_i \ceqq \bops_{j = 1}^i \mf n^\pm_j \sse \mf g$ (cf.~\S~\ref{sec:quantum_orbits}.)
	Then there is an $L'$-equivariant symplectomorphism
	\begin{equation}
		\check{\mc O}' \lxra{\simeq} \on T^* \bigl( \mf u^+_1 \ts \dm \ts \mf u^+_{s-2} \bigr)
		\simeq \prod_{i = 1}^{s-2} \bigl( \on T^* \! \mf u^+_i \bigr).
	\end{equation}
\end{enonce}

\begin{proof}
	There is an isomorphisms of $L'$-modules
	\begin{equation}
		\check{\mf n}^\pm_i \simeq (\mf n^\pm_i)^{s-i-1},
		\qquad i \in \set{1,\dc,s-2}. \qedhere
	\end{equation}
\end{proof}

\begin{defi}[cf.~\cite{hiroe_2024_deformation_of_moduli_spaces_of_meromorphic_connections_on_p_1_via_unfolding_of_irregular_singularities}]
	Choose a tuple $\bm\varepsilon = (\varepsilon_0,\dc,\varepsilon_{s-1})$ of \emph{distinct} numbers $\varepsilon_i \in \mb C$.
	Then the $\bm\varepsilon$-\emph{unfolding of} $\mc A'$ is the following 1-form on the $z$-plane:
	\begin{equation}
		\label{eq:unfolding}
		\on{Unf} \bigl( \mc A' \bigr)
		= \on{Unf}_{\bm\varepsilon} \bigl( \mc A' \bigr)
		\ceqq \sum_{i = 0}^{s-1} \frac{ A'_i \dif z}{(z - \varepsilon_0)\dm(z - \varepsilon_i)}.
	\end{equation}
\end{defi}

\subsubsection{}

By construction~\eqref{eq:unfolding} has simple poles at $z = \varepsilon_i$,
for $i \in \set{0,\dc,s-1}$.
Let
\begin{equation}
	\label{eq:unfolded_residue}
	\wh\Lambda_i \ceqq
	\Ress_{z = \varepsilon_i} \bigl( \on{Unf}\bigl( \mc A' \bigr) \bigr)
	= \sum_{j = i}^{s-1} A'_j \cdot \prod_{\substack{l \in \set{0,\dc,j} \\ l \neq i }} (\varepsilon_i - \varepsilon_l)^{-1} \in \mf t.
\end{equation}
It follows that
\begin{equation}
	\label{eq:unfolded_residue_sum}
	\sum_{i = 0}^{s-1} \wh\Lambda_i
	= - \Ress_{z = \infty} \bigl( \on{Unf} \bigl( \mc A' \bigr) \bigr)
	= \Lambda'.
\end{equation}
Moreover,
one can recursively modify the ordered configuration $\bm\varepsilon \in \mb C^s$ so that $G^{\wh\Lambda_i} = L_{s-i} \sse G$ for $i \in \set{0,\dc,s-1}$:
we always tacitly assume that the latter holds.

Then,
using the exponential map $\mf u^\pm_i \to U^\pm_i \ceqq e^{\mf u^\pm_i} \sse G$ (cf.~Cor.-Def.~\ref{cor:cotangent_splitting_2}),
the tame version of Prop.~\ref{prop:recursion_for_darboux} yields $L'$-equivariant symplectic open immersions
\begin{equation}
	\on T^* \! \mf u^+_{s-i}
	\simeq \on T^* \! U^+_{s-i}
	= \on T^* \! U^+_{s-i} \ts \mc O'_{L'} \bigl( \wh\Lambda_i \bigr) \lhra \mc O'_G \bigl( \wh\Lambda_i \bigr),
	\qquad i \in \set{2,\dc,s-1},
\end{equation}
noting that the $\Ad_{L'}^{\dual}$-orbits of $\wh\Lambda_2,\dc,\wh\Lambda_{s-1}$ are points---%
by the choice of $\bm\varepsilon$.
Taking direct products,
by Cor.-Def.~\ref{cor:cotangent_splitting_2} there is then another such immersion
\begin{equation}
	\label{eq:quasi_unfolding_map}
	\check\iota \cl \check{\mc O}' \lhra \prod_{i = 2}^{s-1} \mc O'_G\bigl( \wh\Lambda_i \bigr).
\end{equation}

Now we upgrade~\eqref{eq:quasi_unfolding_map} to a $G$-equivariant immersion of $\mc O'$ into the product of \emph{all} the $\Ad_G^{\dual}$-orbits through the residues~\eqref{eq:unfolded_residue}.
To this end,
note that the following is an $L'$-moment map:
\begin{equation}
	\label{eq:unfolded_moment_map}
	\check\nu \cl \prod_{i = 2}^{s-1} \mc O'_G \bigl( \wh\Lambda_i \bigr) \lra (\mf l')^{\dual} \simeq \mf l',
	\qquad (A_2,\dc,A_{s-1}) \lmt \sum_{i = 2}^{s-1} \pi'(A_i).
\end{equation}
(We omit the proof of this fact.)
Then~\eqref{eq:unfolded_residue_sum} implies that
\begin{equation}
	\label{eq:moment_map_relation}
	\check\nu \circ \check\iota - \check{\mu}'
	= \sum_{i = 2}^{s-1} \wh\Lambda_i
	= \Lambda' - \wh\Lambda_0 - \wh\Lambda_1,
\end{equation}
in the notation of~\eqref{eq:birkhoff_moment_map} +~\eqref{eq:quasi_unfolding_map}--\eqref{eq:unfolded_moment_map}.
We will then construct an $L_s$-invariant map
\begin{equation}
	\on T^*G \ts \check{\mc O}' \simeq \wt{\mc O}' \supseteq (\wt\mu'_2)^{-1} (-\Lambda') \lxra{\wt\Psi} \mc O'_G \bigl( \wh{\bm\Lambda} \bigr) \ceqq \prod_{i = 0}^{s-1} \mc O'_G \bigl( \wh\Lambda_i \bigr),
\end{equation}
in the notation of~\eqref{eq:decoupling} +~\eqref{eq:second_component_moment_map}:
by~\eqref{eq:orbit_from_extended_orbit},
this will induce the desired $G$-equivariant map
\begin{equation}
	\label{eq:unfolding_map}
	\Psi \cl \mc O' \simeq \wt{\mc O}' \, \bs\!\!\bs\!\!\bs_{\!(- \Lambda')} \, L' \lra \mc O'_G \bigl( \wh{\bm\Lambda} \bigr).
\end{equation}

Choose thus a triple $\bigl( (g,\Lambda),\dif Q \bigr) \in (G \ts \mf g) \ts \check{\mc O}' \simeq \on T^*G \ts \check{\mc O}'$ such that
\begin{equation}
	\label{eq:moment_map_value}
	\wt\mu'_2 \bigl( (g,\Lambda),\dif Q \bigr)
	= \check\mu'(\dif Q) - \pi' \bigl( \Ad_g^{\dual}(\Lambda) \bigr)
	= -\Lambda'.
\end{equation}
Then $(A_2,\dc,A_{s-1}) \eqqc \check\iota(\dif Q) \in \prod_{i = 2}^{s-1} \mc O'_G \bigl( \wh\Lambda_i \bigr)$.
For the `lowest' factors,
instead,
consider
\begin{equation}
	\label{eq:modified_moment_value}
	S
	\ceqq \Ad_g^{\dual}(\Lambda) - \wh\Lambda_0 - \wh\Lambda_1 - \sum_{i = 2}^{s-1} A_i \in \mf g^{\dual} \simeq \mf g.
\end{equation}
Using~\eqref{eq:moment_map_relation} +~\eqref{eq:moment_map_value},
compute
\begin{align}
	\pi'(S)
	 & = \pi' \bigl( \Ad_g^{\dual}(\Lambda) \bigr) - \wh\Lambda_0 - \wh\Lambda_1 - \sum_{i = 2}^{s-1} \pi'(A_i)
	= \pi' \bigl( \Ad_g^{\dual}(\Lambda) \bigr) - \wh\Lambda_0 - \wh\Lambda_1 - \check\nu \bigl( \check\iota(\dif Q) \bigr) \\
	 & = \pi' \bigl( \Ad_g^{\dual}(\Lambda) \bigr) - \check\mu'(\dif Q) - \Lambda'
	= 0.
\end{align}
Therefore,
one can decompose~\eqref{eq:modified_moment_value} as $S = S^+ + S^-$,
with $S^\pm \in \mf u^\pm_{s-1}$.
If we now let
\begin{equation}
	\label{eq:modified_components}
	A_0
	\ceqq \wh\Lambda_0 + S^+,
	\qquad A_1 \ceqq \wh\Lambda_1 + S^-,
\end{equation}
then we finally get an $L_s$-invariant map on the level-set of $\wt\mu'_2$ via
\begin{equation}
	\label{eq:lifted_unfolding_map}
	\wt\Psi \cl \bigl( (g,\Lambda),\dif Q \bigr) \lmt \bigl( \Ad_{g^{-1}}^{\dual}(A_0),\Ad_{g^{-1}}^{\dual}(A_1),\dc,\Ad_{g^{-1}}^{\dual}(A_{s-1}) \bigr) \in \mc O'_G \bigl( \wh{\bm\Lambda} \bigr).
\end{equation}

\begin{theo}
	\label{thm:unfolding_map}

	The $G$-equivariant map~\eqref{eq:unfolding_map} is a symplectic open immersion.
\end{theo}

\begin{proof}
	Let $\bm\omega_G$ be the symplectic form of $\mc O'_G \bigl( \wh{\bm\Lambda} \bigr)$.
	Given $\dif Q \in \check{\mc O}'$,
	choose elements $u^\pm \in U^\pm_{s-1}$ such that $\Ad_{(u^+)}^{\dual}(A_i) = \wh\Lambda_i$ for $i \in \set{0,1}$ (cf.~\eqref{eq:modified_components});
	moreover,
	choose elements $b_i \in U^-_{s-i} \cdot U^+_{s-i}$ such that $\Ad_{b_i}^{\dual}(A_i) = \wh\Lambda_i$,
	for $i \in \set{2,\dc,s-1}$.
	Then $\Psi^*\bm\omega_G = \dif \theta$,
	where
	\begin{align}
		\theta
		 & \ceqq \bigl( \dif \, (u^+ g) \cdot (u^+g)^{-1} \bigm| \wh\Lambda_0 \bigr) + \bigl( \dif \, (u^-g) \cdot (u^-g)^{-1} \bigm| \wh \Lambda_1 \bigr) + \sum_{i = 2}^{s-1} \bigl( \dif \, (b_i g) \cdot (b_i g)^{-1} \bigm| \wh\Lambda_i \bigr) \\
		 & = \sum_{i = 0}^{s-1} \bigl( \dif g \cdot g^{-1} \mid A_i \bigr) + \sum_{i = 2}^{s-1} \bigl( \dif b_i \cdot b_i^{-1} \bigm| \wh\Lambda_i \bigr),
	\end{align}
	using the $\Ad_G$-invariant pairing $( \cdot \mid \cdot) \cl \mf g \ots \mf g \to \mb C$,
	and noting that
	\begin{equation}
		\bigl( \dif u^+ \cdot (u^+)^{-1} \bigm| \wh\Lambda_0 \bigr)
		= 0
		= \bigl( \dif u^- \cdot (u^-)^{-1} \bigm| \wh\Lambda_1 \bigr).
	\end{equation}
	Now observe that $\theta' \ceqq \sum_{i = 2}^{s-1} \bigl( \dif b_i \cdot b_i^{-1} \bigm| \wh\Lambda_i \bigr)$ is the pullback of the symplectic form along the embedding~\eqref{eq:quasi_unfolding_map};
	hence,
	it coincides with the symplectic form $\check\omega$ on the Birkhoff orbit.
	Overall
	\begin{equation}
		\label{eq:pullback_along_unfolding}
		\Psi^*\bm\omega_G
		= \theta' + \sum_{i = 0}^{s-1} \dif \, \bigl( \dif g \cdot g^{-1} \mid A_i \bigr)
		= \check\omega + \sum_{i = 0}^{s-1} \dif \, \bigl( \dif g \cdot g^{-1} \mid A_i \bigr).
	\end{equation}
	On the other hand,
	by construction
	\begin{equation}
		A_0 + A_1
		= \wh\Lambda_0 + \wh\Lambda_1 + S
		= \Ad_g^{\dual}(\Lambda) - \sum_{i = 2}^{s-1} A_i,
	\end{equation}
	whence $\sum_{i = 0}^{s-1} A_i = \Ad_g^{\dual}(\Lambda)$,
	and in turn
	\begin{equation}
		\sum_{i = 0}^{s-1} \dif \, \bigl( \dif g \cdot g^{-1} \mid A_i \bigr)
		= \dif \, \bigl( \dif g \cdot g^{-1} \mid \Lambda \bigr),
	\end{equation}
	which is the symplectic form of $\on T^*\!G$.
	In view of~\eqref{eq:decoupling},
	the equality~\eqref{eq:pullback_along_unfolding} proves that~\eqref{eq:unfolding_map} intertwines the symplectic structures.

	We conclude by showing that the fibres of the map~\eqref{eq:lifted_unfolding_map} are precisely the $L_s$-orbits through $(\wt\mu'_2)^{-1} \bigl( -\Lambda' \bigr) \sse \wt{\mc O}'$,
	so that $\Psi$ is injective;
	the fact that it is an open immersion follows,
	e.g.,
	from~\cite[Thm.~4.2, p.~187]{parshin_shafarevich_1994_algebraic_geometry_iv}.
	Suppose thus that
	\begin{equation}
		\wt\Psi \bigl( (g_1,\Lambda_1),\dif Q_1 \bigr) = \wt\Psi \bigl( (g_2,\Lambda_2),\dif Q_2 \bigr) \in \mc O'_G \bigl( \wh{\bm\Lambda} \bigr),
	\end{equation}
	for suitable elements $g_,g_2 \in G$,
	$\Lambda_1,\Lambda_2 \in \mf g$,
	and $\dif Q_1,\dif Q_2 \in \check{\mc O}'$.
	Comparing the components in $\mc O'_G \bigl( \wh\Lambda_0 \bigr) \ts \mc O'_G \bigl( \wh\Lambda_1 \bigr)$ shows that the group element $h \ceqq g_2 \cdot g_1^{-1}$ lies in $L' \cdot U^+_{s-1} \cap L' \cdot U^-_{s-1} = L'$,
	and that $\Ad_h(\wh\Lambda_0) = \wh\Lambda_0$.
	Thus $h \in (L')^{\Lambda'} = L_s$.
\end{proof}

\begin{coro}
	\label{cor:open_image_unfolding}

	Suppose that $A'_{s-1} \in \mf t_{\reg}$.
	Then the image of the unfolding map is
	\begin{align}
		\Psi(\mc O')
		= \left\{ \bigl( \Ad_{g_0^{-1}}^{\dual}\bigl( \wh\Lambda_0 \bigr),\right.\dc,\Ad_{g_{s-1}^{-1}}^{\dual} & \bigl( \wh\Lambda_{s-1} \bigr) \bigr) \in \mc O'_G \bigl( \wh{\bm\Lambda} \bigr)                 \\
		                                                                                                        & \Big| \left. g_ig_0^{-1} \in T \cdot U^- \cdot U^+  \text{ for } i \in \set{1,\dc,s-1} \right\}.
	\end{align}
\end{coro}

\begin{proof}[Proof omitted]
\end{proof}

\begin{coro}
	\label{cor:unfolding_in_action}

	Choose an integer $s \geq 1$,
	and let $\mc O'_{G,0},\dc,\mc O'_{G,s-1} \sse \mf g^{\dual}$ be \emph{regular} semisimple $\Ad_G^{\dual}$-orbits.
	Then there exist:
	\begin{enumerate}
		\item
		      a finite open cover $\prod_{j = 0}^{s-1} \mc O'_{G,j} = \bigcup_{i \in I} V'_{G,i}$,
		      by $G$-invariant symplectic subvarieties;

		\item
		      and $G$-equivariant symplectomorphisms
		      \begin{equation}
			      V'_{G,i} \lxra{\simeq} \mc O'_i \ceqq G_s.\mc A'_i \sse \mf g_s^{\dual},
			      \qquad i \in I,
		      \end{equation}
		      for suitable UTS principal parts $\mc A'_i \in \mf g'_s$.
	\end{enumerate}
\end{coro}

\begin{proof}
	Choose markings $\Lambda'_j \in \mc O'_{G,j} \cap \mf t^{\dual}$,
	for $j \in \set{0,\dc,s-1}$,
	as well as distinct numbers $\varepsilon_0,\dc,\varepsilon_{s-1} \in \mb C$.
	Then,
	given an $s$-tuple $\bm w = (w_0,\dc,w_{s-1}) \in W^s$,
	define a UTS principal part
	\begin{equation}
		\mc A'_{\bm w}
		= \sum_{j = 0}^{s-1} A_{\bm w,j} \varpi^{-j-1} \dif \varpi,
		\qquad A_{\bm w,j} \in \mf t,
	\end{equation}
	by the equality $\wh\Lambda_j = w_j(\Lambda'_j) \in \mf t$,
	for $j \in \set{0,\dc,s-1}$,
	where $\wh\Lambda_j \in \mf t$ is the residue associated to $\on{Unf} \bigl( \mc A'_{\bm w} \bigr)$ (cf.~\eqref{eq:unfolded_residue}).
	It follows that the leading coefficient $A'_{\bm w,s-1}$ is \emph{regular} semisimple.
	By Cor.~\ref{cor:open_image_unfolding},
	the image $V_{\bm w} \ceqq \Psi \bigl( \mc O'_{\bm w} \bigr)$ of the corresponding unfolding map
	\begin{equation}
		\Psi \cl \mc O'_{\bm w} \lhra \prod_{j = 0}^{s-1} \mc O'_{G,j},
		\qquad \mc O'_{\bm w}
		\ceqq G_s.\mc A'_{\bm w},
	\end{equation}
	consists of tuples of the form
	\begin{equation}
		\bigl( \Ad_{g_0^{-1}}^{\dual} \bigl( w_0(\Lambda'_0) \bigr),\dc,\Ad_{g_{s-1}^{-1}}^{\dual} \bigl( w_{s-1}(\Lambda'_{s-1}) \bigr) \bigr),
		\qquad g_1g_0^{-1},\dc,g_{s-1}g_0^{-1} \in T \cdot U^- \cdot U^+.
	\end{equation}
	One concludes by the cover $G = \bigcup_W \dot w (T \cdot U^- \cdot U^+)$,
	where $\dot w \in N_G(T)$ is a lift of $w \in W$.
\end{proof}

\begin{coro}
	\label{cor:flatness_nu_1}

	Choose integers $s_1,s_2,s_3 \geq 1$.
	For $i \in \set{1,2,3}$,
	let $\mc O'_i \ceqq G_{s_i}.\mc A'_i$ be an orbit through a UTS principal part $\mc A'_i$ (cf.~\eqref{eq:normal_form_reindexing}),
	of pole order $s_i$.
	Moreover,
	suppose that $\nu(\mc O'_i) \geq 1$ (cf.~\eqref{eq:moduli_number}).
	Then the corresponding $G$-moment map is \emph{faithfully flat}.
\end{coro}

\begin{proof}
	If $s_1 = s_2 = s_3 = 1$,
	the conclusion follows from Prop.~\ref{prop:flatness_high_order} +~Cor.~\ref{cor:unfolding_in_action},
	as the orbit product is covered by \emph{generic} UTS orbits of pole order $3$.
	In general,
	given $i \in \set{1,2,3}$ consider the corresponding unfolding map~\eqref{eq:unfolding_map},
	viz.,
	\begin{equation}
		\Psi_i \cl \mc O'_i \lhra \prod_{j = 0}^{s_i-1} \mc O'_G \bigl( \wh\Lambda_j^{(i)} \bigr),
	\end{equation}
	in the notation of~\eqref{eq:unfolded_residue}.
	Then the product map $\bm\Psi \ceqq (\Psi_1,\Psi_2,\Psi_3)$ is a $G$-equivariant symplectic open immersion,
	and it intertwines the corresponding $G$-moment maps by~\eqref{eq:unfolded_residue_sum}.
	Flatness follows from the tame case,
	together with Lem.~\ref{lem:base_change},
	as $\wh\Lambda^{(i)}_0 \in \mf t$ is \emph{regular} semisimple.

	To show surjectivity choose an element $\Lambda \in \mf g \simeq \mf g^{\dual}$.
	Consider the canonical projection $\pi_{\mf t} \cl \mf g \thra \mf t$ along $\mf u^- \ops \mf u^+ \sse \mf g$.
	In view of Cor.~\ref{cor:composite_moment_2},
	there exists a principal part $\mc A_3 \in \mc O'_3$ such that
	\begin{equation}
		\pi_{\mf t} (\mc A_3)
		= \pi_{\mf t}(\Lambda) - A^{(1)}_0 - A^{(2)}_0 \in \mf t,
	\end{equation}
	writing
	\begin{equation}
		\mc A'_i
		= \sum_{j = 0}^{s_i-1} A_j^{(i)} \varpi^{-j-1} \dif \varpi,
		\qquad A_j^{(i)} \in \mf t,
		\quad i \in \set{1,2,3}.
	\end{equation}
	We now construct two principal parts $\mc A_1 \in \mc O'_1$ and $\mc A_2 \in \mc O'_2$ such that $\bm\mu'(\mc A_1,\mc A_2,\mc A_3) = \Lambda$.
	Let us start from $i = 2$.
	Denote by $T = L_{s_2}^{(2)} \sse \dm \sse \dm L_1^{(2)} \sse G$ the fission sequence determined by $\mc A'_2$,
	and introduce the usual Lie subalgebras $\mf n^\pm_j = (\mf n^\pm_j)^{(2)} \sse \mf l_{j-1}^{(2)}$.
	It follows that $\bops_{j = 1}^{s_2} \mf n^\pm_j = \mf u^\pm \sse \mf b^\pm$ (cf.~Cor.-Def.~\ref{cor:cotangent_splitting_2}).
	Let then $\pi_{\mf u^\pm} \cl \mf g \thra \mf u^\pm$ be the canonical projections along $\mf t \ops \mf u^\mp \sse \mf g$,
	and set in addition $Y^{(2)} \ceqq \pi_{\mf u^+} \bigl( \Lambda - \Res(\mc A_3) \bigr)$.
	Decompose the latter as
	\begin{equation}
		Y^{(2)} = Y^{(2)}_1 + \dm + Y^{(2)}_{s_2},
		\qquad Y^{(2)}_j \in \mf n^+_j,
	\end{equation}
	and choose $X_j \in \mf n^+_j$ such that $\bigl[ A_{s-j}^{(2)},X_j \bigr] = Y^{(2)}_j$.
	Since $A_0^{(2)} \in \mf l^{(2)}_{s_2-1}$ is regular semisimple,
	there exists $u^+ \in \exp(\mf n^+_{s_2})$ such that $\Ad_{u^+} \bigl( A_0^{(2)} + Y^{(2)}_{s_2} \bigr) = A_0^{(2)}$.
	Then consider the element
	\begin{equation}
		\bm g
		\ceqq u^+ \cdot e^{X_{s_2-1} \varpi_{a_2}} \dm e^{X_1 \varpi^{s_2-1}_{a_2}} \in G_{s_2},
	\end{equation}
	and finally let $\mc A_2 \ceqq \Ad_{\bm g^{-1}}^{\dual} \bigl( \mc A'_2 \bigr) \in \mc O'_2$.
	By construction $X_j \in \mf l^{(2)}_{j-1}$ commutes with $A_{s_2-j+1}^{(2)},\dc,A_{s_2-1}^{(2)}$,
	and $\Ad_{u^+}$ stabilizes $A_1^{(2)},\dc,A_{s_2-1}^{(2)}$;
	thus:
	\begin{equation}
		\mc A_2
		= \mc A'_2 + Y^{(2)}_{s_2} - \sum_{j = 1}^{s_2-1} \bigl[ X_j,A_{s_2-j}^{(2)} \bigr] \frac{\dif \varpi}\varpi
		= \mc A'_2 + Y^{(2)} \frac{\dif \varpi}\varpi.
	\end{equation}
	Analogously,
	one can construct an element $\mc A_1 \in \mc O'_1$ of the form
	\begin{equation}
		\mc A_1
		= \mc A'_1 + Y^{(1)} \frac{\dif \varpi}\varpi,
		\qquad Y^{(1)} \ceqq \pi_{\mf u^-} \bigl( \Lambda - \Res(\mc A_3) \bigr).
	\end{equation}
	Finally,
	compute
	\begin{equation}
		\Res(\mc A_1) + \Res(\mc A_2) + \Res(\mc A_3)
		= A_0^{(1)} + Y^{(1)} + A_0^{(2)} + Y^{(2)} + \Res(\mc A_3)
		= \Lambda. \qedhere
	\end{equation}
\end{proof}

\section{About stability}
\label{sec:stability}

Let again $G$ be \emph{reductive}.
In this section we provide a sufficient condition on the choice of truncated-current orbits $\bm{\mc O}' = \set{\mc O'_a}_{\bm a}$ so that the level-zero set of the moment map~\eqref{eq:moment_map} only contains \emph{stable} points---%
upon modding out the centre.

\subsection{}

Fix a maximal torus $T \sse G$ and a Borel subgroup $B \sse G$ containing $T$.
Choose also UTS principal parts $\mc A'_a$ as in~\eqref{eq:normal_form_reindexing},
with pole orders $s_a \geq 1$ for $a \in \bm a$.
Let $\mc O'_{\bm a} \sse \prod_{\bm a} \mf g_{s_a}^{\dual}$ denote the corresponding orbit product (cf.~\eqref{eq:orbit_product}),
and cut out the affine subvariety $M = (\bm\mu')^{-1}(0) \sse \mc O'_{\bm a}$ (cf.~\eqref{eq:fin_dim_description_de_rham}).
The diagonal $G$-action factors through $\bm P(G) = G \slash Z(G)$,
and the following criterion is a $G$-bundle generalization of~\cite{kostov_1999_on_the_deligne_simpson_problem,kostov_2004_on_the_deligne_simpson_problem_and_its_weak_version}:

\begin{prop}
	\label{prop:stability}

	Choose:
	(i) a proper parabolic subgroup $P \ssne G$ containing $B$,
	with Lie algebra $\mf p \ceqq \Lie(P)$;
	and (ii) an element $w_a$ of the Weyl group $W$,
	for $a \in \bm a$.
	Moreover,
	suppose that for all such choices there exists a character $\chi \cl \mf p \to \mb C$ such that
	\begin{equation}
		\label{eq:stability_condition}
		\sum_{\bm a} \chi \bigl( \wt \Lambda'_a \bigr) \neq 0,
		\qquad \wt\Lambda'_a
		\ceqq w_a (\Lambda'_a) \in \mf t.
	\end{equation}
	Then all the points of $M$ are $\bm P(G)$-\emph{stable}.
\end{prop}

\begin{proof}
	Suppose that there is a \emph{nonstable} point $\bm{\mc A} = (\mc A_a)_{\bm a} \in M$.
	By the Hilbert--Mumford criterion~\cite{hilbert_1893_ueber_die_vollen_invariantensysteme},
	there exists a 1-parameter subgroup $\lambda \cl \mb C^{\ts} \to G$,
	with $\lambda(\mb C^{\ts}) \nsse Z(G)$,
	such that $\lim_{t \to 0} \bigl( \lambda(t).\bm{\mc A} \bigr)$ exists in $\prod_{\bm a} \mf g_{s_a}^{\dual}$.
	This implies that the coefficients of $\mc A_a$ lie in the Lie algebra $\mf p \ceqq \Lie(P)$ of the proper parabolic subgroup $P = P_{\lambda} \ssne G$ determined by $\lambda$.
	Up to acting by $G$ on $\bm{\mc A}$,
	one can also assume that $B \sse P$.

	Now for $a \in \bm a$ choose an element $g_a \in G$ such that $\Ad_{g_a}^{\dual} (\mc A_a) \in \mf g_{s_a}^{\dual}$ lies in the $\Ad_{\on{Bir}_{s_a}}^{\dual}$-orbit of $\mc A'_a$.
	Lem.~\ref{lem:closed_subgroup} (for the subgroup $H \ceqq g_a P g_a^{-1} \sse G$) implies that the coefficients of the latter lie in $\Ad_{g_a}(\mf p) \sse \mf g$,
	and that there exists $\bm h_a \in g_a P_{s_a} g_a^{-1} \sse G_{s_a}$ such that:
	(i) $\bm h_a \in \on{Bir}_{s_a}$;
	and (ii) $\Ad_{\bm h_ag_a}^{\dual} \bigl( \mc A_a \bigr) = \mc A'_a$.
	Write now
	\begin{equation}
		g_a^{-1}\bm h_ag_a
		= e^{-\bm X},
		\qquad \bm X
		= \sum_{j = 1}^{s_a-1} X_j\varpi_a^j,
		\quad X_j \in \mf p.
	\end{equation}
	Then the equality $\Ad_{g_a^{-1}\bm h_ag_a}^{\dual} \bigl( \mc A_a \bigr) = \Ad_{g_a^{-1}}^{\dual} \bigl( \mc A'_a \bigr) \in \mf g_{s_a}^{\dual}$ yields in particular
	\begin{equation}
		\Res \bigl( \mc A_a \bigr)
		= \Ad_{g_a^{-1}} (\Lambda'_a) + \sum_{j = 1}^{s_a-1} \sum_{n > 0} \frac 1{n!} X_{j,n},
	\end{equation}
	where
	\begin{equation}
		X_{j,n}
		\ceqq \sum_{\substack{i_1,\dc,i_n > 0 \\ i_1 + \dm + i_n = s_a - j}} \ad_{X_{i_1}} \dm \ad_{X_{i_n}} \bigl( \Ad_{g_a^{-1}} ( A'_{a,s_a-j} ) \bigr) \in [\mf p,\mf p].
	\end{equation}

	Choose now a character $\chi \in \Hom_{\Lie}(\mf p,\mb C) \simeq \bigl( \mf p \bs [\mf p,\mf p] \bigr)^{\!\dual}$.
	It follows that
	\begin{equation}
		\label{eq:vanishing_character}
		0
		= \chi(0)
		= \chi \Biggl( \sum_{\bm a} \Res \bigl( \mc A_a \bigr) \Biggr)
		= \sum_{\bm a} \chi (\ul\Lambda'_a),
		\qquad \ul\Lambda'_a
		\ceqq \Ad_{g_a^{-1}} (\Lambda'_a).
	\end{equation}
	But $\ul{\Lambda}'_a$ is a semisimple element of $\mf p$,
	lying in the $\Ad_P$-orbit of an element $\wt{\Lambda}'_a \in \mf t$.
	On the other hand,
	by construction,
	$\ul{\Lambda}'_a$ lies in the $\Ad_G$-orbit of $\Lambda'_a \in \mf t$,
	so that $\wt\Lambda'_a = w_a(\Lambda'_a)$ for a suitable group element $w_a \in W$:
	the $\Ad_P$-invariance of $\chi$ now yields $\sum_{\bm a} \chi \bigl( \wt\Lambda'_a \bigr) = 0$.
\end{proof}

\begin{lemm}
	\label{lem:closed_subgroup}

	Choose a UTS principal part $\mc A' \in \mf g'_s$ as in~\eqref{eq:normal_form_reindexing}.
	Consider also a connected closed subgroup $H \sse G$,
	with Lie algebra $\mf h \ceqq \Lie(H)$,
	and let $\mc A = \sum_{j = 0}^{s-1} A_j \varpi^{-j-1} \dif \varpi \in \mf g_s^{\dual}$ be an element of the $\Ad_{\on{Bir}_s}^{\dual}$-orbit of $\mc A'$.
	If $A_0,\dc,A_{s-1} \in \mf h$,
	then:
	\begin{enumerate}
		\item
		      the same holds for the coefficients of $\mc A'$;

		\item
		      and there exists a group element $\bm h \in H_s \ceqq H(\ms O_s) \sse G_s$ such that:
		      \begin{enumerate}
			      \item
			            $\bm h \in \on{Bir}_s$;

			      \item
			            and $\Ad_{\bm h}^{\dual} \bigl( \mc A \bigr) = \mc A' \in \mf g_s^{\dual}$.
		      \end{enumerate}
	\end{enumerate}
\end{lemm}

\begin{proof}[Proof postponed to~\ref{proof:lem_closed_subgroup}]
\end{proof}

\section*{Acknowledgements}

We thank (in alphabetical order):
A.~Alekseev,
F.~Bischoff,
D.~Calaque,
P.~Dimakis,
M.~Gualtieri,
D.~Lombardo,
M.~Mazzocco,
F.~Naef,
V.~Roubtsov,
V.~Toledano Laredo,
and A.~Vanhaecke;
for listening/answering to some questions.
We separately thank L.~Topley and R.~Wentworth for many helpful discussions throughout the inception/execution of this project.

\appendix

\section{Automorphisms of TCLAs}
\label{sec:tcla_automorphisms}

Here we describe the automorphism group of a TCLA
\begin{equation}
	\mf g_s =
	\mf g (\ms O_s)
	\simeq \mf g \llb \varpi \rrb \bs \varpi^s \mf g \llb \varpi \rrb,
\end{equation}
for any truncation $s \geq 1$.
(Up to isomorphism,
the choice of a uniformizer is w.l.o.g.)

\subsection{Reduction to the simple case}

The crux of the matter is the case where $\mf g$ is \emph{simple} (but cf.~Rmk.~\ref{rmk:simple_assumption}).
Indeed,
if $\mf g$ decomposes into simple ideals $\mf I_i \sse \mf g$,
then $\Aut_{\on{Lie}}(\mf g_s)$ is a split extension of the group of permutations of isomorphic simple summands,
by the direct product of the automorphism groups of the TCLAs $\mf I_i(\ms O_s)$.

\subsection{Filtrations}

Hereafter,
for an integer $i \in \set{0,\dc,s-1}$ write $\mf g_s^{(i)} \ceqq \mf g \ots \varpi^i \sse \mf g$,
so that there is a decreasing Lie-algebra filtration of $\mf g_s$,
with filtered pieces
\begin{equation}
	\label{eq:tcla_filtration}
	\mf g_s^{(\geq i)} \ceqq \bops_{j = i}^{s-1} \mf g_s^{(j)},
	\qquad i \in \set{0,\dc,s-1}.
\end{equation}

Then note that:

\begin{lemm}
	\label{lem:nice_automorphisms}

	The Lie-algebra automorphisms of $\mf g_s$ preserve the filtration~\eqref{eq:tcla_filtration}.
\end{lemm}

\begin{proof}
	By induction on $i \in \set{1,\dc,s-1}$.
	The base follows from the fact that $\mf g_s^{(\geq 1)} = \mf{bir}_s$ is the nilradical,
	and the inductive step from the equality $\mf g = [\mf g,\mf g]$.
\end{proof}

\subsection{Semidirect factors}

We now define three subgroups of $\Aut_{\on{Lie}}(\mf g_s)$,
and prove that they generate it.
They consist of:
(i) automorphisms of $\mf g$,
acting in $\varpi$-graded fashion on $\mf g_s$;
(ii) inner automorphisms of $\mf g_s$,
coming from the unipotent radical $\on{Bir}_s = \on{Rad}_u(G_s)$ of the corresponding TCLG;
and (iii) ring automorphisms of $\ms O_s$,
acting `polynomially' on $\mf g_s$.
Precisely,
let:
\begin{equation}
	\label{eq:diag_auto_subgroup}
	\Aut_{\diag}(\mf g_s)
	\ceqq \Set{ \varphi \ots \Id_{\ms O_s} | \varphi \in \Aut_{\on{Lie}}(\mf g) },
\end{equation}
and
\begin{equation}
	\label{eq:conj_auto_subgroup}
	\Aut_{\on{conj}}(\mf g_s)
	\ceqq \Set{ \Ad_{\bm g} | \bm g \in \on{Bir}_s },
\end{equation}
and
\begin{equation}
	\label{eq:poly_auto_subgroup}
	\Aut_{\on{poly}}(\mf g_s)
	\ceqq \Set{ \Id_{\mf g} \ots F | F \in \ms O_s^{\ts} }.
\end{equation}

\begin{rema}
	Identify $\on{Bir}_s = \exp(\mf{bir}_s) \sse G_s$ with its Lie algebra.
	Then~\eqref{eq:diag_auto_subgroup} acts on~\eqref{eq:conj_auto_subgroup},
	and there is a group embedding $G_s \hra \Aut_{\diag}(\mf g_s) \lts \Aut_{\on{conj}}(\mf g_s)$.
\end{rema}

\begin{theo}
	\label{thm:tcla_automorphisms}

	There is a nested semidirect group factorization
	\begin{equation}
		\label{eq:semidirect_auto_tcla}
		\Aut_{\on{Lie}}(\mf g_s)
		= \bigl( \Aut_{\diag}(\mf g_s) \lts \Aut_{\on{conj}}(\mf g_s) \bigr) \rts \Aut_{\on{poly}}(\mf g_s).
	\end{equation}
\end{theo}

\begin{proof}
	For $i \in \set{0,\dc,s-1}$ let $p_i \cl \mf g_s \thra \mf g_s^{(i)} \simeq \mf g$ be the canonical projection.
	Given an automorphism $\Phi \in \Aut_{\on{Lie}}(\mf g_s)$,
	define
	\begin{equation}
		\Phi_{i,j} \ceqq p_i \circ \eval[1]\Phi_{\mf g_s^{(j)}},
		\qquad i,j \in \set{0,\dc,s-1},
	\end{equation}
	and regard it as a $\mb C$-linear map $\mf g \to \mf g$.
	(Lem.~\ref{lem:nice_automorphisms} implies that $\Phi_{i,j} = 0$ if $i < j$.)

	Now,
	by hypothesis:
	\begin{equation}
		\Phi \bigl( [X, Y]\varpi^{i+j} \bigr)
		= \bigl[ \Phi(X \varpi^i), \Phi(Y \varpi^j) \bigr],
		\qquad X,Y \in \mf g,
	\end{equation}
	which is only relevant if $i+j < s$.
	Applying $p_k$ to both sides yields
	\begin{equation}
		\label{eq:component_map_relations}
		\Phi_{k,i+j} \bigl( [X,Y] \bigr)
		= \sum_{k_1+k_2 = k} \bigl[ \Phi_{k_1,i}(X),\Phi_{k_2,j}(Y) \bigr],
		\qquad k \in \set{0,\dc,s-1}.
	\end{equation}

	Now taking $i = j = k \ceqq 0$ in~\eqref{eq:component_map_relations} shows that $\Phi_{0,0}$ is a Lie-algebra endomorphism of $\mf g$;
	and it must be an automorphism (by Lem.~\ref{lem:nice_automorphisms},
	else $\Phi$ would \emph{not} be surjective).
	To conclude,
	since $\Phi_{0,0} \otimes \Id_{\ms O_s} \in \Aut_{\diag}(\mf g_s)$,
	it is enough to prove that if $\Phi_{0,0} = \Id_{\mf g}$ then $\Phi$ lies in the subgroup generated by $\Aut_{\on{conj}}(\mf g_s)$ and $\Aut_{\on{poly}}(\mf g_s)$.

	To this end,
	fix $k \in \set{1,\dc,s-1}$,
	and suppose in addition that $\Phi_{0,0} = \Id_\mf g$ and $\Phi_{l,0} = 0$ for all $l \in \set{1,\dc,k-1}$.
	Then,
	taking $i = j \ceqq 0$ in ~\eqref{eq:component_map_relations} yields
	\begin{equation}
		\Phi_{k,0} \bigl( [X,Y] \bigr)
		=  \bigl[ \Phi_{k,0}(X),Y \bigr] + \bigl[ X,\Phi_{k,0}(Y) \bigr],
	\end{equation}
	i.e.,
	$\Phi_{k,0}$ is a derivation of $\mf g$.
	But here all derivations are \emph{inner},
	and so there exists $Z \in \mf g$ such that $\Phi_{k,0} = \ad_Z \in \mf{der}(\mf g)$.
	Letting $\bm g \ceqq e^{Z\varpi^k} \in \on{Bir}_s$,
	the automorphism $\Phi' \ceqq \Ad_{\bm g^{-1}} \circ \Phi$ of $\mf g_s$ satisfies $\Phi'_{0,0} = \Id_\mf g$ and $\Phi'_{1,0} = \dm = \Phi'_{k,0} = 0$.

	Since $\Ad_{\bm g} \in \Aut_{\on{conj}}(\mf g_s)$,
	it suffices to show the following:
	if $\Phi_{0,0} = \Id_\mf g$,
	and $\Phi_{k,0} = 0$ for $k > 0$,
	then $\Phi \in \Aut_{\on{poly}}(\mf g_s) \sse \Aut_{\on{Lie}}(\mf g_s)$.
	With such hypotheses,
	taking $(i,j) \ceqq (0,1)$ in~\eqref{eq:component_map_relations} yields
	\begin{equation}
		\Phi_{k,1} \bigl( [X,Y] \bigr)
		= \bigl[ X,\Phi_{k,1}(Y) \bigr],
	\end{equation}
	i.e.,
	$\Phi_{k,1}$ is a $\mf g$-module endomorphism of $\mf g$.
	By Schur's Lemma,
	there exist numbers $\lambda_k \in \mb C$ such that $\Phi_{k,1} = \lambda_k \cdot \Id_{\mf g} \in \End_{\mb C}(\mf g)^{\mf g} \simeq \mb C$.
	Moreover,
	one necessarily has $\lambda_1 \neq 0$,
	else $\Phi$ would \emph{not} be surjective.
	Finally,
	consider the ring automorphism $F \cl \ms O_s \to \ms O_s$ given by $\varpi \mt \sum_{k = 1}^{s-1} \lambda_k \varpi^k$:
	then $\Phi =  \Id_{\mf g} \ots F$,
	because $\Phi$ is determined by $\Phi_{k,0}$ and $\Phi_{k,1}$,
	noting that $\mf g_s$ is generated as a Lie algebra by $\mf g_s^{(0)} \oplus \mf g_s^{(1)} \sse \mf g_s$.
\end{proof}

\begin{rema}
	\label{rmk:simple_assumption}

	The fact that $\mf g$ is simple is only used at the very last step of the proof of Thm.~\ref{thm:tcla_automorphisms}.
	Everything else applies verbatim to a \emph{semisimple} Lie algebra.
\end{rema}

\begin{rema}
	Equivalently,
	the splitting~\eqref{eq:semidirect_auto_tcla} amounts to the direct-product decomposition $\Out(\mf g_s) \simeq \Out(\mf g) \times \ms O_s^{\ts}$,
	for the groups of outer automorphisms,
	noting that the two subgroups $\Aut_{\diag}(\mf g_s), \Aut_{\on{poly}}(\mf g_s) \sse \Aut_{\on{Lie}}(\mf g_s)$ commute.
\end{rema}

\section{Deferred proofs}
\label{sec:proofs}

\subsection{Proof of Lem.~\ref{lem:image_inverse_shapovalov}}
\label{proof:lem_image_inverse_shapovalov}

Recall that $F_c \in \wt M^+_c \wh \ots \wt M^-_c$ is $\mf g_s$-invariant.
(Here one must take completions.)
Then the evaluation map
\begin{equation}
	\ev_{\mc A'} \cl U(\bm{\mf u}^+) \lra \mf g_s^{\dual},
	\qquad X \lmt \ad^{\dual}_X (\mc A'),
\end{equation}
yields a $U(\bm{\mf u}^+)$-linear composition
\begin{equation}
	\wt M^+_c \wh \ots \wt M^-_c \lxra{\simeq} \wt M^+_c \wh \ots N(\bm{\mf u}^+) \lxra{1 \ots \ev_{\mc A'}} M^+_c \ots \mf g_s^{\dual}.
\end{equation}
The left-hand side of~\eqref{eq:image_inverse_shapovalov} is the image of $F_c$ under this arrow:
so it is a $\bm{\mf u}^+$-invariant element,
determined by its component inside $M^+_c[0] \ots \mf g_s^{\dual}$,
cf.~\cite{calaque_felder_rembado_wentworth_2024_wild_orbits_and_generalised_singularity_modules_stratifications_and_quantisation}.
Hence,
it is enough to prove that the right-hand side of~\eqref{eq:image_inverse_shapovalov} is also $\bm{\mf u}^+$-invariant.
Compute then
\begin{equation}
	\label{eq:first_term_invariant_element}
	\Delta(X_j)(v^+_c \ots \mc A') = v^+_c \ots \ad^{\dual}_{X_j} (\mc A'),
\end{equation}
and
\begin{equation}
	\label{eq:second_term_invariant_element}
	\Delta(X_j) \bigl( (Y_i v^+_c) \ots \ad^{\dual}_{X_i} (\mc A') \bigr) = \bigl( [X_j,Y_i] v^+_c \bigr) \ots \ad^{\dual}_{X_i} (\mc A') + (Y_i v^+_c) \ots \ad^{\dual}_{X_j} \ad^{\dual}_{X_i} (\mc A').
\end{equation}
Now denote by $\pi^- \cl \mf g_s \thra \bm{\mf u}^-$ and $\pi_0 \cl \mf g_s \thra \bm{\mf l}$ the canonical projections with respect to the splitting $\mf g_s = \bm{\mf u}^+ \ops \bm{\mf l} \ops \bm{\mf u}^-$.
Then
\begin{equation}
	[X_j,Y_i] v^+_c = \pi_0 \bigl( [X_j,Y_i] \bigr) v^+_c + \pi^-\bigl( [X_j,Y_i] \bigr) v^+_c = c \delta_{ji} v^+_c + \sum_k c^k_{ji} Y_k v^+_c,
\end{equation}
expanding the second summand in the given basis of $\bm{\mf u}^-$,
because
\begin{equation}
	\Braket{ \mc A', \pi_0 \bigl( [X_j,Y_i] \bigr) } = \Braket{ \mc A',[X_j,Y_i] } = \delta_{ji}.
\end{equation}

Hence,
combining~\eqref{eq:first_term_invariant_element}--\eqref{eq:second_term_invariant_element},
it is enough to prove that
\begin{equation}
	\sum_k (Y_k v^+_c) \ots \Bigl( \ad^{\dual}_{X_j} \ad^{\dual}_{X_k} + \sum_i c^k_{ji} \ad^{\dual}_{X_i} \Bigr) \mc A' = 0 \in M^+_c \ots \mf g_s^{\dual};
\end{equation}
but we will actually show the following stronger identity:
\begin{equation}
	\label{eq:coadjoint_action_identity}
	\Bigl( \ad^{\dual}_{X_j} \ad^{\dual}_{X_k} + \sum_i c^k_{ji} \ad^{\dual}_{X_i} \Bigr) \mc A'
	= 0 \in \mf g_s^{\dual}.
\end{equation}
To this end,
the left-hand side of~\eqref{eq:coadjoint_action_identity} vanishes on $\bm{\mf p}^+ \sse \mf g_r$,
since $\mc A'$ is the restriction of the character $\bm\chi^+$.
Regard it as an element of $( \bm{\mf u}^- )^{\dual} \simeq ( \bm{\mf p}^+ )^{\perp} \sse \mf g_r^{\dual}$,
and compute
\begin{equation}
	\label{eq:coadjoin_action_formal_type_1}
	\Braket{ \ad^{\dual}_{X_i} (\mc A'), Y_l }
	= - \Braket{ \mc A',[X_i,Y_l] } = - \delta_{il},
\end{equation}
and
\begin{align}
	\label{eq:coadjoin_action_formal_type_2}
	\Braket{  \ad^{\dual}_{X_j} \ad^{\dual}_{X_k} (\mc A'), Y_l }
	 & = \Braket{ \mc A', \bigl[X_k,[X_j,Y_l] \bigr] }
	=\sum_m c^m_{jl} \Braket{ \mc A', [X_k,Y_m] }      \\
	 & = \sum_m c^m_{jl} \delta_{km}
	= c^k_{jl},
\end{align}
since $\ad_{\bm{\mf l}}^{\dual} (\mc A') = (0)$.
Now~\eqref{eq:coadjoint_action_identity} follows by taking linear combinations of~\eqref{eq:coadjoin_action_formal_type_1}--\eqref{eq:coadjoin_action_formal_type_2}.

\subsection{Proof of Lem.~\ref{lem:coordinate_invariance}}
\label{proof:lem_coordinate_invariance}

Consider an irregular type $Q'$,
with semisimple commuting coefficients $A'_1,\dc,A'_{s-1} \in \mf g$.
Compute
\begin{equation}
	\label{eq:pullback_derivative_irregular_part}
	F^*(\dif Q')
	= \dif \, \bigl( Q' \bigl( F(\varpi) \bigr) \bigr)
	= \sum_{i = 1}^{s-1} A'_i \varpi^{-i-1} \cdot \wt f_i \dif \varpi,
	\qquad \wt f_i \ceqq f^{-i} (f + \varpi f'),
\end{equation}
using~\eqref{eq:pullback_action},
and commuting $F^*$ past the structural derivation $\dif \cl \mf g (\!( \varpi )\!) \to \mf g (\!( \varpi )\!) \dif \varpi$.
Now $\wt f_i(0) \neq 0$ for $i \in \set{1,\dc,s-1}$,
and~\eqref{eq:pullback_derivative_irregular_part} has no residue.
Thus,
there exists a new tuple $\bigl( \wt A'_1,\dc,\wt A'_{s-1} \bigr) \in \mf g^s$ such that:
(i) $\wt A'_i$ is a linear combination of $A'_i,\dc,A'_{s-1}$;
(ii) the coefficient of $A'_i$ in that linear combination is nonzero;
and (iii) one has
\begin{equation}
	F^*(\dif Q') - \dif \wt Q' \in \mf g \llb \varpi \rrb \dif \varpi,
	\qquad
	\wt Q' \ceqq \sum_{i = 1}^{s-1} \wt A'_i \frac{\varpi^{-i}}{-i}.
\end{equation}
Thus,
the elements $\wt A'_1,\dc,\wt A'_{s-1}$ lie in a Cartan subalgebra $\mf g$.

Now choose a semisimple residue $\Lambda' \in \mf g$ which commutes with $Q'$,
and let $\mc A' \ceqq \Lambda' \varpi^{-1}\dif \varpi + \dif Q' \in \mf g'_s$,
so that
\begin{equation}
	\wt{\mc A}'
	\ceqq F^*(\mc A')
	= \Lambda' \varpi^{-1}\dif \varpi + \dif \wt Q'.
\end{equation}
The above implies that $\Lambda'$ commutes with $\wt Q'$,
whence $\wt{\mc A}' \in \mf g'_s$;
moreover,
it follows that the nested $\ad_\mf g$-stabilizers are invariant,
and so:
(i) the second statement is proven;
and (ii) $\mf l_{Q'} = \mf l_{\wt Q'}$,
preserving nonresonance.

\subsection{Proof of Lem.~\ref{lem:flatness_high_order}}
\label{proof:lem_flatness_high_order}

Choose a vector $X \in \mf g \simeq \mf g^{\dual}$.
Using~\eqref{eq:orbit_from_extended_orbit},
the fibre over $X$ can be described as follows:
\begin{align}
	(\mu')^{-1}(X)
	 & = \Set{ \bigl( (g,\Lambda),\dif Q \bigr) \in \on T^*\! G \ts \check{\mc O}' | \Lambda = X, \quad \check\mu'(\dif Q) - \pi' \bigl( \Ad^{\dual}_g(\Lambda) \bigr) = - \Lambda' } \bs \, L_s \\
	 & \simeq \Set{ (g,\dif Q) \in G \ts \check{\mc O}' | \check\mu'(\dif Q) = \pi' \bigl( \Ad^{\dual}_g(X) \bigr) - \Lambda'} \bs \, L_s                                                        \\
	 & = \bigl( G \ts (\check\mu')^{-1} (X') \bigr) \bs \, L_s,
	\qquad X' \ceqq \pi' \bigl(\Ad^{\dual}_g(X) \bigr) - \Lambda'.
\end{align}
Hence,
by hypothesis,
$(\mu')^{-1}(X) \neq \vn$.
Moreover,
the $L_s$-action on $G \ts \check{\mc O}'$ is free,
and so the dimension of the fibre does \emph{not} depend on $X$.

\subsection{Proof of Lem.~\ref{lem:unique_factorization}}
\label{proof:lem_unique_factorization}

For the first statement,
consider the truncated Taylor expansion
\begin{equation}
	\bm g = \exp \Biggl( \sum_{i = 1}^{s-1} X_i \varpi^i \Biggr) \cdot g,
	\qquad X_1,\dc,X_{s-1} \in \mf g,
\end{equation}
and analogously take the Ansätze
\begin{equation}
	\bm u^\pm = \exp \Biggl( \sum_{i = 1}^{s-1} Y^\pm_i \varpi^i \Biggr) \cdot u^\pm,
	\qquad Y^\pm_1,\dc,Y^\pm_{s-1} \in \wt{\mf n}^\pm_1,
	\quad u^\pm \in N^\pm_1,
\end{equation}
and
\begin{equation}
	\bm h = \exp \Biggl( \sum_{i = 1}^{s-1} X'_i \varpi^i \Biggr) \cdot h,
	\qquad X'_1,\dc,X'_{s-1} \in \mf l_1,
	\quad h \in L_1.
\end{equation}
Then we assume that $g = h u^- u^+$,
and must solve the equations
\begin{equation}
	\label{eq:recursive_factorization}
	\quad X_i + R_i
	= \Ad_{hu^-}(Y^+_i) + \Ad_h(Y^-_i) + X'_i,
	\qquad i \in \set{1,\dc,s-1},
\end{equation}
with initial condition $R_1 = 0$,
where the element $R_i \in \mf g$ only depends on
\begin{equation}
	\Set{ g,u^\pm,X_j,Y^\pm_j,X'_j | j \in \set{1,\dc,i-1}},
	\qquad i \in \set{2,\dc,s-1}.
\end{equation}
But there is a splitting $\mf g = \Ad_{u^-}(\wt{\mf n}^+_1) \ops \wt{\mf n}^-_1 \ops \mf l_1$,
and the Lie subalgebras $\wt{\mf n}^\pm_1$ are $\Ad_{L_1}$-invariant:
thus,
the equations~\eqref{eq:recursive_factorization} \emph{uniquely} determine $Y^\pm_i$ and $X'_i$,
for $i \in \set{1,\dc,s-1}$.

The second statement follows from this explicit construction of the factorization.

\subsection{Proof of Lem.~\ref{lem:algebraic_iso}}
\label{proof:lem_algebraic_iso}

Choose $\bm u^- \in \wt N^-_1$ and $\mc A_{s-1} \in \mc O'_{s-1}$,
and let $\ul{\mc A}_{s-1} \ceqq \mc A_{s-1} + \mc A'_{\on{top}} \in \mf g_s^{\dual}$.
Moreover,
write $\bm u^- = e^{\bm X}$ with $\bm X \in \wt{\mf n}^-_1$.
Then,
in the notation of~\eqref{eq:nilpotent_element}:
\begin{equation}
	\label{eq:adjoint_expansion}
	\bm Y'
	= \Ad^{\dual}_{\bm u^-} \bigl( \ul{\mc A}_{s-1} \bigr) - \ul{\mc A}_{s-1}
	= \sum_{k > 0} \frac 1{k!} (\ad^{\dual}_{\bm X})^k \bigl( \ul{\mc A}_{s-1} \bigr),
\end{equation}
using the $\ad^{\dual}_{\mf g_s}$-action.
The first statement follows,
since $\bigl[ \, \mf l_1,\mf n^\pm_1 \bigr] \sse \mf n^\pm_1$,
identifying as usual the dual of $\wt{\mf n}^+_1$ with a space of $\mf n^-_1$-valued 1-forms on $\Spec \mb C\llb \varpi \rrb$.

For the second statement we construct an inverse of~\eqref{eq:algebraic_iso}.
Choose thus $\bm Y' \in (\wt{\mf n}^+_1)^{\dual}$ and $\mc A_{s-1} \in \mc O'_{s-1}$:
the point is to construct an element $\bm u^- = e^{\bm X} \in \wt N_1$ such that~\eqref{eq:adjoint_expansion} holds.
To this end,
consider a 1-parameter subgroup $\lambda \cl \mb C^{\ts} \to L_1$ such that
\begin{equation}
	L_1
	= \Set{ g \in G | \lambda(t) \cdot g = g \cdot \lambda(t) \text{ for } t \in \mb C^{\ts}},
\end{equation}
and
\begin{equation}
	N^+_1
	= \Set{ g \in G | \lim_{t \to 0} \bigl( \lambda(t) \cdot g \cdot \lambda(t)^{-1} \bigr) = 1_G },
\end{equation}
cf.~\cite[Prop.~2.3]{richardson_1988_conjugacy_classes_of_n_tuples_in_lie_algebras_and_algebraic_groups}.
Letting $\theta \ceqq - \eval[1]{\od{\lambda}{t}}_{t = 1} \in \mf l_1$,
we obtain a Lie-algebra $\mb Z$-grading
\begin{equation}
	\mf g
	= \bops_{\mb Z} \mf g_k,
	\qquad \mf g_k
	\ceqq \Set{ X \in \mf g | \ad_{\theta}(X) = k X },
\end{equation}
and it follows that
\begin{equation}
	\mf l_1
	= \mf g_0,
	\qquad \mf n^\pm_1
	= \mf g_{\gtrless 0}
	\ceqq \bops_{\pm \mb Z_{> 0}} \mf g_k.
\end{equation}

Now take an integer $N > 0$ which is greater than all the eigenvalues of $\ad_{\theta}$,
and define also a $\mb Z$-grading on the (algebraic) loop algebra $\mc L \mf g \ceqq \mf g \bigl( \mb C [\varpi^{\pm 1}] \bigr)$ via
\begin{equation}
	\deg(\mf g_k \ots \mb C \varpi^i)
	\ceqq k + i N,
	\qquad i,k \in \mb Z.
\end{equation}
Denote by $(\mc L \mf g)_j \sse \mc L \mf g$ the graded pieces,
and again by $(\mc L \mf g)_{\gtrless} 0$ the positive/negative part of the grading.
Then $\mf n_1^-[\varpi] \sse (\mc L\mf g)_{> 0}$:
write therefore
\begin{equation}
	\label{eq:graded_components}
	\ul{\mc A}_{s-1}
	= \sum_{j \geq 0} \ul A_j \varpi^{-s} \dif \varpi,
	\quad \bm Y'
	= \sum_{j > 0} Y'_j \varpi^{-s} \dif \varpi,
	\qquad \ul A_j,Y'_j \in (\mc L\mf g)_j,
\end{equation}
by regarding as usual $\mf g_s^{\dual}$ as a quotient of the subspace $\varpi^{-s} \mf g [\varpi] \dif \varpi \sse \mc L\mf g(\varpi^{-s} \dif \varpi)$.
Now~\eqref{eq:adjoint_expansion}--\eqref{eq:graded_components} yield the following equalities \emph{modulo} $\varpi^s \mf g [\varpi]$:
\begin{equation}
	\label{eq:adjoint_expansion_2}
	Y'_j = \sum_{k > 0} \frac 1{k!} \sum_{\substack{i_1,\dc,i_k > 0, \\ i_1 + \dm + i_k \leq j}} \ad_{X_{i_1}} \dm \ad_{X_{i_k}} \bigl( \ul A_{j-i_1-\dm-i_k} \bigr),
	\qquad j > 0,
\end{equation}
writing also $\bm X = \sum_{i > 0} X_i$,
for suitable $X_i \in (\mc L\mf g)_i$.
To conclude,
observe that the right-hand side of~\eqref{eq:adjoint_expansion_2} has the form $[X_j,\ul A_0] + R_j$,
with initial condition $R_1 = 0$,
where $R_j \in (\mc L\mf g)_j$ only depends on $\set{ \ul A_i,X_i | i < j }$,
for all $j > 0$.
Moreover,
since by construction $\ul A_0 = A'_{s-1}$ is the leading coefficient of $\mc A'$,
it follows that $\ad_{\ul A_0} \cl \mf g \to \mf g$ restricts to a linear automorphism of $\mf g_k$ for any integer $k$,
whence a linear automorphism of $(\mc L\mf g)_j$ for any integer $j$.
Thus,
the equations~\eqref{eq:adjoint_expansion_2} \emph{uniquely} determine the elements $X_i$,
modulo $\varpi^s \mf g[\varpi]$.
Furthermore,
one can inductively prove that
\begin{equation}
	X_i \in \mf n^-_1 [\varpi] + \varpi^s \mf g[\varpi],
	\qquad i > 0,
\end{equation}
so that indeed $\bm u^- \in \wt N^-_1$.
Finally,
the $L_1$-equivariance follows from the above explicit construction of the map~\eqref{eq:algebraic_iso}---%
and its inverse.

\subsection{Proof of Lem.~\ref{lem:closed_subgroup}}
\label{proof:lem_closed_subgroup}

If $s = 1$ the statement is trivial;
choose thus $s > 1$.

By hypothesis $A_{s-1} = A'_{s-1} \in \mf t$,
because the Birkhoff action cannot modify the leading coefficient.
Moreover,
by~\cite[Prop.~9.3.2]{babbitt_varadarajan_1983_formal_reduction_theory_of_meromorphic_differential_equations_a_group_theoretic_view},
there exists $\bm h \in H_s \cap \on{Bir}_s \sse G_s$ such that the coefficients of $\ul{\mc A} \ceqq \Ad_{\bm h}^{\dual} \bigl( \mc A \bigr)$ commute with $A_{s-1} = A'_{s-1}$.
Up to replacing $\mc A$ by $\ul{\mc A}$,
we may then assume that $A_0,\dc,A_{s-1}$ commute with $A'_{s-1}$.

Then start again under this assumption,
and choose an element
\begin{equation}
	\label{eq:birkhoff_element}
	\bm h = e^{\bm X} \in \on{Bir}_s,
	\qquad \bm X = \sum_{j = 1}^{s-1} X_j \varpi^{-j},
	\quad X_j \in \mf g,
\end{equation}
such that $\Ad_{\bm h}^{\dual} \bigl( \mc A \bigr) = \mc A' \in \mf g_s^{\dual}$.
Comparing coefficients yields the equalities
\begin{equation}
	\label{eq:comparing_coefficients}
	A_{s-i}
	= \sum_{j = 1}^i \sum_{n > 0} \frac 1{n!} \cdot \sum_{\substack{i_1,\dc,i_n > 0 \\ i_1 + \dm + i_n = i-j}} \ad_{X_{i_1}} \dm \ad_{X_{i_n}} (\mc A'_{s-j}),
	\qquad i \in \set{1,\dc,s}.
\end{equation}
Introduce as usual the stabilizers $L_1 \sse G$ and $\mf l_1 = \Lie \bigl( L_1 \bigr)$ of $A'_{s-1}$.
Then we inductively show that $X_k \in \mf l_1$,
in the notation of~\eqref{eq:birkhoff_element},
for $k \in \set{1,\dc,s-1}$.
For the base,
if $i \ceqq 2$ then~\eqref{eq:comparing_coefficients} reads
\begin{equation}
	A_{s-2}
	= A'_{s-2} + \bigl[ X_1,A'_{s-1} \bigr] \in \mf g.
\end{equation}
But $\mf g = \mf l_1 \ops \ad_{A'_{s-1}}(\mf g)$ (as $A'_{s-1}$ is semisimple),
and  $A_{s-2} \in \mf l_1$ (by assumption),
so that $\bigl[ X_1,A'_{s-1} \bigr] = 0$.
For the inductive step,
suppose that $X_1,\dc,X_{k-1} \in \mf l_1$,
Taking $i \ceqq k+1$ in~\eqref{eq:comparing_coefficients},
the right-hand side lies in the affine subspace $\bigl[ X_k,A'_{s-1} \bigr] + \mf l_1 \sse \mf g$.
Since $A_{s-k-1} \in \mf l_1$,
we find that $\bigl[ X_k,A'_{s-1} \bigr] = 0$.

Now consider the principal part $\mc A_{s-1} \ceqq \mc A - A_{s-1} \varpi^{-s} \dif \varpi$,
of pole order $s - 1$,
in the notation of~\eqref{eq:subleading_term}.
The above implies that
\begin{equation}
	\mc A_{s-1}
	= \Ad_{\bm h}^{\dual} \bigl( \mc A' \bigr) - A'_{s-1} \varpi^{-s} \dif \varpi
	= \Ad_{\bm h}^{\dual} \bigl( \mc A'_{s-1} \bigr),
	\qquad \mc A'_{s-1}
	\ceqq \mc A' - A'_{s-1} \varpi^{-s} \dif \varpi.
\end{equation}
Thus:
(i) $\mc A_{s-1}$ lies in the $\Ad_{\on{Bir}_{s-1}}^{\dual}$-orbit of $\mc A'_{s-1}$;
and (ii) $\mc A_{s-1}$ takes coefficients in $\mf h \cap \mf l_1 \sse \mf l_1$.
One now concludes recursively,
on the pole order.

\backmatter

\bibliographystyle{amsplain}
\bibliography{/home/gabriele/Desktop/bibliography_macros/bibliography}

\end{document}